\documentclass[11pt,a4paper]{amsart}

\usepackage[utf8]{inputenc}
\usepackage{relsize}
\usepackage{amssymb}
\usepackage{amsfonts}
\usepackage{amsthm}
\usepackage{amsmath}
\usepackage{mathtools}
\usepackage{amscd}
\usepackage[utf8]{inputenc}
\usepackage{t1enc}
\usepackage[mathscr]{eucal}
\usepackage{indentfirst}
\usepackage{graphicx}
\usepackage{graphics}
\usepackage{wasysym}
\usepackage{pict2e}
\usepackage{lmodern}    
\usepackage{ifthen}  
\usepackage{epic}
\numberwithin{equation}{section}
\usepackage[margin=3.6cm]{geometry}
\usepackage{epstopdf} 
\usepackage{tikz-cd}
\tikzcdset{scale cd/.style={every label/.append style={scale=#1},
    cells={nodes={scale=#1}}}}
\usetikzlibrary{matrix,arrows,decorations.pathmorphing}
\usepackage{yhmath}
\author{Bence Hevesi}
\email{bh624@cam.ac.uk}
\address{Department of Pure Mathematics and Mathematical Statistics\\
Centre for Mathematical Sciences\\
Wilberforce Road\\
Cambridge\\
CB3 0WB}

\usepackage{mathrsfs}
\usepackage{lipsum}

\usepackage[
backend=biber,
style=alphabetic,
maxnames=8
]{biblatex}

\usepackage[hidelinks]{hyperref}
\hypersetup{
    colorlinks=false, %set true if you want colored links
    linktoc= all,     %set to all if you want both sections and subsections linked
}
 \usepackage{url}
 \usepackage{textcomp}

\theoremstyle{plain}
\newtheorem{Th}{Theorem}[section]
\newtheorem{Lemma}[Th]{Lemma}
\newtheorem{Cor}[Th]{Corollary}
\newtheorem{Prop}[Th]{Proposition}

\theoremstyle{definition}

\newtheorem{Def}[Th]{Definition}
\newtheorem{Conj}[Th]{Conjecture}
\newtheorem{Rem}[Th]{Remark}
\newtheorem{?}[Th]{Problem}
\newtheorem{Ex}[Th]{Example}
\newtheorem{assumption}[Th]{Assumption}

\newcommand{\im}{\operatorname{im}}
\newcommand{\Hom}{{\rm{Hom}}}

\begin{document}
\title{Local-global compatibility at $p\neq \ell$ for torsion automorphic forms}
\makeatletter
\@namedef{subjclassname@2020}{%
 \textup{2010} Mathematics Subject Classification}
\makeatother

\begin{abstract}
We prove local-global compatibility results at $p \neq \ell$ for the automorphic group determinants constructed by Scholze \cite{Sch15}, generalising the result of Varma \cite{Var14} to torsion classes appearing in Betti cohomology. Our argument combines the construction of Scholze with the theory of representations of $p$-adic general linear groups with $\mathbf{Z}_{\ell}$-coefficients.
\end{abstract}
\maketitle

\tableofcontents

\section{Introduction}
Let $F$ be a number field and let $n\geq 1$ be an integer. In \cite{Lan79}, Langlands introduced a conjectural framework predicting a parametrisation of automorphic representations of $\textnormal{GL}_n(\mathbf{A}_F)$ by $n$-dimensional complex representations of the hypothetical Langlands group $L_F$. Guided by Langlands’ work, Clozel \cite{Cl90} singled out the class of \textit{algebraic} automorphic representations, which, under Langlands reciprocity, should correspond to compatible systems of Galois representations of $F$. In particular, one has the following conjecture.
\begin{Conj}\label{Conj_Reciprocity_char0}
Let $\pi$ be an algebraic cuspidal automorphic representation of $\textnormal{GL}_n(\mathbf{A}_F)$ and $\iota:\overline{\mathbf{Q}}_{\ell}\xrightarrow{\sim}\mathbf{C}$ be a field isomorphism. There exists a continuous semisimple Galois representation
\begin{equation*}
    r_{\iota}(\pi):\textnormal{Gal}(\overline{F}/F)\to \textnormal{GL}_n(\overline{\mathbf{Q}}_{\ell})
\end{equation*}
such that, for every finite place $v\nmid \ell$ of $F$,
    \begin{equation}\label{eqn_LGC_v}
        \textnormal{WD}(r_{\iota}(\pi)|_{G_{F_v}})^{F-ss}\cong \textnormal{rec}_{F_v}^T(\iota^{-1}\pi_v).
    \end{equation}
\end{Conj}
 In the above conjecture, $\textnormal{WD}(\rho)^{F-ss}$ denotes the Frobenius-semisimplified Weil--Deligne representation associated with $\rho:G_{F_v}\to\textnormal{GL}_n(\overline{\mathbf{Q}}_{\ell})$ via Grothendieck's $\ell$-adic monodromy theorem and\footnote{This is often called the Tate-normalised local Langlands correspondence and has the virtue that it is invariant under any automorphism of $\mathbf{C}$, making it possible to define it over any field abstractly isomorphic to $\mathbf{C}$, like $\overline{\mathbf{Q}}_{\ell}$.} $\textnormal{rec}_{F_v}^T(-)=\textnormal{rec}_{F_v}(-\otimes|\det|_v^{\frac{1-n}{2}})$ where $\textnormal{rec}_{F_v}(-)$ denotes the local Langlands correspondence constructed in \cite{HT01}. 
 
 In particular, \ref{eqn_LGC_v} asserts that, for $v\nmid \ell$, $r_{\iota}(\pi)|_{G_{F_v}}$ is unramified whenever $\pi_v$ is and the Satake-parameter of $\pi_v|\det|_v^{\frac{1-n}{2}}$ is conjugate to $\iota(r_{\iota}(\pi)(\textnormal{Frob}_v))\in \textnormal{GL}_n(\mathbf{C})$ for any geometric Frobenius lift $\textnormal{Frob}_v$.

 \begin{Rem}
     In fact, Clozel conjectures much more than this. He predicts the existence of an associated motive $M_{\pi}$ over $F$ satisfying local-global compatibility at \textit{all} places of $F$. The Galois representation $r_{\iota}(\pi)$ should then be the $\ell$-adic \'etale realisation of $M_{\pi}$. In particular, it is expected that $r_{\iota}(\pi)$ is de Rham, its $\ell$-adic Hodge--Tate structure matches with the infinitesimal character of $\pi_{\infty}$ and \ref{eqn_LGC_v} holds at $\ell$ as well. This part of his conjecture is often referred to as "local-global compatibility at $p=\ell$" and is beyond the scope of this article. 
 \end{Rem}

Due to the spectacular work\footnote{The last missing cases in the conjugate self-dual case were established in \cite{Car12} and we point the reader to the introduction of \textit{loc. cit.} for further references.} of many authors, Conjecture \ref{Conj_Reciprocity_char0} is known under the following assumptions.
\begin{itemize}
    \item $F$ is a \textit{CM number field},
    \item $\pi$ is \textit{regular algebraic} (its infinitesimal character is that of an algebraic representation of $(\textnormal{Res}_{F/\mathbf{Q}}\textnormal{GL}_n)_{\mathbf{R}}$),
    \item $\pi$ is \textit{conjugate self-dual} ($\pi^{\vee}= \pi\circ c$).
\end{itemize} 
The main point is that, under these assumptions, the Hecke eigensystem of $\pi^{\infty}$ appears in the \'etale cohomology of unitary Shimura varieties. Beyond the conjugate self-dual case, the Hecke eigensystem associated with $\pi$ only appears in the Betti cohomology of some smooth manifold, a $G=\textnormal{Res}_{F/\mathbf{Q}}\textnormal{GL}_n$-locally symmetric space $X_K^G$. Despite the lack of direct relation to the cohomology of Shimura varieties, Harris--Lan--Taylor--Thorne \cite{HLTT16} and Scholze \cite{Sch15} construct $r_{\iota}(\pi)$ satisfying local-global compatibility at almost all primes by realising $r_{\iota}(\pi)\oplus r_{\iota}(\pi)^{\vee,c}(1-2n)$ as an $\ell$-adic limit of automorphic Galois representations for the quasi-split unitary group $U(n,n)$. Moreover, Varma \cite{Var14} shows that the construction of \cite{HLTT16} is sufficient to prove a semisimplified version of \ref{eqn_LGC_v} at every $v\nmid \ell$. Therefore, one has the following theorem.
\begin{Th}[Harris--Lan--Taylor--Thorne, Scholze, Varma]\label{Thm_HLTT_Sch_Var}
    Let $F$ be a CM number field, $\pi$ be a regular algebraic cuspidal automorphic representation of $\textnormal{GL}_n(\mathbf{A}_F)$ and $\iota:\overline{\mathbf{Q}}_{\ell}\xrightarrow{\sim}\mathbf{C}$ be a field isomorphism. There exists a continuous semisimple Galois representation
\begin{equation*}
    r_{\iota}(\pi):\textnormal{Gal}(\overline{F}/F)\to \textnormal{GL}_n(\overline{\mathbf{Q}}_{\ell})
\end{equation*}
such that, for every finite place $v\nmid \ell$ of $F$,
    \begin{equation}\label{eqn_ssLGC_v}
        (r_{\iota}(\pi)|_{W_{F_v}})^{ss}\cong \textnormal{rec}_{F_v}^T(\iota^{-1}\pi_v)^{\textnormal{ss}}.
    \end{equation}
\end{Th}
In parallel with developments in Langlands reciprocity, it has become clear—through, for instance, the work of Calegari–Geraghty \cite{CG18}—that it is fruitful to study an extension of Conjecture \ref{Conj_Reciprocity_char0} to Hecke eigensystems appearing in the Betti cohomology of $X_K^G$. To discuss such a generalisation, we need to introduce some notation.

For a finite place $v\nmid \ell$ of $F$, let $\mathfrak{Z}_{\textnormal{GL}_n(F_v)}$ be the Bernstein centre of the category of smooth $\mathbf{Z}_{\ell}$-representations of $\textnormal{GL}_n(F_v)$. For a finite set of finite places $S$ of $F$ containing the $\ell$-adic places $S_{\ell}(F)$, set
\begin{equation*}
\mathbf{T}^S:=\otimes_{v\notin S}'\mathfrak{Z}_{\textnormal{GL}_n(F_v)}.
\end{equation*}
The commutative algebra $\mathbf{T}^S$ acts on the Betti cohomology groups $H^{\ast}(X_K^G,\mathbf{Z}/\ell^m\mathbf{Z})$ via Hecke correspondences and we can introduce the faithful quotients
\begin{equation*}
    \mathbf{T}^S(K,m):=\im\Big(\mathbf{T}^S\to \textnormal{End}_{\mathbf{Z}_{\ell}}\big(H^{\ast}(X_K^G,\mathbf{Z}/\ell^m\mathbf{Z})\big)\Big)
\end{equation*}
that, for $v\notin S$,
admit natural maps
\begin{equation*}
    \textnormal{nat}_v:\mathfrak{Z}_{\textnormal{GL}_n(F_v)}\to \mathbf{T}^S(K,m).
\end{equation*}

On the other hand, for every $v\nmid \ell$, one can introduce an affine scheme 
\begin{equation*}
    X_{F_v,n}=\textnormal{Spec}(\mathfrak{R}^{\textnormal{ps}}_{F_v,n}),
\end{equation*} an increasing union of finite type affine schemes over $\mathbf{Z}_{\ell}$, satisfying the following properties.
\begin{itemize}
    \item It represents the functor sending an Artinian $\mathbf{Z}_{\ell}$-algebra $A$ to the set of $n$-dimensional continuous group determinants $W_{F_v}\to A$ in the sense of \cite{Che09}.
    \item Sending an $L$-parameter to its associated determinant induces a $\mathbf{Z}_{\ell}$-algebra homomorphism
    \begin{equation*}
\xi:\mathfrak{R}^{\textnormal{ps}}_{F_v,n}\to\mathcal{O}\big(Z^1_n(W_{F_v})\big)^{\textnormal{GL}_n}
    \end{equation*}
    to the ring of global sections of the stack $[Z^1_n(W_{F_v})/\textnormal{GL}_n]$ of $n$-dimensional $L$-parameters. After inverting $\ell$, the map $\xi[\frac{1}{\ell}]$ is an isomorphism.
\end{itemize}
In particular, postcomposition of $\xi$ by the isomorphism
\begin{equation*}
    \Psi_{F_v,n}:\mathcal{O}\big(Z^1_n(W_{F_v})\big)^{\textnormal{GL}_n} \xrightarrow{\sim}\mathfrak{Z}_{\textnormal{GL}_n(F_v)}
\end{equation*} 
of \cite{HM18} interpolating the (Tate-normalised) local Langlands correspondence yields a map
\begin{equation*}
    \Phi_{F_v,n}:\mathfrak{R}^{\textnormal{ps}}_{F_v,n}\to \mathfrak{Z}_{\textnormal{GL}_n(F_v)}.
\end{equation*}
\begin{Conj}\label{Conj_LGC_torsion}
     Let $S_{\ell}(F)\subset S_{\textnormal{avoid}}\subset S_{\textnormal{bad}}$ be finite sets of finite places of $F$. Let $K\leq \textnormal{GL}_n(\mathbf{A}_F^{\infty})$ be a good\footnote{For the definition, see \S\ref{sec_locally_symm_spaces}.} compact open subgroup such that $K_v = \textnormal{GL}_n(\mathcal{O}_{F_v})$ for $v\notin S_{\textnormal{bad}}$. Let $\mathfrak{m}\leq\mathbf{T}^{S_{\textnormal{bad}}}(K,m) $ be a maximal ideal.

     There exists an integer $N\geq 1$ depending only on $n$ and $[F:\mathbf{Q}]$, an ideal $I\leq \mathbf{T}^{S_{\textnormal{avoid}}}(K,m)_{\mathfrak{m}}$ with $I^N=0$ and a continuous group determinant
     \begin{equation*}
        D_{\mathfrak{m}}:G_{F,S_{\textnormal{bad}}}\to \mathbf{T}^{S_{\textnormal{avoid}}}(K,m)_{\mathfrak{m}}/I
    \end{equation*}
    such that, for every finite place $v\notin S_{\textnormal{avoid}}$, we have
    \begin{equation}\label{eqn_LGC_torsion}
        D_{\mathfrak{m}}|_{G_{F_v}}=\textnormal{nat}_v\circ \Phi_{F_v,n}
    \end{equation}
    as maps $\mathfrak{R}^{\textnormal{ps}}_{F_v,n}\to \mathbf{T}^{S_{\textnormal{avoid}}}(K,m)_{\mathfrak{m}}/I$.
\end{Conj}
\begin{Rem}
    Given an algebraic representation $W$ of $(\textnormal{Res}_{F/\mathbf{Q}}\textnormal{GL}_n)_{\mathbf{R}}$ and a field isomorphism $\iota:\overline{\mathbf{Q}}_{\ell}\xrightarrow{\sim}\mathbf{C}$, we can build a $\mathbf{Z}_{\ell}$-local system $\mathcal{W}$ on $X_K$. By an argument with the Hochschild--Serre spectral sequence, Conjecture \ref{Conj_LGC_torsion} implies the analogous result for $H^{\ast}(X_K,\mathcal{W})$ as well. Therefore, by the description of rational Betti cohomology of $X_K^G$ in terms of automorphic representations of $G$ (cf. \cite{Fr98}, \cite{FS98}), it implies Theorem \ref{Thm_HLTT_Sch_Var}.
\end{Rem}

The state of Conjecture \ref{Conj_LGC_torsion} is as follows.

\begin{itemize}
    \item The work of Scholze \cite{Sch15} verifies the Conjecture under the following assumptions.
    \begin{enumerate}
        \item $F$ is a CM number field that contains an imaginary quadratic subfield.
        \item The set $S_{\textnormal{bad}}$ is stable under complex conjugation and is \textit{unconditional} in the sense of Definition \ref{DefinitionOnPlaces} (ii).
        \item We have $S_{\textnormal{avoid}}=S_{\textnormal{bad}}$.
    \end{enumerate}
    In  particular, for every $S_{\ell}(F)\subset S_{\textnormal{avoid}}'\subset S_{\textnormal{bad}}'\subset  S_{\textnormal{bad}}$, $K$ and $\mathfrak{m}$ we have determinants
    \begin{equation*}
        D_{\mathfrak{m}}:G_{F,S_{\textnormal{bad}}}\to \mathbf{T}^{S_{\textnormal{avoid}}'}(K,m)_{\mathfrak{m}}/I
    \end{equation*}
    satisfying \ref{eqn_LGC_torsion} at $v\notin S_{\textnormal{bad}}$. In other words, to verify Conjecture \ref{Conj_LGC_torsion} for $(S_{\textnormal{avoid}}',S_{\textnormal{bad}}')$, it is left to prove that, after possibly enlarging $I$, the following hold.
    \begin{itemize}
        \item The continuous determinant $D_{\mathfrak{m}}$ factors through $G_{F,S_{\textnormal{bad}}'}$.
        \item For $v\notin S_{\textnormal{bad}}\setminus S_{\textnormal{avoid}}'$, \ref{eqn_LGC_torsion} holds.
    \end{itemize}
    \item For $S_{\textnormal{avoid}}\neq S_{\textnormal{bad}}$, the only available results are for "non-Eisenstein" maximal ideals and a restricted class of level subgroups. For example, \cite{ACC23}, Theorem 3.1.1 shows Conjecture \ref{Conj_LGC_torsion} under the following assumptions.
    \begin{enumerate}
        \item $F$ is a CM number field that contains an imaginary quadratic subfield.
        \item The set $S_{\textnormal{avoid}}$ is stable under complex conjugation and \textit{unconditional} in the sense of Definition \ref{DefinitionOnPlaces} (ii).
        \item For every $v\in S_{\textnormal{ram}}:=S_{\textnormal{bad}}\setminus S_{\textnormal{avoid}}$, the residue characteristic of $v$ splits in an imaginary quadratic subfield of $F$.
        \item The Galois representation $\overline{\rho}_{\mathfrak{m}}$ constructed in \cite{Sch15} is absolutely irreducible.
        \item For\footnote{Here $\textnormal{Iw}_v(b,c)\leq \textnormal{GL}_n(\mathcal{O}_{F_v})$ is the Iwahori subgroup of matrices that reduce to upper triangular ones modulo $\varpi_v^c$ and strictly upper triangular ones modulo $\varpi_v^b$.} $v\in S_{\textnormal{ram}}$, $\textnormal{Iw}_{v}(1,1)\leq K_v\leq \textnormal{Iw}_v(0,1)$. 
    \end{enumerate}
    
    For a generalisation of the above result of \cite{ACC23} to parahoric level subgroups $\mathcal{P}_v(1,1)\leq K_v\leq \mathcal{P}_v(0,1)$, see \cite{MT22}, Theorem 7.7.
\end{itemize}

\subsection{Results of this paper}
The following\footnote{This follows from Theorem \ref{MainTHM}. However, note that it is formulated slightly differently, that is, to apply Theorem \ref{MainTHM} one takes $S_{\textnormal{bad}}^{4.1}=S_{\textnormal{bad}}^{1.6}\cup (S_{\textnormal{bad}}^{1.6})^c$.} is our main theorem, generalising \cite{Sch15}, Theorem V.3.1, \cite{ACC23}, Theorem 3.1.1, and \cite{MT22}, Theorem 7.7.

\begin{Th}\label{Thm_MainResultIntro}
Conjecture \ref{Conj_LGC_torsion} holds for any CM number field $F$ and finite sets of finite places $S_{\ell}(F)\subset S_{\textnormal{avoid}}\subset S_{\textnormal{bad}}$ of $F$ satisfying the following.
    \begin{enumerate}
        \item  $F$ contains an imaginary quadratic subfield.
        \item The set $S_{\textnormal{avoid}}$ is stable under complex conjugation and is \textit{unconditional} in the sense of Definition \ref{DefinitionOnPlaces} (ii).
        \item For $v\in S_{\textnormal{ram}}:=S_{\textnormal{bad}}\setminus S_{\textnormal{avoid}}$, the residue characteristic of $v$ splits in an imaginary quadratic subfield of $F$.
    \end{enumerate}
\end{Th}

\begin{Rem}
In \cite{Sch15}, the assumptions (i), (ii) are present to have access to the unconditional base change result of \cite{Shi14}. We assume these conditions for the same reason.

In our proof, we also need condition (iii) to ensure that the transfers from $U(n,n)/F^+$ to $\textnormal{GL}_{2n}/F$ are compatible at $v\in S_{\textnormal{ram}}$ even when the automorphic representation is ramified at $v$.
In addition, (iii) plays a crucial role in the various twisting arguments we employ, and is therefore integral to our proof.
\end{Rem}

Before discussing the proof, we record a simple corollary.\footnote{We thank Andrea Dotto for drawing our attention to this application.}
\begin{Cor}
    Let $F$ and $S_{\textnormal{ram}}\subset S_{\textnormal{bad}}$ be as in Theorem \ref{Thm_MainResultIntro}. Consider $v\in S_{\textnormal{ram}}$, a good compact open subgroup $K^v\leq \textnormal{GL}_n(\mathbf{A}_{F}^{\infty,v})$ with $K^v_w=\textnormal{GL}_n(\mathcal{O}_{F_w})$ for $w\notin S_{\textnormal{bad}}$ and a maximal ideal $\mathfrak{m}\leq \mathbf{T}(K,1)$. Let $\overline{\rho}_{\mathfrak{m}}$ be the corresponding semisimple Galois representation constructed in \cite{Sch15}.
    
    The smooth $\mathfrak{Z}_{\textnormal{GL}_n(F_v)}[\textnormal{GL}_n(F_v)]$-module
    \begin{equation*}
        H(K^v)_{\mathfrak{m}}:=\varinjlim_{K_v}H^{\ast}(X_{K^vK_v}^G,\mathbf{F}_{\ell})_{\mathfrak{m}}
    \end{equation*}
    localises at the maximal ideal $\mathfrak{m}_v\leq \mathfrak{Z}_{\textnormal{GL}_n(F_v)}$ corresponding to $(\overline{\rho}_{\mathfrak{m}}|_{G_{F_v}})^{\textnormal{ss}}$ under the Tate-normalised mod $\ell$ semisimple local Langlands correspondence of Vigneras \cite{Vig01}.

    In particular, $H(K^v)_{\mathfrak{m}}$ is of finite length as a smooth $\mathbf{F}_{\ell}[\textnormal{GL}_n(F_v)]$-module.
\end{Cor}
\begin{proof}
    The first claim is an immediate consequence of Theorem \ref{Thm_MainResultIntro}.

    For the second claim note that, due to the existence of the Borel--Serre compactification, $H(K^v)_{\mathfrak{m}}$ is an admissible $\mathbf{F}_{\ell}$-representation of $\textnormal{GL}_n(F_v)$. By the first claim, it also has bounded depth. In particular, it must also be finitely generated and, therefore, of finite length.
\end{proof}

We sketch the proof of Theorem \ref{Thm_MainResultIntro}. Following the construction of \cite{Sch15}, we use the boundary cohomology of the quasi-split unitary group $\widetilde{G}$ over the maximal totally real subfield $F^+\leq F$ to study the Hecke eigensystems appearing in the cohomology of $G=\textnormal{Res}_{F/F^+}\textnormal{GL}_n$-locally symmetric spaces by realising $G$ as the Levi factor of the Siegel parabolic subgroup $P\leq \widetilde{G}$. The main innovation of our work is to blend the construction of Scholze with the robust theory of $\mathbf{Z}_{\ell}$-representations of $p$-adic reductive groups. We proceed in three steps.
\begin{enumerate}
    \item We first show that for any maximal ideal $\mathfrak{m}\leq \mathbf{T}^{S_{\textnormal{bad}}}(K,m)$ and finite place $\overline{v}=v\cdot v^c$ of $F^+$ under $ S_{\textnormal{ram}}$, the $\mathfrak{Z}_{G(F^+_{\Bar{v}})}=\mathfrak{Z}_{\textnormal{GL}_n(F_v)}\otimes_{\mathbf{Z}_{\ell}}\mathfrak{Z}_{\textnormal{GL}_n(F_{v^c})}$-module
    \begin{equation*}
        H_!(K^{\overline{v}})_{\mathfrak{m}}:=\im(H^{\ast}_c(K^{\overline{v}})_{\mathfrak{m}}\to H^{\ast}(K^{\overline{v}})_{\mathfrak{m}})
    \end{equation*} localises at the maximal ideal $(\mathfrak{m}_v,\mathfrak{m}_{v^c})$ associated with $\big((\overline{\rho}_{\mathfrak{m}}|_{G_{F_v}})^{\textnormal{ss}},(\overline{\rho}_{\mathfrak{m}}|_{G_{F_{v^c}}})^{\textnormal{ss}}\big)$ for the semisimple Galois representation $\overline{\rho}_{\mathfrak{m}}$ constructed in \cite{Sch15}.
    
    To do this, we realise $\textnormal{Ind}_{P(F_{\overline{v}}^+)}^{\widetilde{G}(F^+_{\Bar{v}})}H_!(K^{\overline{v}})_{\mathfrak{m}}$ as a $\widetilde{G}(F^+_{\Bar{v}})$-equivariant subquotient of
    \begin{equation*}
        H_{\partial}(\widetilde{K}^{\Bar{v}})_{\widetilde{\mathfrak{m}}}:=\varinjlim_{\widetilde{K}_{\overline{v}}\leq \widetilde{G}(F^+_{\overline{v}})}H^{\ast}(\partial X^{\widetilde{G}}_{\widetilde{K}^{\overline{v}}\widetilde{K}_{\overline{v}}}, \mathbf{F}_{\ell})_{\widetilde{\mathfrak{m}}}\in \textnormal{Mod}_{\textnormal{sm}}(\mathbf{F}_{\ell}[\widetilde{G}(F^+_{\overline{v}})])
    \end{equation*}
    where $\partial X^{\widetilde{G}}_{\widetilde{K}}$ denotes the Borel--Serre boundary and $\widetilde{\mathfrak{m}}\leq \widetilde{\mathbf{T}}^{S_{\textnormal{bad}}}$ is the pullback of $\mathfrak{m}$ under the Satake transform.
    
    By \cite{Sch15} Chapter IV and local-global compatibility for cusp forms for $\widetilde{G}$, $H_{\partial}(\widetilde{K}^{\Bar{v}})_{\widetilde{\mathfrak{m}}}$ must localise at the maximal ideal $\widetilde{\mathfrak{m}}_{\Bar{v}}$ corresponding to
    \begin{equation*}
        (\rho_{\widetilde{\mathfrak{m}}}|_{G_{F_v}})^{\textnormal{ss}}=\big(\rho_{\mathfrak{m}}|_{G_{F_v}}\big)^{\textnormal{ss}}\oplus \big(\rho_{\mathfrak{m}}^{\vee,c}(1-2n)|_{G_{F_{v}}}\big)^{\textnormal{ss}}.
    \end{equation*} 
    This determines the support $H_!(K^{\overline{v}})_{\mathfrak{m}}:=\im(H^{\ast}_c(K^{\overline{v}})_{\mathfrak{m}}\to H^{\ast}(K^{\overline{v}})_{\mathfrak{m}})$ as a $\mathfrak{Z}_{\widetilde{G}(F^+_{\Bar{v}})}$-module\footnote{The $\mathfrak{Z}_{\widetilde{G}(F^+_{\Bar{v}})}$-module structure is via the usual map $\mathfrak{Z}_{\widetilde{G}(F^+_{\Bar{v}})}\to \mathfrak{Z}_{G(F^+_{\Bar{v}})}$ intertwining parabolic induction.} and so provides information on its support as a $\mathfrak{Z}_{G(F^+_{\Bar{v}})}$-module. By considering a suitable twist of $\mathfrak{m}$ by a global mod $\ell$ character, we show that this forces the support with respect to $\mathfrak{Z}_{G(F^+_{\Bar{v}})}$-action to be $(\mathfrak{m}_v,\mathfrak{m}_{v^c})$.
    \item In the second step, we prove Theorem \ref{Thm_MainResultIntro} for
    \begin{equation*}
        \mathbf{T}^{S_{\textnormal{avoid}}}_!(K,m)_{\mathfrak{m}}:=\im\Big(\mathbf{T}^{S_{\textnormal{avoid}}}\to \textnormal{End}_{\mathbf{Z}_{\ell}}\big(H_!(X_K^G,\mathbf{Z}/\ell^m\mathbf{Z})\big)\Big)_{\mathfrak{m}}.
    \end{equation*}
    We do this by applying the "Key lemma" \ref{KeyLemma} which shows that, after twisting $\mathfrak{m}$ by a suitable global mod $\ell$ character, any lift of 
    \begin{equation*}
        \big(\rho_{\mathfrak{m}}|_{G_{F_v}}\big)^{\textnormal{ss}}\oplus \big(\rho_{\mathfrak{m}}^{\vee,c}(1-2n)|_{G_{F_{v}}}\big)^{\textnormal{ss}}
    \end{equation*} to a continuous representation $\widetilde{\rho}:G_{F_v}\to \textnormal{GL}_{2n}(\overline{\mathbf{Q}}_{\ell})$ with associated smooth representation $\widetilde{\pi}$ of $\textnormal{GL}_{2n}(F_v)$ must satisfy the following.
    \begin{itemize}
        \item There is an isomorphism $\widetilde{\rho}=\rho_1\oplus\rho_2$ with $\rho_1$ lifting $\big(\rho_{\mathfrak{m}}|_{G_{F_v}}\big)^{\textnormal{ss}}$.
        \item The localised Jacquet module
        \begin{equation*}
            \big(J_{P_{(n,n)}(F_v)}(\widetilde{\pi})\big)_{(\mathfrak{m}_v,\mathfrak{m}_{v^c})}
        \end{equation*} is of the form $\pi_1\otimes\pi_2$ with $\pi_1$ matching $\rho_1$ under the Tate-normalised local Langlands correspondence.
    \end{itemize}
    
    By (i), we have
    \begin{equation*}
        \mathbf{T}^{S_{\textnormal{avoid}}}_!(K,m)_{\mathfrak{m}} = \left(\mathbf{T}^{S_{\textnormal{avoid}}}_!(K,m)_{\mathfrak{m}}\right)_{(\mathfrak{m}_v,\mathfrak{m}_{v^c})}.
    \end{equation*} In particular, applying \cite{Sch15}, Chapter IV to the Jacquet-module at $\overline{v}$ of the boundary cohomology of $\widetilde{G}$ localised at $(\mathfrak{m}_v,\mathfrak{m}_{v^c})$, we obtain fine enough lifts of $\mathbf{T}^{S_{\textnormal{avoid}}}_!(K,m)_{\mathfrak{m}} $ to characteristic $0$ to deduce local-global compatibility.
    \item Finally, to obtain local-global compatibility for $\mathbf{T}^{S_{\textnormal{avoid}}}(K,m)$, it suffices to verify local-global compatibility for the boundary cohomology of $G$. We establish this by running an induction on $n\geq 1$. The simplicity of this inductive argument demonstrates how robust the formulation of Conjecture \ref{Conj_LGC_torsion} is.
    
\end{enumerate}

\subsection{An application to automorphic Galois representations}\label{subsection_intro_application}
One motivation to consider an extension of Conjecture \ref{Conj_Reciprocity_char0} to torsion automorphic forms comes from the work of Calegari--Geraghty. In \cite{CG18}, they extend the Taylor--Wiles method to the cohomology of $\textnormal{Res}_{F/\mathbf{Q}}\textnormal{GL}_n$-locally symmetric spaces provided that, among other things, 
Conjecture \ref{Conj_LGC_torsion} holds at Taylor--Wiles primes for non-Eisenstein maximal ideals. Their conjecture was essentially verified in \cite{ACC23} for CM number fields and was an ingredient in the proof of their automorphy lifting theorems.

On the other hand, the Taylor--Wiles method has also found applications in cases where the mod $\ell$ Hecke eigensystem is Eisenstein. For instance, in \cite{New23} patching is carried out at a dimension one prime of the Hecke algebra to deduce vanishing of the conjugate self-dual part of the adjoint Bloch--Kato Selmer group of residually reducible conjugate self-dual automorphic Galois representations. In \cite{Aca26} we use Theorem \ref{Thm_MainResultIntro} as an ingredient to extend their vanishing results to not necessarily conjugate self-dual automorphic Galois representations.

\subsubsection*{Acknowledgements}
This article grew out of answering some questions raised by Ana Caraiani. I am deeply grateful to her for posing these questions and for encouraging me to work on this problem. I also thank Vincent Pilloni, and Justin Trias for conversations from which this work benefited.

This project has received funding from the European Union’s Horizon 2020 research and innovation programme under the Marie Sklodowska-Curie grant agreement No 101034255. This work was supported by the Engineering and Physical Sciences Research Council [EP/S021590/1]. 
The EPSRC Centre for Doctoral Training in Geometry and Number Theory (The London School of Geometry and Number Theory), University College London, King's College London and Imperial College London. This project has received funding from the European Research Council (ERC) under the European Union’s Horizon 2020 research and innovation programme (grant agreement No. 804176).

\section*{Notation and Conventions}\label{Notations}
Given a number field $F$, we will denote by $S(F)$ its set of finite places and by $S_p(F)$ its set of $p$-adic places. We set $G_F$ to be the absolute Galois group $\textnormal{Gal}(\overline{F}/F)$ and for a finite set $S\subset S(F)$ we denote by $G_{F,S}$ the quotient of $G_F$ corresponding to the maximal Galois extension of $F$, unramified outside $S$. For $v\in S(F)$, set $F_v$ to be the $v$-adic completion of $F$, fix a choice of uniformiser $\varpi_v$ and set $k_v:=\mathcal{O}_{F_v}/\varpi_v$ to be its residue field. Set $G_{F_v}:=\textnormal{Gal}(\overline{F}_v/F_v)$. Moreover, $I_{F_v}\subset G_{F_v}$ will denote its inertia subgroup and set $\textnormal{Frob}_v\in G_{F_v}/I_{F_v}$ to be the geometric Frobenius. We denote by $\mathbf{A}_F$ the ring of adeles of $F$ and for $S$ a finite set of places of $F$ we denote by $\mathbf{A}_{F}^S$ its prime-to-$S$ part.

For $G$ a reductive group over a number field $F$ and a finite set $S\subset S(F)$, we will denote by $G^S:=G(\mathbf{A}_{F}^{S\cup \infty})$ and by $G_S:=G(\prod_{v\in S}F_v)$. %Moreover, if $S=S_p(F)$, we only write $G^p$, respectively $G_p$.

%Let $G$ be a linear algebraic group over $\mathcal{O}_L$ where $\mathcal{O}_L$ is the ring of integers of a finite extension $L/\mathbf{Q}_p$. Let $\varpi_L\in \mathcal{O}_L$ be a choice of uniformiser. For $n\in \mathbf{Z}_{\geq 0}$, denote $\ker(G(\mathcal{O}_L)\to G(\mathcal{O}_L/\varpi_L^n))$ by $G^n$.

For $G$ a reductive group over a finite extension $L/\mathbf{Q}_p$ with parabolic subgroup $Q=M\ltimes N \subset G$, we denote by $\delta_Q:M(L)\to \mathbf{Q}^{\times}$ the corresponding modulus character $x\mapsto |\det(\textnormal{ad}(x)|\textnormal{Lie}N)|_L$. 

For a smooth representation $\sigma$ of $M(L)$, we denote by $\textnormal{Ind}_{Q(L)}^{G(L)}\sigma$ the unnormalised parabolic induction. For a smooth $\overline{\mathbf{Q}}_{\ell}$-representation $\pi$ of $G(L)$, we will denote by $J_Q(\pi)$ its \textit{unnormalised} Jacquet module associated with $Q$. In general we denote by $\textnormal{c-Ind}$ compact induction for smooth representations.

For a smooth irreducible representation $\pi$ of $\textnormal{GL}_n(L)$ with supercuspidal support $(\textnormal{GL}_{n_1}(L)\times...\times \textnormal{GL}_{n_k}(L),\pi_1\otimes...\otimes \pi_k)$, we define the multiset $\textnormal{SC}(\pi):=\{\pi_1,...,\pi_k\}$.

Set $W_L$ to be the Weil group of $L$ and write $\textnormal{Art}_L:L^{\times}\xrightarrow{\sim} W_L^{\textnormal{ab}}$ for the Artin map of local Class Field Theory normalised by sending uniformisers to lifts of the geometric Frobenius.

We denote by $\textnormal{rec}_L$ the local Langlands correspondence for $L$ constructed in \cite{HT01}. If it is clear from the context, we will just write $\textnormal{rec}$ instead. Moreover, for $\pi$ an irreducible admissible $\textnormal{GL}_n(L)$-representation, we set $\textnormal{rec}^T(\pi)=\textnormal{rec}(\pi\otimes|\det|_L^{\frac{1-n}{2}})$. Then $\textnormal{rec}^T$ commutes with $\textnormal{Aut}(\mathbf{C})$ and therefore $\textnormal{rec}^T$ makes sense over $\overline{\mathbf{Q}}_{\ell}$ by choosing an abstract isomorphism $\iota:\overline{\mathbf{Q}}_{\ell}\xrightarrow{\sim} \mathbf{C}$. In the literature, this is often called the \textit{Tate normalisation} of the local Langlands correspondence.

For an $\ell$-adic Galois representation $\rho:G_L\to \textnormal{GL}_n(\overline{\mathbf{Q}}_{\ell})$ with $\ell\neq p$, we denote by $\textnormal{WD}(\rho)$ the associated Weil--Deligne representation. For a Weil--Deligne representation $(r,N)$, we denote by $(r,N)^{F-ss}$ its Frobenius semisimplification and by $(r,N)^{ss}$ its semisimplification.

For a ring $R$, we denote by $D(R)$ the derived category of left $R$-modules and by $D^+(R)$ the bounded below derived category. Given a locally profinite group $G$, we will denote by $\textnormal{Mod}_{\textnormal{sm}}(R[G])$ the category of smooth $R$-representations of $G$ and by $D^+_{\textnormal{sm}}(G,R)$ its bounded below derived category.

Given a topological group $G$, and a topological space $X$ with a continuous right action of $G$, we denote by $\textnormal{Sh}_G(X)$ the category of $G$-equivariant sheaves on $X$ in the sense of \cite{NT16}, Definition 2.22, (2). Moreover, for a ring $R$, we denote by $\textnormal{Sh}_G(X,R)$ the category of $G$-equivariant sheaves of $R$-modules on $X$.

For $G/L$ a split reductive group with a choice of a Borel subgroup $B$ and a maximal torus $T$, denote by $w_0^G$ the longest element in the Weyl group $W_G:=W(G,T)$. For a standard parabolic subgroup $Q\subset G$ with Levi decomposition $M\ltimes N$, set $W^Q\subset W_G$ to be the set of minimal length representatives of $ W_G/W_Q$. We denote by $w_0^Q$ the longest element in $W^Q$ that is, in fact, given by $w_0^Gw_0^M$. Similar notations apply to ${}^QW$. Moreover, for another standard parabolic subgroup $Q'\subset G$ with Levi decomposition $M'\ltimes N'$, denote by ${}^{Q'}W^{Q}\subset W_G$ the set of minimal length representatives of $W_{Q'}\backslash W_G/W_Q$.

\section{Representations of $p$-adic reductive groups with $\ell$-adic coefficients}
Let $R$ denote a commutative $\mathbf{Z}_{\ell}$-algebra. Fix a rational prime $p$ different from $\ell$, a finite field extension  $ L/ \mathbf{Q}_p$, and an integer $n\geq 1$. In this section, we study the category $\textnormal{Mod}_{\textnormal{sm}}(R[\textnormal{G}])$ of smooth $R$-representations of a standard Levi subgroup $\textnormal{G}$ of $\textnormal{GL}_n(L)$ and its relation to continuous semisimple representations of the Weil group $W_L$ via the local Langlands correspondence for various choices of $R$.

We start with a brief introduction to the theory of smooth $R$-representations of $\textnormal{G}$. We then introduce the moduli of pseudocharacters of $L$-parameters and the corresponding map interpolating the semisimple local Langlands correspondence. Finally, in the last subsection we collect the technical results we need in our proof of local-global compatibility.
\subsection{Reminder on smooth representation theory}
\subsubsection{The integral Bernstein centre}
We denote by $\mathfrak{Z}_{\textnormal{G},R}$ the ring of endomorphisms of the identity endofunctor $\textnormal{id}_{\textnormal{Mod}_{\textnormal{sm}}(R[\textnormal{G}])}\in \textnormal{End}(\textnormal{Mod}_{\textnormal{sm}}(R[\textnormal{G}]))$
and we refer to it as the Bernstein centre of $\textnormal{G}$ over $R$. By abuse of notation we will simply write $\mathfrak{Z}_{\textnormal{G}}$ for the Bernstein centre $\mathfrak{Z}_{\textnormal{G},\mathbf{Z}_{\ell}}$ of $\textnormal{G}$ over $\mathbf{Z}_{\ell}$. 

Concretely, elements $z\in \mathfrak{Z}_{\textnormal{G},R}$ are given by collections of endomorphisms
\begin{equation*}
    z_{\pi}:\pi\to \pi
\end{equation*}
for every representation $\pi\in \textnormal{Mod}_{\textnormal{sm}}(R[\textnormal{G}])$ such that, for every morphism $f:\pi\to \pi'$ in $\textnormal{Mod}_{\textnormal{sm}}(R[\textnormal{G}])$,
\begin{equation}\label{centrecoherence}
    z_{\pi'}\circ f=f\circ z_{\pi}
\end{equation}
and the ring structure is given by componentwise addition and composition.

Consequently, for every $\pi\in \textnormal{Mod}_{\textnormal{sm}}(R[\textnormal{G}])$, we obtain an algebra homomorphism
\begin{equation*}
    \mathfrak{Z}_{\textnormal{G},R}\xrightarrow{t_{\textnormal{G},\pi}}Z(\textnormal{End}_{R[\textnormal{G}]}(\pi)),
\end{equation*}
\begin{equation*}
    z\mapsto z_{\pi}.
\end{equation*}
In particular, for every compact open subgroup $K\leq \textnormal{G}$, we have a natural algebra homomorphism
\begin{equation*}
t_{\textnormal{G},K}:\mathfrak{Z}_{\textnormal{G},R}\xrightarrow{t_{\textnormal{G},\textnormal{c-Ind}_K^{\textnormal{G}}\mathbf{1}}}Z(\textnormal{End}_{R[\textnormal{G}]}(\textnormal{c-Ind}_K^{\textnormal{G}}\mathbf{1}))\cong Z(R[K\backslash \textnormal{G}/K])
\end{equation*}
where the target is equipped with the algebra structure given by convolution product. Namely, the isomorphism we postcompose by is realised by acting with the Hecke algebra $R[K\backslash \textnormal{G}/K]$ on $\textnormal{c-Ind}_K^{\textnormal{G}}\mathbf{1}$ via convolution.

An unravelling of the definitions then gives the following simple observation.
\begin{Lemma}\label{BernsteinCentreLemma}
    Let $\pi\in \textnormal{Mod}_{\textnormal{sm}}(R[\textnormal{G}])$, and $K\leq \textnormal{G}$ be a compact open subgroup. We have a commutative diagram
    \begin{equation*}
        \begin{tikzcd}
	&& {Z(\textnormal{End}_{R[\textnormal{G}]}(\pi))} \\
	{\mathfrak{Z}_{\textnormal{G},R}} &&&& {\textnormal{End}_{R[K\backslash \textnormal{G}/K]}(\pi^K)} \\
	&& {Z(R[K\backslash \textnormal{G}/K])}
	\arrow[from=1-3, to=2-5]
	\arrow["{t_{\textnormal{G},\pi}}", from=2-1, to=1-3]
	\arrow["{t_{\textnormal{G},K}}"', from=2-1, to=3-3]
	\arrow["{z_K\mapsto z_{K}\ast(-)}"', from=3-3, to=2-5]
\end{tikzcd}
    \end{equation*}
    natural in $\pi$ where $(-)\ast (-)$ denotes the action via convolution and the top right map is (the restriction to the centre of) the natural map sending a $\textnormal{G}$-equivariant map $\pi\to \pi$ to its restriction along $\pi^K\subset \pi$.
\end{Lemma}

 We also note that, for a $\mathbf{Z}_{\ell}$-algebra homomorphism $f:R\to R'$, by viewing an $R'[\textnormal{G}]$-module as an $R[\textnormal{G}]$-module via restriction of scalars, we obtain a pullback morphism $\mathfrak{Z}_{\textnormal{G},R}\to \mathfrak{Z}_{\textnormal{G},R'}$ that moreover is injective, as long as $f$ is injective (cf. \cite{DHKM24,} Lemma 3.1, Lemma 3.2).

\subsubsection{The integral Bernstein decomposition and preliminaries}
The Bernstein decomposition (cf. \cite{BD84}) states that the category $\textnormal{Mod}_{\textnormal{sm}}(\overline{\mathbf{Q}}_{\ell}[\textnormal{G}])$ admits a direct sum decomposition labelled by the so-called inertial supercuspidal supports. Moreover, the corresponding direct factors of $\mathfrak{Z}_{\textnormal{G},\overline{\mathbf{Q}}_{\ell}}$ are explicitly described, showing that they are of finite type over $\overline{\mathbf{Q}}_{\ell}$. We recall the characteristic $\ell$ and $\ell$-integral versions of this decomposition, due to Vigneras, and Helm, respectively, and remind the reader how these decompositions interact with each other.
\subsubsection{Parabolic induction}
For a parabolic subgroup $\textnormal{Q}\leq \textnormal{G}$ with a Levi decomposition $\textnormal{Q}=\textnormal{M}\textnormal{N}$, we have the unnormalised Jacquet functor $J_{\textnormal{Q}}:\textnormal{Mod}_{\textnormal{sm}}(R[\textnormal{G}])\to \textnormal{Mod}_{\textnormal{sm}}(R[\textnormal{M}])$, and the unnormalised parabolic induction $\textnormal{Ind}_{\textnormal{Q}}^{\textnormal{G}}(-):\textnormal{Mod}_{\textnormal{sm}}(R[\textnormal{M}])\to \textnormal{Mod}_{\textnormal{sm}}(R[\textnormal{G}])$, the latter being right adjoint to the former.

We fix a square root $q^{1/2}$ of the cardinality $q\in (\mathbf{Z}[p^{-1}])^{\times}\subset \mathbf{Z}_{\ell}^{\times}$ of the residue field of $L$. Assuming that $R$ is a $\mathbf{Z}_{\ell}[q^{1/2}]$-algebra, we have a well-defined square root $\delta_{\textnormal{Q}}^{1/2}:\textnormal{M}\to R^{\times}$ of the modulus character. In particular, we can define the \textit{normalised} Jacquet functor $\textnormal{n-}J_{\textnormal{Q}}:=\delta_{\textnormal{Q}}^{-1/2}J_{\textnormal{Q}}(-)$, and the \textit{normalised} parabolic induction $\textnormal{n-\textnormal{Ind}}_{\textnormal{Q}}^{\textnormal{G}}:=\textnormal{Ind}_{\textnormal{Q}}^{\textnormal{G}}(\delta_{\textnormal{Q}}^{1/2}-)$ that once again form an adjoint pair.

\begin{Rem}\label{ModulusRemark}
    Let $\textnormal{G}=\textnormal{GL}_n(L)$ and $\textnormal{Q}:=P_{(n_1,...,n_k)}(L)$ the standard parabolic subgroup of block upper-triangular matrices corresponding to a partition $n=n_1+...+n_k$. Write $\textnormal{M}:=M_{(n_1,...n_k)}(L)\cong\textnormal{GL}_{n_1}(L)\times...\times \textnormal{GL}_{n_k}(L)$ for its standard Levi subgroup and, for $1\leq i\leq k$, denote by $\textnormal{GL}_{n_i}(L)\cong \textnormal{G}_i\leq \textnormal{M}$ the image of the embedding $A\mapsto \textnormal{diag}(\mathbf{1}_{n_1\times n_1},...,\mathbf{1}_{n_{i-1}\times n_{i-1}},A,\mathbf{1}_{n_{i+1}\times n_{i+1}},...,\mathbf{1}_{n_k\times n_k})$. Then an easy computation shows that
    \begin{equation}\label{modulusformula}
        \delta_{\textnormal{Q}}|_{\textnormal{G}_i}(A)=|\det(A)|_L^{(n-n_i)-2(n_1+...+n_{i-1})}.
    \end{equation}
    In particular, $\delta_{\textnormal{Q}}|_{\textnormal{G}_i}$ is an integral power of the determinant character.
\end{Rem}

\subsubsection{(Super)cuspidal support over algebraically closed fields}
Let $R$ now be an \textit{algebraically closed field}. We say that an irreducible $R$-representation $\Pi$ of $\textnormal{G}$ is \textit{cuspidal} if, for every proper parabolic subgroup $\textnormal{Q}\leq \textnormal{G}$ with a Levi decomposition $\textnormal{Q}=\textnormal{M}\textnormal{N}$, the Jacquet module $J_{\textnormal{Q}}(\Pi)$ vanishes.

We say that an irreducible $R$-representation $\Pi$ of $\textnormal{G}$ is \textit{supercuspidal} if there is no proper parabolic subgroup $\textnormal{Q}\leq \textnormal{G}$ with a Levi decomposition $\textnormal{Q}=\textnormal{M}\textnormal{N}$ and an irreducible $R$-representation $\pi$ of $\textnormal{M}$ such that $\Pi$ is isomorphic to a subquotient of $\textnormal{n-Ind}_{\textnormal{Q}}^{\textnormal{G}}\pi$. Even though in characteristic $0$ being supercuspidal coincides with being cuspidal, in characteristic $\ell$ the notions might differ.

By the \textit{cuspidal support} (respectively \textit{supercuspidal support}) of an irreducible $R$-representation $\Pi$ of $\textnormal{G}$, we will mean the set of pairs $(\textnormal{M},\pi)$ of a Levi subgroup $\textnormal{M}$ of $\textnormal{G}$, and a cuspidal $R$-representation of $\textnormal{M}$ such that $\Pi$ is isomorphic to an $R[\textnormal{G}]$-equivariant quotient (respectively subquotient) of $\textnormal{n-Ind}_{\textnormal{Q}}^{\textnormal{G}}\pi$ for some parabolic subgroup $\textnormal{Q}=\textnormal{M}\textnormal{N}$. More generally, for an absolutely irreducible smooth representation $\Pi$ of $\textnormal{G}$ over some field $E$, we will refer to the cuspidal (respectively supercuspidal) support of $\Pi\otimes_{E}\overline{E}$ as the cuspidal (respectively supercuspidal) support of $\pi$.

Recall that we say two pairs $(\textnormal{M},\pi)$, and $(\textnormal{M}',\pi')$ as above are \textit{G-conjugate} if there exists an element $g\in \textnormal{G}$ such that $g\textnormal{M}g^{-1}=\textnormal{M}'$, and $\pi\cong \pi'\circ \textnormal{ad}(g)$. By work of Vigneras (cf. \cite{Vig96}, II 2.20 respectively, \cite{Vig98}, V.4), we know that both the cuspidal and supercuspidal support is unique up to $\textnormal{G}$-conjugation.

\begin{Def}
    Let $(\textnormal{M},\pi)$ be a representative of a supercuspidal support for $\textnormal{Mod}_{\textnormal{sm}}(R[\textnormal{G}])$. We call $(\textnormal{M},\pi)$ a \textit{standard} representative if the Levi subgroup $\textnormal{M}\leq\textnormal{G}$ is standard with respect to the Borel subgroup of upper triangular matrices.
\end{Def}
\begin{Rem}\label{StandardRemark}
    Notice that every supercuspidal support has a standard representative. Moreover, if $\textnormal{G}=\textnormal{GL}_n(L)$ and $(\textnormal{M},\pi)$ is a standard representative of a supercuspidal support for $\textnormal{Mod}_{\textnormal{sm}}(R[G])$, then $\textnormal{M}$ can uniquely be written as a product of auxiliary general linear groups
    \begin{equation*}
        \textnormal{M}=\textnormal{GL}_{n_1}(L)\times...\times\textnormal{GL}_{n_k}(L)
    \end{equation*}
    for a partition $n=n_1+...+n_k$ for some integer $k\geq 1$. In particular, $\textnormal{M}$ is uniquely determined by this partition. Accordingly, we can write
    \begin{equation*}
        \pi=\pi_1\otimes...\otimes\pi_k.
    \end{equation*}

    Moreover, given another standard representative $(\textnormal{M}',\pi')$, $\pi'$ must be of the form $\pi_{\sigma(1)}\otimes...,\otimes\pi_{\sigma(k)}$ for some permutation $\sigma\in S_k$.

    More generally, for $\textnormal{G}=\textnormal{GL}_{m_1}(L)\times...\times\textnormal{GL}_{m_l}(L)$, standard representatives $(\textnormal{M},\pi)$ of supercuspidal supports for $\textnormal{Mod}_{\textnormal{sm}}(R[\textnormal{G}])$  are of the form 
    \begin{equation*}
(\pi_{1,1}\otimes...\otimes\pi_{1,k_1})\otimes...\otimes(\pi_{l,1}\otimes...\otimes\pi_{l,k_l})
    \end{equation*} 
    for integers $l\geq 1$, $\{k_j\geq 1\}_{1\leq j\leq l}$
    with the $\pi_{i,j}$'s being supercuspidal representations of auxiliary general linear groups. Again, another standard representative is obtained by applying an element of $S_{k_1}\times...\times S_{k_l}$ to the labels of $\pi$.
\end{Rem}

\subsubsection{Integral representations and the mod $\ell$ reduction map}
Recall that, when $R$ is a Noetherian $\mathbf{Z}_{\ell}$-algebra, we say that a smooth $R[\textnormal{G}]$-module $\pi$ is \textit{admissible} if $\pi^K$ is a finitely generated $R$-module for every compact open subgroup $K\leq\textnormal{G}$.

We say that an admissible $\overline{\mathbf{Q}}_{\ell}$-representation $\pi$ of $\textnormal{G}$
is $\ell$\textit{-integral} if there exists a finite field extension $E/\mathbf{Q}_{\ell}$ with ring of integers $\mathcal{O}$, and an $\mathcal{O}[\textnormal{G}]$-submodule $\pi^{\circ}\subset \pi$ that is admissible as an $\mathcal{O}[\textnormal{G}]$-module and spans $\pi$ over $\overline{\mathbf{Q}}_{\ell}$. We will call $\pi^{\circ}$ an $\mathcal{O}$-\textit{lattice} of $\pi$.
\begin{Rem}
If $\pi^{\circ}\subset \pi$ is an $\mathcal{O}$-lattice of an admissible $\overline{\mathbf{Q}}_{\ell}[\textnormal{G}]$-module, it is an $\mathcal{O}[\textnormal{G}]$-lattice of $\pi^{\circ}\otimes_{\mathcal{O}}E$ in the sense of \cite{Vig96}, I.9.1. Moreover, $\mathcal{O}$ being a principal ideal domain, and $\pi^{\circ}$ having countable $\mathcal{O}$-rank, it is equivalent to ask for the assertion that $\pi^{\circ}$ is $\mathcal{O}$-free (cf. \cite{Vig96}, I.9.2). In particular, the notion of $\ell$-integral is equivalent to the one appearing in \cite{Vig98}, IV.1.5, and \cite{Vig01}, 1.5 respectively.

We further note that if $\pi$ is an irreducible $\overline{\mathbf{Q}}_{\ell}$-representation of $\textnormal{G}$, then $\pi$ is $\ell$-integral if and only if it admits a $\overline{\mathbf{Z}}_{\ell}$-free $\overline{\mathbf{Z}}_{\ell}[\textnormal{G}]$-submodule that generates it over $\overline{\mathbf{Q}}_{\ell}$ (i.e. our definition agrees with \cite{Vig96}, II.4.11).
Indeed, this follows from the previous paragraph and the fact that $\pi$ can always be realised over a finite field extension $E/\mathbf{Q}_{\ell}$ (cf. \cite{Vig96}, II 4.9, 4.10).
\end{Rem}

\begin{Ex}
Let $G$ be a connected reductive group over a number field $F$. Let $v|p$ be a finite place of $F$ and assume that $G_{F_{v}}\cong \textnormal{GL}_n$. Set $\textnormal{G}:=G(F_{v})\cong \textnormal{GL}_{n}(F_v)$. Then, for a cohomological cuspidal automorphic representation $\pi$ of $G(\mathbf{A}_{F})$, and a field isomorphism $\iota:\overline{\mathbf{Q}}_{\ell}\xrightarrow{\sim}\mathbf{C}$, $\iota^{-1}\pi_{v}$ is an example of an $\ell$-integral admissible $\overline{\mathbf{Q}}_{\ell}[\textnormal{G}]$-module.
\end{Ex}
%\begin{Lemma}\label{ellreduction}
%    Let $\pi$ be an $\ell$-integral admissible $\overline{\mathbf{Q}}_{\ell}$-representation of $\textnormal{G}$ of finite length. Then, for every $\mathcal{O}$-lattice $\pi^{\circ}\subset \pi$ for the ring of integers of some subfield $\mathbf{Q}_{\ell}\subset E\subset\overline{\mathbf{Q}}_{\ell}$, finite over $\mathbf{Q}_{\ell}$, $\pi^{\circ}\otimes_{\mathcal{O}}\overline{\mathbf{F}}_{\ell}$ is of finite length as a smooth $\overline{\mathbf{F}}_{\ell}[\textnormal{G}]$-module. Moreover, its semisimplification is independent of the choice of lattice $\pi^{\circ}$.
%\end{Lemma}
%\begin{proof}
%    Note that $\pi^{\circ}\otimes_{\mathcal{O}}E$ is a finitely generated $E[\textnormal{G}]$-module and, consequently, $\pi^{\circ}$ is a finitely generated $\mathcal{O}[\textnormal{G}]$-module (see \cite{Vig96}, I. 9.3, vii). Therefore, as the Brauer--Nesbitt principle of Vigneras holds for $\textnormal{G}$ (cf. \cite{Vig96}, II.5.11b), we see that $\pi^{\circ}\otimes_{\mathcal{O}}k$ is of finite length and its semisimplification is independent of the chosen $\mathcal{O}$-lattice of $\pi^{\circ}\otimes_{\mathcal{O}}E$.
%\end{proof}

Given an $\ell$-integral smooth $\overline{\mathbf{Q}}_{\ell}[\textnormal{G}]$-module $\pi$ of \textit{finite length} with a choice of $\mathcal{O}-$lattice $\pi^{\circ}$, the Brauer--Nesbitt principle for $\textnormal{G}$ asserts that $\pi^{\circ}\otimes_{\mathcal{O}}\overline{\mathbf{F}}_{\ell}$ is of finite length and its semisimplification $r_{\ell}(\pi)$ is independent of the choice of $\pi^{\circ}$.

Finally, we recall the compatibility between forming the mod $\ell$ reduction and passage to  supercuspidal support (cf. \cite{Vig01}, 1.5).
\begin{Prop}\label{modlred}
    Let $\Pi$ be an irreducible smooth $\ell$-integral $\overline{\mathbf{Q}}_{\ell}$-representation of $\textnormal{G}$ and denote by $(\textnormal{M},\pi)$ its supercuspidal support. Then $\pi$ is an $\ell$-integral $\overline{\mathbf{Q}}_{\ell}[M]$-module and $r_{\ell}(\pi)$ is an irreducible cuspidal $\overline{\mathbf{F}}_{\ell}[\textnormal{M}]$-module. Moreover, the supercuspidal support of any Jordan--H\"older constituent of $r_{\ell}(\Pi)$ coincides with the supercuspidal support of $r_{\ell}(\pi)$.
\end{Prop}

\subsubsection{Inertial supercuspidal supports and Bernstein block decompositions}
Let $R\in \{\overline{\mathbf{Q}}_{\ell},\overline{\mathbf{F}}_{\ell}\}$, and consider two pairs $(\textnormal{M},\pi)$, and $(\textnormal{M}',\pi')$ of Levi subgroups $\textnormal{M}$ and $\textnormal{M}'$ of $\textnormal{G}$, and smooth representations $\pi$ and $\pi'$ of $\textnormal{M}$ and $\textnormal{M}'$, respectively. We say $(M,\pi)$, and $(M',\pi')$ are \textit{inertially equivalent} if there is a pair $(\textnormal{M},\pi'')$ that is $\textnormal{G}$-conjugate to $(\textnormal{M}',\pi')$ and $\pi''$ is a twist of $\pi$ by an unramified character $\chi:\textnormal{M}\to R^{\times}$.
We denote by $[\textnormal{M},\pi]$ the inertial equivalence class represented by $(\textnormal{M},\pi)$.

We will refer to the inertial equivalence class $[\textnormal{M},\pi]$ of the (super)cuspidal support of a smooth irreducible $R$-representation $\Pi$ of $\textnormal{G}$ as its \textit{inertial (super)cuspidal support}.

More generally, for a simple smooth $W(\overline{\mathbf{F}}_{\ell})[\textnormal{G}]$-module $\Pi$, Helm defined the notion of the \textit{mod $\ell$ inertial supercuspidal support} of $\Pi$, given by the inertial equivalence $[\textnormal{M},\pi]$ of a Levi subgroup $\textnormal{M}$ of $\textnormal{G}$ and a supercuspidal smooth $\overline{\mathbf{F}}_{\ell}[M]$-module (cf. \cite{Hel16}, Definition 4.12). Instead of giving the definition, we note that it is compatible with the mod $\ell$ and characteristic $0$ definitions. Namely, for an $\ell$-integral smooth irreducible $\overline{\mathbf{Q}}_{\ell}[\textnormal{G}]$-module $\Pi$, its mod $\ell$ inertial supercuspidal support is exactly the inertial supercuspidal support of \textit{any} simple $\overline{\mathbf{F}}_{\ell}[\textnormal{G}]$-subquotient of $r_{\ell}(\Pi)$ (cf. \cite{Hel16}, Proposition 4.13).

For $R\in \{\overline{\mathbf{F}}_{\ell},W(\overline{\mathbf{F}}_{\ell}),\overline{\mathbf{Q}}_{\ell}\}$ we set $\mathfrak{B}_{\textnormal{G},R}$ to be the set of (mod $\ell$) inertial equivalence classes for the category $\textnormal{Mod}_{\textnormal{sm}}(R[\textnormal{G}])$. For a (mod $\ell$) inertial supercuspidal support $[\textnormal{M},\pi]\in \mathfrak{B}_{\textnormal{G},R}$, we can form the full subcategory $\textnormal{Mod}_{\textnormal{sm}}(R[\textnormal{G}])_{[\textnormal{M},\pi]}\subset \textnormal{Mod}_{\textnormal{sm}}(R[\textnormal{G}])$ of smooth $R[\textnormal{G}]$-modules with each of their simple $R[\textnormal{G}]$-subquotients having (mod $\ell$) inertial supercuspidal support given by $[\textnormal{M},\pi]$. Then the Bernstein decomposition (cf. \cite{BD84} for $R=\overline{\mathbf{Q}}_{\ell}$, \cite{Vig98} IV.6.2 for $R=\overline{\mathbf{F}}_{\ell}$, and \cite{Hel16}, Theorem 11.8 for $R=W(\overline{\mathbf{F}}_{\ell})$) asserts that there is a direct product decomposition
\begin{equation*}
    \textnormal{Mod}_{\textnormal{sm}}(R[\textnormal{G}])=\prod_{[\textnormal{M},\pi]\in\mathfrak{B}_{\textnormal{G},R}}\textnormal{Mod}_{\textnormal{sm}}(R[\textnormal{G}])_{[\textnormal{M},\pi]}.
\end{equation*}
 One refers to the obtained direct factors as \textit{Bernstein blocks}.

The Bernstein decomposition yields a decomposition $\mathfrak{Z}_{\textnormal{G},R}=\prod_{[\textnormal{M},\pi]\in \mathfrak{B}_{\textnormal{G},R}}\mathfrak{Z}_{[\textnormal{M},\pi],R}$ where $\mathfrak{Z}_{[\textnormal{M},\pi],R}$ is the centre of the corresponding block. Bernstein and Helm then explicitly compute $\mathfrak{Z}_{[\textnormal{M},\pi],R}$ for $R=\overline{\mathbf{Q}}_{\ell}$ and $R=W(\overline{\mathbf{F}}_{\ell})$, respectively, deducing that it is a finitely generated, reduced and flat $R$-algebra.

Moreover, this allows one to describe the geometric points of $\mathfrak{Z}_{\textnormal{G},W(\overline{\mathbf{F}}_{\ell})}$ (cf. \cite{BD84}, \cite{Hel16}, Corollary 12.12). Namely, for $R\in \{\overline{\mathbf{F}}_{\ell},\overline{\mathbf{Q}}_{\ell}\}$ and a smooth irreducible $R[\textnormal{G}]$-module $\Pi$, $\mathfrak{Z}_{\textnormal{G},W(\overline{\mathbf{F}}_{\ell})}$ acts on $\Pi$ via scalars by Schur's lemma, giving an $R$-point of the Bernstein centre. This in fact sets up a one-to-one correspondence between the $R$-points of $\mathfrak{Z}_{\textnormal{G},W(\overline{\mathbf{F}}_{\ell})}$ and supercuspidal supports in $\textnormal{Mod}_{\textnormal{sm}}(R[\textnormal{G}])$.

We also note that there is a specialisation map
\begin{equation*}
\textnormal{sp}_{\textnormal{G}}:\mathfrak{B}_{\textnormal{G},\overline{\mathbf{Q}}_{\ell}}\to \mathfrak{B}_{\textnormal{G},W(\overline{\mathbf{F}}_{\ell})}(=\mathfrak{B}_{\textnormal{G},\overline{\mathbf{F}}_{\ell}}).
\end{equation*}
Namely, every inertial supercuspidal support $[\textnormal{M},\pi]\in \mathfrak{B}_{\overline{\mathbf{Q}}_{\ell}}$ has an $\ell$-integral representative $(\textnormal{M},\pi')$ (cf. \cite{Vig96}, II 4.12) and we can set $\textnormal{sp}_{\textnormal{G}}([\textnormal{M},\pi])$ to be the supercuspidal support of $r_{\ell}(\pi')$. It is in fact a surjective map with finite fibers (\cite{Vig98}, IV.6.2).\footnote{ When $\ell$ is a \textit{banal} prime for $\textnormal{G}$, it is known to be a bijection.}

Finally, \cite{Hel16}, Proposition 12.1 shows that, for $[\textnormal{M},\pi]\in \mathfrak{B}_{\textnormal{G},W(\overline{\mathbf{F}}_{\ell})}(=\mathfrak{B}_{\textnormal{G},\overline{\mathbf{F}}_{\ell}})$, the natural isomorphism $\mathfrak{Z}_{\textnormal{G},W(\overline{\mathbf{F}}_{\ell})}\otimes_{W(\overline{\mathbf{F}}_{\ell})}\overline{\mathbf{Q}}_{\ell}\cong \mathfrak{Z}_{\textnormal{G},\overline{\mathbf{Q}}_{\ell}}$ induces an isomorphism 
\begin{equation*}
\mathfrak{Z}_{[\textnormal{M},\pi],W(\overline{\mathbf{F}}_{\ell})}\otimes_{W(\overline{\mathbf{F}}_{\ell})}\overline{\mathbf{Q}}_{\ell}\cong\prod_{[\widetilde{\textnormal{M}},\widetilde{\pi}]\in \textnormal{sp}_{\textnormal{G}}^{-1}([\textnormal{M},\pi])} \mathfrak{Z}_{[\widetilde{\textnormal{M}},\widetilde{\pi}],\overline{\mathbf{Q}}_{\ell}}.
\end{equation*}

\subsection{The semisimple local Langlands correspondence in families}
To discuss local-global compatibility for $\ell$-adic families of automorphic Hecke eigensystems, we will use the interpolation map of semisimple local Langlands for $\textnormal{GL}_n$ constructed by Helm--Moss \cite{HM18}.
\subsubsection{The correspondence over $\overline{\mathbf{Q}}_{\ell}$ and $\overline{\mathbf{F}}_{\ell}$ and their relation}
Recall that, as an application of Grothendieck's $\ell$-adic monodromy theorem, there is an exact equivalence of categories $\varphi\mapsto \textnormal{WD}(\varphi)$ (cf. \cite{Del73}, \S8) between the following.
\begin{enumerate}
    \item Continuous\footnote{For a representation $\varphi:W_L\to \textnormal{GL}_n(\overline{\mathbf
    Q}_{\ell})$ we ask for continuity with respect to the $\ell$-adic topology on the target. Namely, a representation is continuous if it factors through some $\textnormal{GL}_{n}(E)$ for some finite field extension and the factored map is continuous for the $\ell$-adic topology on $\textnormal{GL}_n(E)$.} representations of $W_L$ on finite dimensional $\overline{\mathbf{Q}}_{\ell}$-vectorspaces.
    \item Weil--Deligne representations $(r,N)$ of $L$ on finite dimensional $\overline{\mathbf{Q}}_{\ell}$-vectorspaces.
\end{enumerate}
We therefore have a commutative diagram
\begin{equation}\label{ssLLchar0}
    \begin{tikzcd}[scale cd=0.8]
	{\{\varphi:W_L\to\textnormal{GL}_n(\overline{\mathbf{Q}}_{\ell}) \textnormal{ continuous}\}_{/\sim}} &&& {\{\textnormal{irreducible }\Pi\in\textnormal{Mod}_{\textnormal{sm}}(\overline{\mathbf{Q}}_{\ell}[\textnormal{G}])\}_{/\sim}} \\
	\\
	{\{\varphi:W_L\to\textnormal{GL}_n(\overline{\mathbf{Q}}_{\ell}) \textnormal{ continuous, s.s.}\}_{/\sim}} &&& {\{\textnormal{s.c. supports (M},\pi\textnormal{) for }\textnormal{Mod}_{\textnormal{sm}}(\overline{\mathbf{Q}}_{\ell}[\textnormal{G}])\}_{/\sim}}
	\arrow["{\varphi\mapsto \pi(\varphi):= \textnormal{rec}_L^{T,-1}(\textnormal{WD}(\varphi)^{F-ss})}", from=1-1, to=1-4]
	\arrow["{\varphi\mapsto \varphi^{\textnormal{ss}}}"', from=1-1, to=3-1]
	\arrow["{\Pi\mapsto \textnormal{scs}(\Pi)}", from=1-4, to=3-4]
	\arrow["{\varphi\mapsto \pi^{\textnormal{ss}}(\varphi)}"', from=3-1, to=3-4]
\end{tikzcd}
\end{equation}
with the bottom map $\pi^{\textnormal{ss}}(-)$ being a bijection. Here $\textnormal{rec}_L^{\textnormal{T}}(-):=\textnormal{rec}_L(-\otimes|\det|_L^{\frac{1-n}{2}})$ is the Tate-normalisation of the local Langlands correspondence $\textnormal{rec}_{L}$ constructed in \cite{HT01}, $\textnormal{scs}(-)$ denotes the map passing to the underlying supercuspidal support and in the right bottom set we consider equivalence classes with respect to $\textnormal{G}$-conjugation.

Moreover, a representation $\varphi:W_L\to \textnormal{GL}_n(\overline{\mathbf{Q}}_{\ell})$ extends to an $\ell$-adic Galois representation $\rho:G_L\to\textnormal{GL}_n(\overline{\mathbf{Q}}_{\ell})$ if and only if $\textnormal{WD}(\varphi)=(r,N)$ is $\ell$-\textit{integral}, meaning that $r(\textnormal{Frob}_L)$ has eigenvalues lying in $\overline{\mathbf{Z}}_{\ell}^{\times}$ for a (or equivalently any) choice of lift of the geometric Frobenius. On the other hand, $\textnormal{rec}_L$ is known to interchange the notion of $\ell$-integrality. In particular, objects on the LHS of the diagram \ref{ssLLchar0} that extend to $G_L$ are exactly the ones sent to $\ell$-integral objects on the right.

Vigneras proved (cf. \cite{Vig01}) the existence of a mod $\ell$ version of the semisimple correspondence.
\begin{Th}[Vigneras]\label{ModEllLL}
  There is a bijection between sets of isomorphism classes
  \begin{equation*}
      \begin{tikzcd}[scale cd=0.8]
	{\{\overline{\varphi}:W_L\to\textnormal{GL}_n(\overline{\mathbf{F}}_{\ell}) \textnormal{ continuous, s.s.}\}_{/\sim}} && {\{\textnormal{s.c. supports (M},\overline{\pi}\textnormal{) for }\textnormal{Mod}_{\textnormal{sm}}(\overline{\mathbf{F}}_{\ell}[\textnormal{G}])\}_{/\sim}}.
	\arrow["{\overline{\varphi}\mapsto \overline{\pi}^{\textnormal{ss}}(\overline{\varphi})}", from=1-1, to=1-3]
\end{tikzcd}
  \end{equation*}
  It is characterised by the property that, for every $\ell$-integral, continuous representation $\varphi:W_L\to\textnormal{GL}_n(\overline{\mathbf{Q}}_{\ell})$, the supercuspidal support $\overline{\pi}^{\textnormal{ss}}(\overline{\varphi}^{\textnormal{ss}})$ coincides with
  \begin{enumerate}
      \item the supercuspidal support of any simple subquotient of the mod $\ell$ reduction of $\pi(\varphi)$, or equivalently with
      \item the supercuspidal support of the mod $\ell$ reduction of $\pi^{\textnormal{ss}}(\varphi^{\textnormal{ss}})$.
  \end{enumerate}
\end{Th}
\begin{proof}
    The existence of $\overline{\pi}^{\textnormal{ss}}$ such that $\overline{\pi}^{\textnormal{ss}}(\overline{\varphi}^{\textnormal{ss}})$ coincides with i) is \cite{Vig01}, Theorem 1.6. More precisely, we twist the correspondence of \textit{loc. cit.} by the $\ell$-integral character $|\det|_L^{\frac{1-n}{2}}$ so that it is compatible with the Tate-normalised local Langlands correspondence, not the unitary one. Finally, i) is the same as ii) by Proposition~\ref{modlred} and \ref{ssLLchar0}.
\end{proof}

\subsubsection{Moduli of Langlands parameters and their group determinants}
In \cite{Hel20}, Helm constructed a moduli scheme of (framed) $n$-dimensional $\ell$-adically continuous representations of $W_L$ on $\mathbf{Z}_{\ell}$-algebras by discretising the tame inertia group. By fixing a lift $\textnormal{Fr}\in W_L$ of the geometric Frobenius and a pro-generator $s$ of the tame inertia $I_L/P_L$, we set $W^0_L\leq W_L$ to be the preimage of $\langle\textnormal{Fr},s\rangle\leq W_L/P_L$ in $W_L$. We equip it with the topology making $P_L\leq W_L^0$ an open subgroup endowed with its natural profinite topology. Let $\{P^e_L\leq P_L\}$, $e\geq 1$, be an exhaustive filtration of $P_L$ by open normal subgroups.

A representation $\varphi:W_L\to \textnormal{GL}_n(R)$ of an $\ell$-adically separated $\mathbf{Z}_{\ell}$-algebra is called $\ell$-adically continuous if, for every integer $m\geq 1$, $\varphi\mod{\ell^mR}$ is continuous. For a general $\mathbf{Z}_{\ell}$-algebra $R$, $\varphi$ is called $\ell$-adically continuous if it factors through some $\ell$-adically continuous representation $\varphi':W_L\to \textnormal{GL}_n(R')$ for an $\ell$-adically separated $\mathbf{Z}_{\ell}$-algebra $R'$.
Any $\ell$-adically continuous representation $\varphi:W_{L}\to \textnormal{GL}_n(R)$ factors over $W_L/P^e_L$ for some integer $e\geq 1$.
Moreover, with this definition, \cite{Hel20} Proposition 4.9 and Proposition 8.2\footnote{See also \cite{DHKM20}, Theorem 1.6.} show that, for any $\mathbf{Z}_{\ell}$-algebra $R$, there are functorial bijections between the sets
\begin{itemize}
    \item $Z^1_n(W_L)(R):=$\{$\ell$-adically continuous representations $W_L\to\textnormal{GL}_n(R)$\},
    \item $Z^1_n(W_L^0)(R):=$\{continuous\footnote{Continuous when the target is equipped with either the discrete or $\ell$-adic topology.} representations $W_L^{0}\to \textnormal{GL}_n(R)$\}, and
    \item $\bigcup_{e\geq 1}Z^1_n(W_L^0/P^e_L)(R):=\bigcup_{e \geq 1} \{\textnormal{representations }W_L^0/P_L^e\to\textnormal{GL}_n(R)\}$.
\end{itemize}
It is further shown that $Z^1_n(W_L^0/P_L^e)$ is represented by an $\ell$-adically separated, reduced and flat $\mathbf{Z}_{\ell}$-algebra $\mathfrak{R}_{L,n}^e$ that is 
locally a complete intersection (\cite{Hel20}, Proposition 4.2, Corollary 4.7, \cite{DHKM20}, Theorem 1.6). Moreover, the transition maps in the colimit $\varinjlim_{e\geq 1}Z^1_n(W_L^0/P^e_L)$ are simply adjoining connected components.

We now analogously consider the functor
\begin{equation}\label{adequateiso}
    X^e_{L,n}:\textnormal{Alg}_{\mathbf{Z}_{\ell}}\to \textnormal{Sets},
\end{equation}
\begin{equation*}
    R\mapsto \{\textnormal{n-dimensional determinants }R[W^0_L/P_L^e]\to R\}
\end{equation*}
and set $X_{L,n}:=\varinjlim_{e\geq 1}X_{L,n}^e$.
\begin{Prop}\label{GIT}
For $e\geq 1$, $X_{L,n}^e$ is represented by a finitely generated $\mathbf{Z}_{\ell}$-algebra $\mathfrak{R}^{\textnormal{ps},e}_{L,n}$. 

Moreover, the natural map $Z^1_n(W^0_L/P^e_L)\to X_{L,n}^e$ sending a representation to its associated determinant induces an algebra homomorphism
\begin{equation}
    \mathfrak{R}^{\textnormal{ps},e}_{L,n}\to (\mathfrak{R}_{L,n}^{e})^{\textnormal{GL}_n}
\end{equation}
with kernel and cokernel finite modules consisting of $\ell$-torsion nilpotent elements. In particular, $\mathfrak{R}^{\textnormal{ps},e}_{L,n}$ is $\ell$-adically separated.
\end{Prop}
\begin{proof}
    The first part is \cite{Che14}, Proposition 1.6 and Proposition 2.38, using that $\mathbf{Z}_{\ell}[W_L^0/P_L^e]$ is a finitely generated $\mathbf{Z}_{\ell}$-algebra.

    The second part follows from \cite{WE18}, Theorem 2.20. Indeed, since the target and source are Noetherian,  $\textnormal{Spec}\big((\mathfrak{R}_{L,n}^e)^{\textnormal{GL}_n}\big)\to\textnormal{Spec}(\mathfrak{R}^{\textnormal{ps},e}_{L,n})$ being an adequate homeomorphism translates to the claim on the map of global sections.
\end{proof}
The following Proposition is not needed in our arguments and is only included to establish some basic properties of $X_{L,n}$.
\begin{Prop}
    The transition map $X^e_{L,n}\to X^{e+1}_{L,n}$ is given by inclusion of a union of connected components. In particular, $X_{L,n}=\cup_{e\geq 1}X_{L,n}^e$ is representable by an increasing union of affine $\mathbf{Z}_{\ell}$-schemes of finite type.
\end{Prop}
\begin{proof}
    Let $H:=P^{e}_L/P^{e+1}_L$ and consider the functor
    \begin{equation*}
Y:\textnormal{Alg}_{\mathbf{Z}_{\ell}}\to \textnormal{Sets},
    \end{equation*}
    \begin{equation*}
        R\mapsto \{\textnormal{n-dimensional determinants }R[H]\to R\}.
    \end{equation*}
    We claim that $Y$ is a finite \'etale $\mathbf{Z}_{\ell}$-scheme from which the proof follows by noting that $X^e_{L,n}=X_{L,n}^{e+1}\times_{Y,f}\textnormal{Spec}(\mathbf{Z}_{\ell})$ where $f:\textnormal{Spec}(\mathbf{Z}_{\ell})\to Y$ corresponds to the trivial determinant.

    Since $H$ is a $p$-group, by Maschke, Wedderburn--Artin and Wedderburn's little theorem, we have an isomorphism $\mathbf{Z}_{\ell}[H]\cong \prod_{i=1,...,k}\textnormal{M}_{d_i}(\mathcal{O}_{E_i})$ for some unramified extensions $E_i/\mathbf{Q}_{\ell}$. Therefore, by \cite{Che14}, Lemma 2.2 (iii) and Lemma 2.15, for some finite \'etale local $\mathbf{Z}_{\ell}$-algebra $\mathcal{O}$, $Y\times_{\mathbf{Z}_{\ell}}\mathcal{O}$ becomes a finite product of copies of $\mathcal{O}$. The claim follows by descent.
\end{proof}

\begin{Prop}\label{ArtinianPoints}
    The restriction of $X_{L,n}$ to the subcategory $\textnormal{Art}_{\mathbf{Z}_{\ell}}\subset \textnormal{Alg}_{\mathbf{Z}_{\ell}}$ of Artinian $\mathbf{Z}_{\ell}$-algebras can be identified with the functor sending $A\in \textnormal{Art}_{\mathbf{Z}_{\ell}}$ to the set
    \begin{equation}\label{galoisA}
        \{\textnormal{continuous $n$-dimensional determinants }A[G_L]\to A\}.
    \end{equation}
\end{Prop}
\begin{proof}
     For the proof, we denote by $X$ the functor sending $A\in \textnormal{Art}_{\mathbf{Z}_{\ell}}$ to \ref{galoisA}. We first give the desired map of functors $\psi: X\to X_{L,n}$. Fix $A\in \textnormal{Art}_{\mathbf{Z}_{\ell}}$ and a continuous $n$-dimensional determinant $D:G_L\to A$. By \cite{Che14}, Lemma 2.33, we see that the induced map on algebras $A[G_L]\to A[G_L]/\textnormal{ker}(D)$ factors through $A[G_L/Q]$ for some compact open normal subgroup $Q\subset G_L$, a finite ring. Therefore, the restriction $D|_{W^0_L}$ must factor through $W_L^0/P^e_L$ for some integer $e\geq 1$. This defines the image of $D$ under $\psi_A:X(A)\to X_{L,n}(A)$.

     To see that $\psi$ is an injection of functors we note that $A[W^0_L]\to A[G_L]\to A[G_L/Q]$ is a surjection.

     To see surjectivity, consider an element $D_e:A[W^0_L/P^e_L]\to A$ in $X_{L,n}(A)$. We recall that it factors through a Cayley--Hamilton representation
     \begin{equation*}
         A[W^0_L/P^e_L]\xrightarrow[]{\rho_{D_e}^{\textnormal{CH}}}A^{\textnormal{CH}}:=A[W^0_L/P^e_L]/\textnormal{CH}(D_e)\xrightarrow{D_e^{\textnormal{CH}}}A.
     \end{equation*}
     Moreover, $A[W^0_L/P^e_L]$ being finitely generated over $A$, $A^{\textnormal{CH}}$ is known to be finite over $A$ (cf. \cite{Che14}, Proposition 2.13) and consequently a finite ring itself. In particular, $(A[W^0_L]\to A[W^0_L/P^e_L])\circ \rho_{D_e}^{CH}$ extends uniquely to an algebra homomorphism $A[G_L]\to A^{\textnormal{CH}}$, $G_L$ being the profinite completion of $W^0_L$.
\end{proof}

\subsubsection{Interpolation of the correspondence}
Finally, we state the theorem of Helm, Helm--Moss on the interpolation of the semisimple local Langlands correspondence.
\begin{Th}[\cite{HM18}, Corollary 7.7, \cite{DHKM24}, Corollary 8.4]
    There is a (necessarily unique) isomorphism of rings
    \begin{equation*}
        \Psi_{L,n}:(\mathfrak{R}_{L,n})^{\textnormal{GL}_n}\xrightarrow{\sim} \mathfrak{Z}_{\textnormal{G},\mathbf{Z}_{\ell}}
    \end{equation*}
    interpolating the semisimple local Langlands correspondence $\pi^{\textnormal{ss}}(-)$ from \ref{ssLLchar0}.
\end{Th}
Let us make precise what "interpolating $\pi^{\textnormal{ss}}(-)$" means. Recall that the $\overline{\mathbf{Q}}_{\ell}$-points $x$ of $\mathfrak{Z}_{\textnormal{G},\mathbf{Z}_{\ell}}$ are in bijection with supercuspidal supports $(\textnormal{M}_x,\pi_x)$ of irreducible representations in $\textnormal{Mod}_{\textnormal{sm}}(\overline{\mathbf{Q}}_{\ell}[\textnormal{G}])$. On the other hand, the $\overline{\mathbf{Q}}_{\ell}$-points $y$ of $\mathfrak{R}_{L,n}^{\textnormal{GL}_n}$ are identified with continuous semisimple representations $\varphi_y:W_L\to \textnormal{GL}_n(\overline{\mathbf{Q}}_{\ell})$. The map $\Psi_{L,n}$ is uniquely determined by the property that for any $x:\mathfrak{Z}_{\textnormal{G},\mathbf{Z}_{\ell}}\to\overline{\mathbf{Q}}_{\ell}$, we have 
\begin{equation*}
    \pi^{\textnormal{ss}}(\varphi_{x\circ\Psi_{L,n}})=\pi_x.
\end{equation*}

For our local-global compatibility results, we will in fact need an interpolation of the local Langlands correspondence with source being the ring $\mathfrak{R}_{L,n}^{\textnormal{ps}}:=\varprojlim_{e\geq 1}\mathfrak{R}_{L,n}^{\textnormal{ps},e}$. Such an interpolation is easily obtained using $\Psi_{L,n}$. Indeed, using Proposition~\ref{GIT} we obtain a natural map of rings $\mathfrak{R}^{\textnormal{ps}}_{L,n}\to (\mathfrak{R}_{L,n})^{\textnormal{GL}_n}$. We then define
\begin{equation*}
    \Phi_{L,n}:\mathfrak{R}^{\textnormal{ps}}_{L,n}\to (\mathfrak{R}_{L,n})^{\textnormal{GL}_n}\xrightarrow{\Psi_{L,n}}\mathfrak{Z}_{\textnormal{G},\mathbf{Z}_{\ell}}.
\end{equation*}

\subsection{The key lemmas}
Let $R$ be a commutative $\mathbf{Z}_{\ell}$-algebra, and let $\textnormal{G}$ be the general linear group $\textnormal{GL}_n(L)$. We collect here some technical lemmas regarding objects in $\textnormal{Mod}_{\textnormal{sm}}(R[\textnormal{G}])$, including Lemma~\ref{KeyLemma}, the "Key Lemma" of this article.

We start by discussing Casselman's lift in our context following \cite{Hel16}. Let $\textnormal{Q}=\textnormal{M}\textnormal{N}\leq \textnormal{G}$ be a parabolic subgroup and denote by $\overline{\textnormal{Q}}=\textnormal{M}\overline{\textnormal{N}}$ its opposite parabolic subgroup. Let $\mathcal{Q}\leq \textnormal{G}$ be a compact open subgroup admitting an Iwahori decomposition $\mathcal{Q}=\overline{\textnormal{N}}^1\textnormal{M}^0\textnormal{N}^0=\textnormal{N}^0\textnormal{M}^0\overline{\textnormal{N}}^1$ with respect to $\textnormal{Q}$. Consider the set of \textnormal{positive} elements
\begin{equation}
    \textnormal{M}^+:=\{m\in \textnormal{M}\mid m\textnormal{N}^0m^{-1}\subset\textnormal{N}^0\textnormal{ and }m^{-1}\overline{\textnormal{N}}^1m\subset \overline{\textnormal{N}}^1\}
\end{equation}
and let $z\in \textnormal{M}$ be a central element that is \textit{strongly positive} in the sense of \cite{ACC23}, \S2.1.9, meaning that for every open compact subgroup $U\leq \textnormal{N}$ (respectively $\overline{U}\leq \overline{\textnormal{N}}$), the set $\{z^nUz^{-n}\}_{n\geq 1}$ (respectively $\{z^{-n}\overline{U}z^n\}_{n\geq 1})$ forms a basis of neighbourhoods of the identity in $\textnormal{N}$ (respectively $\overline{\textnormal{N}}$).\footnote{We note that in the terminology of \cite{Hel16}, $\mathcal{Q}$ becomes decomposed with respect to $\textnormal{Q}$ and strongly positive elements are called strictly positive.}

We then define the map of Hecke algebras
\begin{equation*}
    t:\mathcal{H}(\textnormal{M}^+,\textnormal{M}^0)\otimes_{\mathbf{Z}}R\xrightarrow{}\mathcal{H}(\textnormal{G},\mathcal{Q})\otimes_{\mathbf{Z}}R,
\end{equation*}
\begin{equation*}
    [\textnormal{M}^0m\textnormal{M}^0]\mapsto \delta_{\textnormal{Q}}(m)[\mathcal{Q}m\mathcal{Q}],
\end{equation*}
an injective algebra homomorphism (cf. \cite{ACC23}, Lemma 2.1.12). Given a smooth $R[\textnormal{G}]$-module $\pi$, we can view $\pi^{\mathcal{Q}}$ as a $\mathcal{H}(\textnormal{M}^+,\textnormal{M}^0)\otimes_{\mathbf{Z}}R$-module via $t$. 

We then define $(\pi^{\mathcal{Q}})^{z\textnormal{-inv}}\subset \pi^{\mathcal{Q}}$ to be the maximal $\mathcal{H}(\textnormal{M}^+,\textnormal{M}^0)\otimes_{\mathbf{Z}}R$-submodule on which $z$ acts invertibly. Moreover, we denote by $(\pi^{\mathcal{Q}})[z^{\infty}]\subset \pi^{\mathcal{Q}}$ the submodule of $z$-torsion elements.
\begin{Lemma}\label{Casselmanslift}
    Let $\pi$ be a smooth $R[\textnormal{G}]$-module. Then the natural projection $\textnormal{pr}:\pi\to J_{\textnormal{Q}}(\pi)$ induces a $\mathcal{H}(\textnormal{M}^+,\textnormal{M}^0)$-equivariant map
    \begin{equation*}
        \textnormal{pr}_{\mathcal{Q}}:\pi^{\mathcal{Q}}\to (J_{\textnormal{Q}}(\pi))^{\textnormal{M}^0}
    \end{equation*}
    where we view the source as a $\mathcal{H}(\textnormal{M}^+,\textnormal{M}^0)$-module via the map $t$. Moreover, $\textnormal{pr}_{\mathcal{Q}}$ induces an isomorphism $(\pi^{\mathcal{Q}})\otimes_{R[z]}R[z^{\pm1}]\xrightarrow{\sim} (J_{\textnormal{Q}}(\pi))^{\textnormal{M}^0}$ of $\mathcal{H}(\textnormal{M},\textnormal{M}^0)$-modules.

    If we further assume $R$ to be Noetherian and $\pi$ to be admissible, then we have a $\mathcal{H}(\textnormal{M}^+,\textnormal{M}^0)$-equivariant direct sum decomposition
    \begin{equation*}
        \pi^{\mathcal{Q}}\cong (\pi^{\mathcal{Q}})[z^{\infty}]\oplus (\pi^{\mathcal{Q}})^{z\textnormal{-inv}}
    \end{equation*} and $\textnormal{pr}_{\mathcal{Q}}$ induces a series of $\mathcal{H}(\textnormal{M},\textnormal{M}^0)$-equivariant isomorphisms
    \begin{equation*}
        (\pi^{\mathcal{Q}})^{z\textnormal{-inv}}\xrightarrow{\sim} (\pi^{\mathcal{Q}})\otimes_{R[z]}R[z^{\pm1}]  \xrightarrow{\sim} (J_{\textnormal{Q}}(\pi))^{\textnormal{M}^0}.
    \end{equation*}
\end{Lemma}
\begin{proof}
    We check below that $\textnormal{pr}_{\mathcal{Q}}$ is Hecke equivariant. The rest of the statement follows from \cite{Hel16}, Lemma 11.12, Lemma 11.13 and \cite{DHKM24a}, Corollary 1.5.\footnote{We note that in \textit{loc. cit.} Lemma 11.12 the statement is stated for $R=W(\overline{\mathbf{F}}_{\ell})$ but the argument goes through for any $\mathbf{Z}_{\ell}$-algebra $R$.}

    To check the compatibility of the Hecke actions, pick a vector $v\in \pi^{\mathcal{Q}}$ and an element $m\in \textnormal{M}^+$. We then compute
    \begin{equation*}
        t([\textnormal{M}^0m\textnormal{M}^0])\ast v= [\textnormal{N}^0:m\textnormal{N}^0m^{-1}]^{-1}\cdot [\mathcal{Q}m\mathcal{Q}]\ast v=
    \end{equation*}
    \begin{equation*}
    [\textnormal{N}^0:m\textnormal{N}^0m^{-1}]^{-1}\sum_{q\in \mathcal{Q}/\mathcal{Q}\cap m\mathcal{Q}m^{-1}}qmv=
    \end{equation*}
    \begin{equation}\label{Heckeeq}
    [\textnormal{N}^0:m\textnormal{N}^0m^{-1}]^{-1}\sum_{\Tilde{m}\in \textnormal{M}^{0}/(\textnormal{M}^0\cap m\textnormal{M}^0m^{-1})}\left(\sum_{n\in \textnormal{N}^0/\Tilde{m}m\textnormal{N}^0(\Tilde{m}m)^{-1}}n\Tilde{m}mv\right).
    \end{equation}
    For the last equality we used that the map of sets
    \begin{equation*}
        \{(\Tilde{m},n)\mid \Tilde{m}\in \textnormal{M}^0/m\textnormal{M}^0m^{-1}\cap \textnormal{M}^0 \textnormal{ and }n\in \textnormal{N}^0/\Tilde{m}m\textnormal{N}^0(\Tilde{m}m)^{-1}\}\xrightarrow{\sim}
    \end{equation*}
    \begin{equation*}
         \mathcal{Q}/(\mathcal{Q}\cap m\mathcal{Q}m^{-1})
    \end{equation*}
    \begin{equation*}
        (\Tilde{m},n)\mapsto n\Tilde{m}
    \end{equation*}
    induces a bijection by the existence of Iwahori decomposition of $\mathcal{Q}.$
    
    Since $\textnormal{pr}$ is $\textnormal{Q}$-equivariant (with the trivial action of $\textnormal{N}$ on the target), after applying $\textnormal{pr}_{\mathcal{Q}}$ to \ref{Heckeeq} we get
    \begin{equation*}
        \textnormal{pr}_{\mathcal{Q}}(t([\textnormal{M}^0m\textnormal{M}^0])\ast v)=
    \end{equation*}
    \begin{equation*}
        [\textnormal{N}^0:m\textnormal{N}^0m^{-1}]^{-1}\left(\sum_{\Tilde{m}\in \textnormal{M}^{0}/(\textnormal{M}^0\cap m\textnormal{M}^0m^{-1})}[\textnormal{N}^0:\Tilde{m}m\textnormal{N}^0(\Tilde{m}m)^{-1}]\Tilde{m}m\textnormal{pr}_{\mathcal{Q}}(v) \right)=
    \end{equation*}
    \begin{equation*}
        \sum_{\Tilde{m}\in \textnormal{M}^{0}/(\textnormal{M}^0\cap m\textnormal{M}^0m^{-1})}\Tilde{m}m\textnormal{pr}_{\mathcal{Q}}(v)=m\ast \textnormal{pr}_{\mathcal{Q}}(v)
    \end{equation*}
    where the last equality uses that, for $\Tilde{m}\in \textnormal{M}^0$, 
    \begin{equation*}
[\textnormal{N}^0:\Tilde{m}m\textnormal{N}^0(\Tilde{m}m)^{-1}]=[\Tilde{m}^{-1}\textnormal{N}^0\Tilde{m}:m\textnormal{N}^0m^{-1}]=[\textnormal{N}^0:m\textnormal{N}^0m^{-1}].
    \end{equation*}
\end{proof}

\begin{Lemma}\label{HomBaseChange}
    Let $R'$ be a commutative $R$-algebra, $K\leq \textnormal{G}$ be a compact open subgroup. Then for every $\pi\in \textnormal{Mod}_{\textnormal{sm}}(R[\textnormal{G}])$, we have a canonical isomorphism
    \begin{equation*}
        (\pi\otimes_RR')^K\xrightarrow{\sim} (\pi^K)\otimes_RR'.
    \end{equation*}
\end{Lemma}
\begin{proof}
    This is a special case of \cite{DHKM24}, Lemma A1.
\end{proof}
%\begin{proof}
 %   We can write the $R$-module $R'$ as the cokernel of a map of free $R$-modules. In particular, assuming that $K$ is pro-$p$, the claim follows from the fact that taking $K$-invariants commutes with cokernels and coproducts. For arbitrary $K$ with a pro-$p$ compact open normal subgroup $K'\leq K$, we can write $\pi^K=(\pi^{K'})^{K/K'}$ to reduce the question to commutation of $-\otimes_RR'$ with $(-)^{K/K'}=\Hom_{R[K/K']}(R,-)$. Then it follows easily from the fact that $R$ is finitely presented as an $R[K/K']$-module.
%\end{proof}

\begin{Cor}\label{EquivConditions}
    Let $\pi\in \textnormal{Mod}_{\textnormal{sm}}(\mathbf{Z}_{\ell}[\textnormal{G}])$ and $\mathfrak{m}\leq \mathfrak{Z}_{\textnormal{G}}$ be a maximal ideal. Consider the following conditions.
    \begin{enumerate}
        \item If $\pi_{\mathfrak{m}'}$ is non-trivial for a maximal ideal $\mathfrak{m}'\leq \mathfrak{Z}_{\textnormal{G}}$, then $\mathfrak{m}'=\mathfrak{m}$.
        \item For any compact open subgroup $K\leq \textnormal{G}$, if $(\pi^{K})_{\mathfrak{m}'}$ is non-trivial for a maximal ideal $\mathfrak{m}'\leq \mathfrak{Z}_{\textnormal{G}}$, then $\mathfrak{m}'=\mathfrak{m}$.
        \item The natural map $\pi\to \pi_{\mathfrak{m}}$ is an isomorphism.
        \item For every compact open subgroup $K\leq \textnormal{G}$, the natural map $\pi^{K}\to (\pi^{K})_{\mathfrak{m}}$ is an isomorphism.
        \item If $\pi/\mathfrak{m}'\pi$ is non-trivial for a maximal ideal $\mathfrak{m}'\leq \mathfrak{Z}_{\textnormal{G}}$, then $\mathfrak{m}'=\mathfrak{m}$.
        \item For any compact open subgroup $K\leq \textnormal{G}$, if $(\pi^{\textnormal{K}})/\mathfrak{m}'\pi^{K}$ is non-trivial for a maximal ideal $\mathfrak{m}'\leq \mathfrak{Z}_{\textnormal{G}}$, then $\mathfrak{m}'=\mathfrak{m}$.
    \end{enumerate}
    Then $i)\Leftrightarrow ii)$, $iii)\Leftrightarrow iv)$, and $v)\Leftrightarrow vi)$. If $\pi$ is $\mathbf{Z}_{\ell}$-admissible, then all conditions are equivalent.
\end{Cor}
\begin{proof}
    To prove $i)\Leftrightarrow ii)$ and $iii)\Leftrightarrow iv)$, apply Lemma~\ref{HomBaseChange} with $R=\mathfrak{Z}_{\textnormal{G}}$ and $R'=R_{\mathfrak{m}}$.

    To prove $v)\Leftrightarrow vi)$, apply Lemma~\ref{HomBaseChange} with $R=\mathfrak{Z}_{\textnormal{G}}$ and $R'=R/\mathfrak{m}$.

    Now assume that $\pi$ is $\mathbf{Z}_{\ell}$-admissible. Then, for every compact open subgroup $K\leq \textnormal{G}$, there exist finitely many maximal ideals $\mathfrak{m}_i\leq\mathfrak{Z}_{\textnormal{G}}$ and an isomorphism
    \begin{equation*}
        \pi^{K}\cong \oplus_{i=1}^k(\pi^{K})_{\mathfrak{m}_i}
    \end{equation*}
    of $\mathfrak{Z}_{\textnormal{G}}$-modules. Indeed, this is because the image $\mathfrak{Z}$ of $\mathfrak{Z}_{\textnormal{G}}$ in $\textnormal{End}_{\mathbf{Z}_{\ell}}(\pi^{K})$ is a finite $\mathbf{Z}_{\ell}$-algebra, and so it is a finite direct product of local $\mathbf{Z}_{\ell}$-algebras. In particular, we clearly have $ii)\Leftrightarrow iv)$.

    Finally, $ii)\Leftrightarrow vi)$ follows from the Nakayama lemma.
\end{proof}

\begin{Th}\label{TheMapI}
    Let $R$ be $\mathbf{Z}_{\ell}$-flat. Then there is a unique map
    \begin{equation*}
        I_R:\mathfrak{Z}_{\textnormal{G},R}\to \mathfrak{Z}_{\textnormal{M},R}
    \end{equation*}
    making the diagram
    \begin{equation*}
        \begin{tikzcd}
	{\mathfrak{Z}_{\textnormal{G},R}} & {\textnormal{End}_{R[\textnormal{G}]}(\textnormal{Ind}_{\textnormal{Q}}^{\textnormal{G}}\pi)} \\
	{\mathfrak{Z}_{\textnormal{M},R}} & {\textnormal{End}_{R[\textnormal{M}]}(\pi)}
	\arrow["{t_{\textnormal{G},\textnormal{Ind}_{\textnormal{Q}}^{\textnormal{G}}\pi}}", from=1-1, to=1-2]
	\arrow["{I_R}", from=1-1, to=2-1]
	\arrow["{t_{\textnormal{M},\pi}}", from=2-1, to=2-2]
	\arrow["{\textnormal{Ind}_{\textnormal{Q}}^{\textnormal{G}}}", from=2-2, to=1-2]
\end{tikzcd}
    \end{equation*}
for every $\pi\in \textnormal{Mod}_{\textnormal{sm}}(R[\textnormal{M}])$. When $R=\mathbf{Z}_{\ell}$, we use the abbreviation $I:=I_R$.

In particular, the following diagrams are commutative
\begin{equation}\label{JacquetComm}
    \begin{tikzcd}
	{\mathfrak{Z}_{\textnormal{G},R}} & {\textnormal{End}_{R[\textnormal{G}]}(\pi)} \\
	{\mathfrak{Z}_{\textnormal{M},R}} & {\textnormal{End}_{R[\textnormal{M}]}(J_{\textnormal{Q}}(\pi)),}
	\arrow["{t_{\textnormal{G},\pi}}", from=1-1, to=1-2]
	\arrow["{I_R}", from=1-1, to=2-1]
	\arrow["{J_{\textnormal{Q}}}", from=1-2, to=2-2]
	\arrow["{t_{\textnormal{M},J_{\textnormal{Q}}(\pi)}}", from=2-1, to=2-2]
\end{tikzcd}
\end{equation}
\begin{equation}\label{2ndAdjComm}
    \begin{tikzcd}
	{\mathfrak{Z}_{\textnormal{G},R}} & {\textnormal{End}_{R[\textnormal{G}]}(\pi)} \\
	{\mathfrak{Z}_{\textnormal{M},R}} & {\textnormal{End}_{R[\textnormal{M}]}(\delta_{\textnormal{Q}}J_{\overline{\textnormal{Q}}}(\pi))}
	\arrow["{t_{\textnormal{G},\pi}}", from=1-1, to=1-2]
	\arrow["{I_R}", from=1-1, to=2-1]
	\arrow["{\delta_{\textnormal{Q}}J_{\overline{\textnormal{Q}}}}", from=1-2, to=2-2]
	\arrow["{t_{\textnormal{M},\delta_{\textnormal{Q}}J_{\overline{\textnormal{Q}}}(\pi)}}", from=2-1, to=2-2]
\end{tikzcd}
\end{equation}
for every $\pi \in \textnormal{Mod}_{\textnormal{sm}}(R[\textnormal{G}])$.
\end{Th}
\begin{proof}
    The statement except the commutativity of \ref{JacquetComm} and \ref{2ndAdjComm} is \cite{DHKM24a}, Theorem 4.1.

    To prove the commutativity of \ref{JacquetComm}, let $\pi \in \textnormal{Mod}_{\textnormal{sm}}(R[\textnormal{G}])$ and write
    \begin{equation*}
        A:\textnormal{Hom}_{R[\textnormal{M}]}(J_{\textnormal{Q}}(\pi),J_{\textnormal{Q}}(\pi))\xrightarrow{\sim}\textnormal{Hom}_{R[\textnormal{G}]}(\pi,\textnormal{Ind}_{\textnormal{Q}}^{\textnormal{G}}(J_{\textnormal{Q}}\pi))
    \end{equation*} for the Frobenius reciprocity adjunction. Write $\eta:=A(\textnormal{id}):\pi\to \textnormal{Ind}_{\textnormal{Q}}^{\textnormal{G}}(J_{\textnormal{Q}}(\pi))$ for the unit and let $z\in \mathfrak{Z}_{\textnormal{G},R}$. We then have
    \begin{equation*}
        A(J_{\textnormal{Q}}(t_{\textnormal{G}}(z)))=\eta \circ 
 t_{\textnormal{G}}(z)=
    \end{equation*}
    \begin{equation*}
        t_{\textnormal{G}}(z)\circ \eta= \textnormal{Ind}_{\textnormal{Q}}^{\textnormal{G}}(t_{\textnormal{M}}(I_R(z)))\circ \eta=
    \end{equation*}
    \begin{equation*}
        A(t_{\textnormal{M}}(I_R(z)))
    \end{equation*}
    where for the first and the last equality we use the functoriality of adjunction, for the second equality the defining property of the Bernstein centre and for the third the defining property of $I_R$. We conclude by applying $A^{-1}$ to the obtained equality.

    To prove the commutativity of \ref{2ndAdjComm}, one argues similarly, using that by Bernstein's second adjointness (cf. \cite{Hel16}, Theorem 11.7, \cite{DHKM24a}, Corollary 1.3) $\textnormal{Ind}_{\textnormal{Q}}^{\textnormal{G}}$ is left adjoint to $\delta_{\textnormal{Q}}J_{\overline{\textnormal{Q}}}$.
\end{proof}
\begin{Rem}\label{TheMapIRemark}
    Recall that $R\in\{\overline{\mathbf{F}}_{\ell},\overline{\mathbf{Q}}_{\ell}\}$-points $x:\mathfrak{Z}_{\textnormal{G}}\to R$ are in bijection with supercuspidal supports $(\textnormal{M}_{\textnormal{sc}},\pi_{\textnormal{sc}})$ for $\textnormal{Mod}_{\textnormal{sm}}(R[\textnormal{G}])$. In particular, precomposition by $I:\mathfrak{Z}_{\textnormal{G}}\to\mathfrak{Z}_{\textnormal{M}}$ induces a map $I^{\ast}(-):(\textnormal{M}_{\textnormal{sc}},\pi_{\textnormal{sc}})\mapsto (\widetilde{\textnormal{M}}_{\textnormal{sc}},\widetilde{\pi}_{\textnormal{sc}})$ from supercuspidal supports for $\textnormal{Mod}_{\textnormal{sm}}(R[\textnormal{M}])$ to supercuspidal supports for $\textnormal{Mod}_{\textnormal{sm}}(R[\textnormal{G}])$.

    An easy computation shows that for a supercuspidal support represented by $(\textnormal{M}_{\textnormal{sc}},\pi_{\textnormal{sc}})$ for $\textnormal{Mod}_{\textnormal{sm}}(R[\textnormal{M}])$, the induced supercuspidal support for $\textnormal{Mod}_{\textnormal{sm}}(R[\textnormal{G}])$ is represented by
    \begin{equation*}
(\widetilde{\textnormal{M}}_{\textnormal{sc}},\widetilde{\pi}_{\textnormal{sc}})=(\textnormal{M}_{\textnormal{sc}},\delta_{\textnormal{Q}}^{-1/2}\pi_{\textnormal{sc}}).
    \end{equation*}

    In particular, $I^{\ast}$ sends a supercuspidal support represented by $(\textnormal{M}_{\textnormal{sc}},\pi_{\textnormal{sc}})$ to an $\ell$-integral supercuspidal if and only if $(\textnormal{M}_{\textnormal{sc}},\pi_{\textnormal{sc}})$ itself was $\ell$-integral. In other words, $x:\mathfrak{Z}_{\textnormal{M}}\to \overline{\mathbf{Q}}_{\ell}$ factors through $\overline{\mathbf{Z}}_{\ell}$ if and only if $I\circ x$ does. 
\end{Rem}

\begin{Lemma}\label{LocalEisensteinFunctoriality}
    Consider a partition $n=n_1+...+n_k$ and let $\textnormal{M}=M_{(n_1,...,n_k)}(L)$ be the corresponding standard parabolic subgroup. Define a map
    \begin{equation*}
        I^{\textnormal{Gal}}:\mathfrak{R}^{\textnormal{ps}}_{L,n}\to \mathfrak{R}^{\textnormal{ps}}_{L,n_1}\otimes_{\mathbf{Z}_{\ell}}...\otimes_{\mathbf{Z}_{\ell}}\mathfrak{R}^{\textnormal{ps}}_{L,n_k}
    \end{equation*}
    of $\mathbf{Z}_{\ell}$-algebras by sending an $R$-point $(D_1,...,D_k)$ of the target to the $R$-point
    \begin{equation*}
        D_1(-(n_2+...+n_k))\oplus...\oplus D_{k-1}(-n_k)\oplus D_k
    \end{equation*}
    of the source for $R\in \textnormal{Alg}_{\mathbf{Z}_{\ell}}$. The following diagram is then commutative
    \begin{equation*}
        \begin{tikzcd}
	{\mathfrak{R}^{\textnormal{ps}}_{L,n}} && {\mathfrak{Z}_{\textnormal{G}}} \\
	{\mathfrak{R}^{\textnormal{ps}}_{L,n_1}\otimes_{\mathbf{Z}_{\ell}}...\otimes_{\mathbf{Z}_{\ell}}\mathfrak{R}^{\textnormal{ps}}_{L,n_k}} && {\mathfrak{Z}_{\textnormal{M}}.}
	\arrow["{\Phi_{L,n}}", from=1-1, to=1-3]
	\arrow["{I^{\textnormal{Gal}}}", from=1-1, to=2-1]
	\arrow["I", from=1-3, to=2-3]
	\arrow["{(\Phi_{L,n_1},...,\Phi_{L,n_k})}", from=2-1, to=2-3]
\end{tikzcd}
    \end{equation*}
\end{Lemma}
\begin{proof}
    We claim that it suffices to check commutativity after postcomposing the diagram with the map
    \begin{equation*}
        \gamma:\mathfrak{Z}_{\textnormal{M}}\to \prod_{x:\mathfrak{Z}_{\textnormal{M}}\to \overline{\mathbf{Q}}_{\ell}}\overline{\mathbf{Q}}_{\ell}
    \end{equation*}
    induced by specialising at all of the $\overline{\mathbf{Q}}_{\ell}$-points of the centre for $\textnormal{M}$. To see that this indeed suffices, we prove that $\gamma$ is injective. It is enough to check that, for every inertial equivalence class $[\textnormal{M}_{\textnormal{sc}},\pi_{\textnormal{sc}}]$ for $\textnormal{Mod}_{\textnormal{sm}}(\mathbf{Z}_{\ell}[\textnormal{M}])$, we have an inclusion $\gamma_{[\textnormal{M}_{\textnormal{sc}},\pi_{\textnormal{sc}}]}:\mathfrak{Z}_{[\textnormal{M}_{\textnormal{sc}},\pi_{\textnormal{sc}}]}\to \prod_{x}\overline{\mathbf{Q}}_{\ell}$ induced by $\overline{\mathbf{Q}}_{\ell}$-specialisations. By \cite{Hel16}, Theorem 12.8, the source is an $\ell$-torsion free, reduced and finite type $\mathbf{Z}_{\ell}$-algebra. In particular, it suffices to check that $\gamma_{[\textnormal{M}_{\textnormal{sc}},\pi_{\textnormal{sc}}]}[1/\ell]$ is injective. This holds for any finite type reduced $\mathbf{Q}_{\ell}$-algebra.

    Finally, an easy computation, using Remark~\ref{TheMapIRemark}, shows that the diagram is commutative after specialising at $\overline{\mathbf{Q}}_{\ell}$-points.
\end{proof}

The following two corollaries of Theorem \ref{TheMapI} will allow us to control the interaction of localisations along the Bernstein centre with parabolic induction and the Jacquet functor.

\begin{Cor}\label{ParabolicLemma}
    Let $\widetilde{\mathfrak{m}}$ be a maximal ideal of $\mathfrak{Z}_{\textnormal{G},\mathbf{Z}_{\ell}}$ and $\pi$ be a  smooth $\mathbf{Z}_{\ell}[\textnormal{M}]$-module. Then the following are equivalent:
    \begin{enumerate}
        \item The localisation $\pi_{\mathfrak{m}}$ is a non-trivial $\mathbf{Z}_{\ell}[\textnormal{M}]$-module for some maximal ideal $\mathfrak{m}$ of $\mathfrak{Z}_{\textnormal{M}}$ satisfying $I^{\ast}(\mathfrak{m})=\widetilde{\mathfrak{m}}$.
        \item The localisation $(\textnormal{Ind}_{\textnormal{Q}}^{\textnormal{G}}\pi)_{\widetilde{\mathfrak{m}}}$ is a non-trivial $\mathbf{Z}_{\ell}[\textnormal{G}]$-module.
    \end{enumerate}
\end{Cor}
\begin{proof}
    $i)\Rightarrow ii):$ Assume that $\pi_{\mathfrak{m}}$ is non-trivial for some maximal ideal. In particular, $\textnormal{Ind}_{\textnormal{Q}}^{\textnormal{G}}(\pi_{\mathfrak{m}})$ is non-trivial. By the universal property of localisation, the \textit{non-zero} natural map $\textnormal{Ind}_{\textnormal{Q}}^{\textnormal{G}}\pi\to \textnormal{Ind}_{\textnormal{Q}}^{\textnormal{G}}(\pi_{\mathfrak{m}})$ factors through $(\textnormal{Ind}_{\textnormal{Q}}^{\textnormal{G}}\pi)_{\widetilde{\mathfrak{m}}}$, showing that the latter must be non-trivial as well.

    $ii)\Rightarrow i):$ Assume that $\left(\textnormal{Ind}_{\textnormal{Q}}^{\textnormal{G}}\pi\right)_{\widetilde{\mathfrak{m}}}$ is non-trivial. We first note that $\left(\textnormal{Ind}_{\textnormal{Q}}^{\textnormal{G}}\pi\right)_{\widetilde{\mathfrak{m}}}\cong\textnormal{Ind}_{\textnormal{Q}}^{\textnormal{G}}(\pi_{\widetilde{\mathfrak{m}}})$. To see this, we show that the latter satisfies the universal property for the localisation $(\mathfrak{Z}_{\textnormal{G}})_{\widetilde{\mathfrak{m}}}[\textnormal{G}]\otimes_{\mathfrak{Z}_{\textnormal{G}}[\textnormal{G}]}-:\textnormal{Mod}(\mathfrak{Z}_{\textnormal{G}}[\textnormal{G}])\to\textnormal{Mod}((\mathfrak{Z}_{\textnormal{G}})_{\widetilde{\mathfrak{m}}}[\textnormal{G}]) $. To simplify notation, we set $\mathfrak{Z}:=\mathfrak{Z}_{\textnormal{G}}$. Pick arbitrary $N\in\textnormal{Mod}(\mathfrak{Z}_{\widetilde{\mathfrak{m}}}[\textnormal{G}])$ and compute
    \begin{equation*}\label{adjunctions}
        \Hom_{\mathfrak{Z}[\textnormal{G}]}(\textnormal{Ind}_{\textnormal{Q}}^{\textnormal{G}}\pi,N)\cong \Hom_{\mathfrak{Z}[\textnormal{G}]}(\textnormal{Ind}_{\textnormal{Q}}^{\textnormal{G}}\pi,N^{\textnormal{sm}})\cong
    \end{equation*}
    \begin{equation*}
        \Hom_{\mathfrak{Z}[M]}(\pi,\delta_{\textnormal{Q}}J_{\overline{\textnormal{Q}}}(N^{\textnormal{sm}}))\cong  \Hom_{\mathfrak{Z}_{\widetilde{\mathfrak{m}}}[M]}(\pi_{\widetilde{\mathfrak{m}}},\delta_{\textnormal{Q}}J_{\overline{\textnormal{Q}}}(N^{\textnormal{sm}}))
    \end{equation*}
    \begin{equation*}
        \Hom_{\mathfrak{Z}_{\widetilde{\mathfrak{m}}}[\textnormal{G}]}(\textnormal{Ind}_{\textnormal{Q}}^{\textnormal{G}}(\pi_{\widetilde{\mathfrak{m}}}),N^{\textnormal{sm}})\cong \Hom_{\mathfrak{Z}_{\widetilde{\mathfrak{m}}}[\textnormal{G}]}(\textnormal{Ind}_{\textnormal{Q}}^{\textnormal{G}}(\pi_{\widetilde{\mathfrak{m}}}),N).
    \end{equation*}
    In the first and last identification we used the functor of taking smooth vectors $(-)^{\textnormal{sm}}$, right adjoint to the inclusion $\textnormal{Mod}_{\textnormal{sm}}(\mathfrak{Z}[\textnormal{G}])\hookrightarrow\textnormal{Mod}(\mathfrak{Z}[\textnormal{G}])$. The second and fourth identification used Bernstein's second adjointness and the third used the universal property of localisation. We see that, in particular, $\pi_{\widetilde{\mathfrak{m}}}$ is a non-trivial $(\mathfrak{Z}_{\textnormal{M}})_{\widetilde{\mathfrak{m}}}$-module. 
    
    We claim that the maximal ideals of $(\mathfrak{Z}_{\textnormal{M}})_{\widetilde{\mathfrak{m}}}:=(\mathfrak{Z}_{\textnormal{G}})_{\widetilde{\mathfrak{m}}}\otimes_{\mathfrak{Z}_{\textnormal{G}}} \mathfrak{Z}_{\textnormal{M}}$ are exactly the maximal ideals $\mathfrak{m}$ of $\mathfrak{Z}_{\textnormal{M}}$ such that $I^{\ast}(\mathfrak{m})=\widetilde{\mathfrak{m}}$. This follows from Remark \ref{TheMapIRemark}. To see this, consider a closed point $x:(\mathfrak{Z}_{\textnormal{M}})_{\widetilde{\mathfrak{m}}}=(\mathfrak{Z}_{\textnormal{G}})_{\widetilde{\mathfrak{m}}}\otimes_{\mathfrak{Z}_{\textnormal{G}}}\mathfrak{Z}_{\textnormal{M}}\to K(x)$ and note that $K(x)$ is a finite extension of $\mathbf{F}_{\ell}$ or $\mathbf{Q}_{\ell}$. If the induced point $x_{\textnormal{G}}$ of $(\mathfrak{Z}_{\textnormal{G}})_{\widetilde{\mathfrak{m}}}$ is the closed point corresponding to $\widetilde{\mathfrak{m}}$, then the induced point $x_{\textnormal{M}}$ of $\mathfrak{Z}_{\textnormal{M}}$ must also be closed, and $I^{\ast}(\mathfrak{m}_{x_{\textnormal{M}}})=\widetilde{\mathfrak{m}}$. Otherwise, $k(\widetilde{\mathfrak{m}})$ must be of characteristic $\ell$ and $x_{\textnormal{G}}$ has to correspond to an $\ell$-integral supercuspidal support. Therefore, $x_{\textnormal{M}}$ also gives rise to an $\ell$-integral supercuspidal support. But then the mod $\ell$ reduction of said supercuspidal support induces a specialisation of $x_{\textnormal{M}}$ showing that $x$ cannot be closed, a contradiction. 
    
    From the claim it follows that $(\pi_{\widetilde{\mathfrak{m}}})_{\mathfrak{m}}\cong \pi_{\mathfrak{m}}$ cannot be trivial for each of these maximal ideals as that would force $\pi_{\widetilde{\mathfrak{m}}}$ to be trivial as well (cf. \cite[\href{https://stacks.math.columbia.edu/tag/00HN}{Lemma 00HN}]{stacks-project}).
\end{proof}
In a similar vein, we have the following.
\begin{Cor}\label{JacquetLemma}
    Let $\widetilde{\mathfrak{m}}$ be a maximal ideal of $\mathfrak{Z}_{\textnormal{G}}$ and $\pi$ be a smooth $\mathbf{Z}_{\ell}[\textnormal{G}]$-module. Then the following are equivalent:
    \begin{enumerate}
        \item The localisation $J_{\textnormal{Q}}(\pi)_{\mathfrak{m}}$ is a non-trivial $\mathbf{Z}_{\ell}[\textnormal{M}]$-module for some maximal ideal $\mathfrak{m}$ of $\mathfrak{Z}_{\textnormal{M}}$ satisfying $I^{\ast}(\mathfrak{m})=\widetilde{\mathfrak{m}}$.
        \item The localisation $J_{\textnormal{Q}}(\pi_{\widetilde{\mathfrak{m}}})$ is a non-trivial $\mathbf{Z}_{\ell}[\textnormal{M}]$-module.
    \end{enumerate}
\end{Cor}
\begin{proof}
    $i)\Rightarrow ii):$ The localisation $J_{\textnormal{Q}}(\pi)_{\mathfrak{m}}$ being non-trivial implies that $J_{\textnormal{Q}}(\pi)_{\widetilde{\mathfrak{m}}}$ is non-trivial. Analogously to the argument \ref{ParabolicLemma}, using Frobenius reciprocity in place of Bernstein's second adjointness, one shows that $J_{\textnormal{Q}}(\pi)_{\widetilde{\mathfrak{m}}}\cong J_{\textnormal{Q}}(\pi_{\widetilde{\mathfrak{m}}})$.

    $ii)\Rightarrow i):$ By assumption, $J_{\textnormal{Q}}(\pi_{\widetilde{\mathfrak{m}}})\cong J_{\textnormal{Q}}(\pi)_{\widetilde{\mathfrak{m}}}$ is a non-trivial $(\mathfrak{Z}_{M})_{\widetilde{\mathfrak{m}}}$-module. As we saw in the proof of Corollary~\ref{ParabolicLemma}, $(\mathfrak{Z}_{M})_{\widetilde{\mathfrak{m}}}$ is a semilocal ring with maximal ideals $\mathfrak{m}$ such that $I^{\ast}(\mathfrak{m})=\widetilde{\mathfrak{m}}$. We then conclude just as in the proof of Corollary~\ref{ParabolicLemma}.
\end{proof}

For the rest of the section, let $\textnormal{G}:=\textnormal{GL}_{2n}(L)$ and $\textnormal{Q}:=P_{(n,n)}(L)$ be the standard parabolic subgroup of block upper triangular matrices with Levi factor $\textnormal{M}:=M_{(n,n)}(L)=\textnormal{GL}_n(L)\times \textnormal{GL}_n(L)$. Let $\mathfrak{m}\leq \mathfrak{Z}_{\textnormal{M}}=\mathfrak{Z}_{\textnormal{M},\mathbf{Z}_{\ell}}$ be a maximal ideal with residue field $k(\mathfrak{m})$. Let $(\textnormal{M}_{\textnormal{sc}},\pi_{\textnormal{sc}})$ be a standard representative of $\mathfrak{m}$ and recall that by Remark~\ref{StandardRemark} it is of the form $(\pi_{1,1}\otimes...\otimes\pi_{1,k_1})\otimes(\pi_{2,1}\otimes...\otimes\pi_{2,k_2})$ where $\pi_{i,j}$, $i=1,2$, $1\leq j\leq k_i$ are supercuspidal smooth $\overline{k(\mathfrak{m})}$-representations of auxiliary general linear groups.
\begin{Def}
We say that a maximal ideal $\mathfrak{m}\leq\mathfrak{Z}_{\textnormal{M}}$ is \textit{(n,n)-generic} if for a standard representative $$(\pi_{1,1}\otimes...\otimes\pi_{1,k_1})\otimes(\pi_{2,1}\otimes...\otimes\pi_{2,k_2})$$ of the corresponding supercuspidal support the following is satisfied. For every choice of integers $1\leq i_1\leq k_1$, $1\leq i_2\leq k_2$ and $m\in \mathbf{Z}$ we have
\begin{equation*}
    \pi_{1,i_1}\ncong \pi_{2,i_2}|\det|^m_L.
\end{equation*}
\end{Def}
\begin{Lemma}\label{GeomLemmaLemma}
    Let $\pi$ be a smooth $\mathbf{Z}_{\ell}[\textnormal{M}]$-module such that $\pi_{\mathfrak{m}}$ is non-trivial only for a fixed $(n,n)$-generic maximal ideal $\mathfrak{m}\leq \mathfrak{Z}_{\textnormal{M}}$.

    We then have an isomorphism
    \begin{equation*}
        (J_{\textnormal{Q}}(\textnormal{Ind}_{\textnormal{Q}}^{\textnormal{G}}\pi))_{\mathfrak{m}}\xrightarrow{\sim}\pi_{\mathfrak{m}}
    \end{equation*}
    in $\textnormal{Mod}_{\textnormal{sm}}(\mathbf{Z}_{\ell}[\textnormal{M}])$, natural in $\pi$.
\end{Lemma}
\begin{proof}
    Recall that the integral (non-normalised) geometric lemma (cf. \cite{BZ77} Lemma 2.12, \cite{Cas95}, Proposition 6.33, \cite{MT23}, Proposition A.1)\footnote{For the integral version see \cite{MT23}.} asserts that $J_{\textnormal{Q}}(\textnormal{Ind}_{\textnormal{Q}}^{\textnormal{G}}\pi)$ admits a $\mathbf{Z}_{\ell}[\textnormal{M}]$-equivariant filtration with subquotients $(I_w)_{w\in \prescript{\textnormal{M}}{}{}W^{\textnormal{M}}}$ of the form
    \begin{equation*}
        I_w\cong \textnormal{Ind}_{\textnormal{Q}_w^M}^{\textnormal{M}}\left(\delta_{\textnormal{Q}_w}^{1/2}w\circ\left(\delta_{\textnormal{Q}_{w^{-1}}}^{-1/2}J_{\textnormal{Q}_{w^{-1}}^{\textnormal{M}}}(\pi)\right) \right).
    \end{equation*}

    Here $\textnormal{Q}_w^{\textnormal{M}}=\textnormal{M}_{w}\textnormal{N}_w$ (respectively $\textnormal{Q}_{w^{-1}}^{\textnormal{M}}=\textnormal{M}_{w^{-1}}\textnormal{N}_{w^{-1}}$) is the standard parabolic subgroup of $\textnormal{M}$ with standard Levi quotient $\textnormal{M}_{w}:=\textnormal{M}\cap w(\textnormal{M})$ (respectively $w^{-1}(\textnormal{M})\cap \textnormal{M}$). Moreover, $\textnormal{Q}_w:=\textnormal{Q}_{w}^{\textnormal{M}}\textnormal{N}$ and $\textnormal{Q}_{w^{-1}}:=\textnormal{Q}_{w^{-1}}^{\textnormal{M}}\textnormal{N}$, respectively. Furthermore, the filtration can be arranged so that the first graded piece is $I_1\cong \pi$ (with the surjection simply given by sending $f:\textnormal{G}\to \pi$ to $f(1)$). In particular, it suffices to see that $(I_w)_{\mathfrak{m}}$ is non-zero only if $w=1$.

    Therefore, consider $w\in \prescript{M}{}{}W^M$ and assume that $(I_w)_{\mathfrak{m}}$ is non-trivial. By Corollary~\ref{ParabolicLemma}, there is a maximal ideal $\mathfrak{m}_{w}\leq \mathfrak{Z}_{\textnormal{M}_w}$ such that
    \begin{equation*}
        \left(\delta_{\textnormal{Q}_w}^{1/2}w\circ\left(\delta_{\textnormal{Q}_{w^{-1}}}^{-1/2}J_{\textnormal{Q}_{w^{-1}}^{\textnormal{M}}}(\pi)\right) \right)_{\mathfrak{m}_w}\neq 0
    \end{equation*}
   and $I^{\ast}(\mathfrak{m}_w)=\mathfrak{m}$. In particular, if we write $(\textnormal{M}_{\textnormal{sc},w},\pi_{\textnormal{sc},w})$ for a choice of standard representative for $\mathfrak{m}_w$, then $(\textnormal{M}_{\textnormal{sc}},\pi_{\textnormal{sc}}):=(\textnormal{M}_{\textnormal{sc},w},\delta_{\textnormal{Q}_w^{\textnormal{M}}}^{-1/2}\pi_{\textnormal{sc},w})$ is a standard representative for $\mathfrak{m}$ (where we view the group now as a Levi subgroup in $\textnormal{M}$). 

   Set $\textnormal{M}_{\textnormal{sc},w^{-1}}:=w^{-1}\textnormal{M}_{\textnormal{sc}}w=w^{-1}\textnormal{M}_{\textnormal{sc},w}w\leq \textnormal{M}_{w^{-1}}$ and
   \begin{equation*}
\pi_{\textnormal{sc},w^{-1}}:=\delta_{\textnormal{Q}_{w^{-1}}}^{1/2}(w^{-1})^{\ast}\left(\delta_{\textnormal{Q}_w^{\textnormal{M}}}^{1/2}\delta_{\textnormal{Q}_{w}}^{-1/2}\pi_{\textnormal{sc}}\right)=\left(\delta_{\textnormal{Q}_{w^{-1}}^{\textnormal{M}}}^{1/2}\left(\delta_{\textnormal{Q}}^{1/2}(w^{-1})^{\ast}(\delta_{\textnormal{Q}}^{-1/2})\right)\right)(w^{-1})^{\ast}(\pi_{\textnormal{sc}})
   \end{equation*}
   and write $\mathfrak{m}_{w^{-1}}\leq \mathfrak{Z}_{\textnormal{M}_{w^{-1}}}$ for the maximal ideal corresponding to $(\textnormal{M}_{w^{-1}},\pi_{\textnormal{sc},w^{-1}})$. Then we have that $J_{\textnormal{Q}_{w^{-1}}^{\textnormal{M}}}(\pi)_{\mathfrak{m}_{w^{-1}}}$ is non-trivial. 
   
   Therefore, by Corollary~\ref{JacquetLemma}, $\pi_{\widetilde{\mathfrak{m}}}$ is non-trivial for $\widetilde{\mathfrak{m}}:=I^{\ast}(\mathfrak{m}_{w^{-1}})$. We see that $\widetilde{\mathfrak{m}}$ corresponds to the supercuspidal support represented by $(\widetilde{\textnormal{M}}_{\textnormal{sc}},\widetilde{\pi}_{\textnormal{sc}})$ where $\widetilde{\textnormal{M}}_{\textnormal{sc}}=\textnormal{M}_{\textnormal{sc},w^{-1}}\leq \textnormal{M}$ and
   \begin{equation}\label{TildeEquation}
       \widetilde{\pi}_{\textnormal{sc}}=\delta_{\textnormal{Q}_{w^{-1}}^{\textnormal{M}}}^{-1/2}\pi_{\textnormal{sc},w^{-1}}=\left(\delta_{\textnormal{Q}}^{1/2}(w^{-1})^{\ast}(\delta_{\textnormal{Q}}^{-1/2}) \right)(w^{-1})^{\ast}(\pi_{\textnormal{sc}}).
   \end{equation}
On the other hand, $\widetilde{\mathfrak{m}}$ must coincide with $\mathfrak{m}$ by assumption, meaning that $(\textnormal{M}_{\textnormal{sc}},\pi_{\textnormal{sc}})$ must be $\textnormal{M}$-conjugate to $(\widetilde{\textnormal{M}}_{\textnormal{sc}},\widetilde{\pi}_{\textnormal{sc}})$.

To make this more explicit, write
\begin{equation*}
    \pi_{\textnormal{sc}}=(\pi_{1,1}\otimes...\otimes \pi_{1,k_{1}})\otimes(\pi_{2,1}\otimes...\otimes\pi_{2,k_{2}})\textnormal{, and}
\end{equation*}
\begin{equation*}
    \widetilde{\pi}_{\textnormal{sc}}=(\widetilde{\pi}_{1,1}\otimes...\otimes\widetilde{\pi}_{1,l_1})\otimes (\widetilde{\pi}_{2,1}\otimes...\otimes\widetilde{\pi}_{2,l_2}).
\end{equation*}
Then the two representatives being $\textnormal{M}$-conjugate exactly means that we have equalities
\begin{equation}\label{FirstMultisetEquality}
    \{\pi_{i,1},...,\pi_{i,k_i}\}=\{\widetilde{\pi}_{i,1},...,\widetilde{\pi}_{i,l_i}\}
\end{equation}
of multisets of supercuspidal smooth $\overline{k(\mathfrak{m})}$-representations of auxiliary $\textnormal{GL}_m(L)$'s for $i=1,2$. We claim that $\mathfrak{m}$ being $(n,n)$-generic implies that this only happens for $w= 1$.

In order to verify the claim, we spell out the supercuspidal representations $\widetilde{\pi}_{i,j}$'s in terms of the $\pi_{i,j}$'s. To do so, write
\begin{equation*}
    \textnormal{M}_{\textnormal{sc},w^{-1}}=(\textnormal{GL}_{m_{1,1}}(L)\times...\times\textnormal{GL}_{m_{1,l_1}}(L))\times (\textnormal{GL}_{m_{2,1}}(L)\times...\textnormal{GL}_{m_{2,l_2}}(L))\leq
\end{equation*}
\begin{equation*}
 \textnormal{M}=\textnormal{GL}_{n}(L)\times \textnormal{GL}_n(L).
\end{equation*}
Then one computes that, for $i=1,2$ and $1\leq j\leq l_i$, the restriction to $G_{i,j}:=\textnormal{GL}_{m_{i,j}}(L)$ of the character
\begin{equation*}
\left(\delta_{\textnormal{Q}}^{1/2}(w^{-1})^{\ast}(\delta_{\textnormal{Q}}^{-1/2}) \right):\textnormal{M}_{\textnormal{sc},w^{-1}}\to (\mathbf{Z}_{\ell}[q^{1/2}])^{\times}
\end{equation*}
from the RHS of \ref{TildeEquation} is in fact an integral power of $|\det_{G_{i,j}}|_{L}$. To see this, one uses the formula for the modulus character from Remark~\ref{ModulusRemark}. In particular, if we write
\begin{equation*}
    \left(\textnormal{M}_{\textnormal{sc},w^{-1}},(w^{-1})^{\ast}(\pi_{\textnormal{sc}})\right)=
\end{equation*}
\begin{equation*}
    \left((G_{1,1}\times...\times G_{1,l_1})\times (G_{2,1}\times...\times G_{2,l_2}),(\pi_{1,1}'\otimes...\otimes\pi_{1,l_1}')\otimes(\pi_{2,1}'\otimes...\otimes\pi_{2,l_2}')\right),
\end{equation*}
we see that, for arbitrary $i=1,2$, $1\leq j\leq l_i$, we have
\begin{equation}\label{determinantdifference}
\widetilde{\pi}_{i,j}\cong\pi_{i,j}'|\det|_{L}^{q_{i,j}}
\end{equation} for some integer $q_{i,j}$.

Finally, to spell out $(w^{-1})^{\ast}\pi_{\textnormal{sc}}$, note that conjugation by $w^{-1}$ on elements of $\textnormal{G}$ restricts to an isomorphism $\textnormal{M}_{\textnormal{sc},w}\xrightarrow{\sim}\textnormal{M}_{\textnormal{sc},w^{-1}}$. In particular, it determines a bijection of sets
\begin{equation*}
    \sigma_w:\{(1,1),...,(1,l_1),(2,1),...,(2,l_2)\}\xrightarrow{\sim}\{(1,1),...,(1,k_1),(2,1),...,(2,k_2)\}
\end{equation*} 
such that
\begin{equation*}
    \pi_{i,j}'=\pi_{\sigma_w(i,j)}
\end{equation*}
for every $i=1,2$, $1\leq j\leq l_i$. Using that $\prescript{\textnormal{M}}{}{}W^{\textnormal{M}}$ is the set of minimal length representatives for $W_{\textnormal{M}}\setminus W_{\textnormal{G}}/W_{\textnormal{M}}$, we see that if $\sigma_w$ restricts to bijections $\{(i,1),...,(i,l_i)\}\xrightarrow{\sim}\{(i,1),...,(i,k_i)\}$ for $i=1,2$ then $w=1$. Indeed, then there obviously is an element $w_{\textnormal{M}}\in W_{\textnormal{M}}\cong S_n\times S_n$ with $w_{\textnormal{M}}w=1$ and $1$ has the smallest length in $W_{\textnormal{G}}$.

In particular, if $w\neq 1$, then we can pick $1\leq j\leq l_1$ such that $\sigma_w(1,j)=(2,j')$ for some $1\leq j'\leq k_2$. This means that $\pi'_{1,j}=\pi_{2,j'}.$
However, by \ref{FirstMultisetEquality} and \ref{determinantdifference}, $\pi'_{1,j}\cong \widetilde{\pi}_{1,j} |\det|_L^{-q_{1,j}}\cong \pi_{1,j''}|\det|_L^{-q_{1,j}}$ for some $1\leq j''\leq k_1$. This contradicts the $(n,n)$-genericity of $\mathfrak{m}$. Therefore, $w$ can only be $1$ as claimed.
\end{proof}

Finally, we are ready to prove the key lemma of the article. Let $\phi_1,\phi_2:G_L\to \textnormal{GL}_n(\overline{\mathbf{F}}_{\ell})$ be semisimple mod $\ell$ Langlands parameters and $\chi:G_L\to \overline{\mathbf{F}}_{\ell}^{\times}$ be an unramified character. Write $\mathfrak{m}_{\chi}\leq \mathfrak{Z}_{\textnormal{M}}$ for the maximal ideal corresponding to the supercuspidal support $\pi^{\textnormal{ss}}(\phi_1\otimes \chi)\otimes \pi^{\textnormal{ss}}(\phi_2|\det|^{n})$ for $\textnormal{Mod}_{\textnormal{sm}}(\overline{\mathbf{F}}_{\ell}[\textnormal{M}])$.

\begin{Lemma}\label{KeyLemma}
    The following is satisfied for all but finitely many \textit{unramified} characters
    \begin{equation*}
        \chi:G_L\to \overline{\mathbf{F}}_{\ell}^{\times}.
    \end{equation*}
    
    For every continuous $\ell$-adic Galois representation
    \begin{equation*}
        \rho:G_L\to \textnormal{GL}_{2n}(\overline{\mathbf{Q}}_{\ell})
    \end{equation*}
    satisfying
    \begin{equation*}
        \overline{\rho}^{\textnormal{ss}}\cong (\phi_1\otimes \chi)\oplus \phi_2
    \end{equation*}
    we have a direct sum decomposition
    \begin{equation*}
        \rho\cong \rho_1\oplus \rho_2
    \end{equation*}
    with $\overline{\rho}_1^{\textnormal{ss}}\cong \phi_1\otimes\chi$, and $\overline{\rho}_2^{\textnormal{ss}}\cong \phi_2$.

    In particular, there is an isomorphism
    \begin{equation*}
        J_{\textnormal{Q}}(\pi(\rho))_{\mathfrak{m}_{\chi}}\cong \pi(\rho_1)\otimes (\pi(\rho_2)|\det|_L^{n})
    \end{equation*}
of smooth $\overline{\mathbf{Q}}_{\ell}[\textnormal{M}]$-modules.
\end{Lemma}
We will use in the proof the following simple observation.
\begin{Lemma}\label{WDlemma}
    Let $r:W_{L}\to \textnormal{GL}_m(\overline{\mathbf{Q}}_{\ell})$ be a representation such that $r(I_{L})$ is finite. Write $\tau:=r|_{I_{L}}$. Set $d_{\tau}:=|\tau(I_{L})|$, and $m_{\tau}:=(d_{\tau}^{d_{\tau}})!$. Then, for any lift $\varphi\in W_L$ of the Frobenius, the element $r(\varphi^{m_{\tau}})$ is central in $r(W_{L})$.
\end{Lemma}
\begin{proof}
    Denote by $\overline{\varphi}$ the matrix $r(\varphi)$, and set $t_1,...,t_{d_{\tau}}$ to be the list of all elements of $r(I_{L})$. Then $\overline{\varphi}t_1,...,\overline{\varphi}t_{d_{\tau}}$ is the list of the images of all Frobenius lifts in $r(W_{L})$.

    Moreover, there is a permutation $\sigma\in S_{d_{\tau}}$ such that
    \begin{equation*}
\overline{\varphi}t_i=t_{\sigma(i)}\overline{\varphi}
    \end{equation*}
    for every $1\leq i \leq d_{\tau}$. Therefore, for every $1\leq i \leq d_{\tau}$, there is $1\leq k_i\leq d_{\tau}$ such that
    \begin{equation*}
\overline{\varphi}^{k_i}t_i=t_i\overline{\varphi}^{k_i}.\end{equation*}
In particular, $\overline{\varphi}^{k_1\cdot...\cdot k_{d_{\tau}}}$ commutes with any element of $r(W_{L})$. Since $1\leq k_1\cdot...\cdot k_{d_{\tau}}\leq d_{\tau}^{d_{\tau}}$, it divides $m_{\tau}$. In particular, $\overline{\varphi}^{m_{\tau}}$ commutes with any element of $r(W_L)$ as well.
\end{proof}

\begin{proof}[Proof of Lemma~\ref{KeyLemma}]
    Denote by $[\widetilde{M},\widetilde{\pi}]$ the mod $\ell$ inertial supercuspidal support for $\textnormal{Mod}_{\textnormal{sm}}(\overline{\mathbf{F}}_{\ell}[\textnormal{G}])$ corresponding to $\phi_1\oplus \phi_2$. By \cite{Vig98}, IV.6.2, the list of inertial supercuspidal supports for $\textnormal{Mod}_{\textnormal{sm}}(\overline{\mathbf{Q}}_{\ell}[\textnormal{G}])$ lifting $[\widetilde{M},\widetilde{\pi}]$ is finite. Denote by $\tau_1,...,\tau_N:I_{L}\to\textnormal{GL}_{2n}(\overline{\mathbf{Q}}_{\ell})$ the corresponding inertial types under the local Langlands correspondence.

    Consider an unramified character $\chi:G_L\to \overline{\mathbf{F}}_{\ell}^{\times}$, and a continuous $\ell$-adic Galois representation $\rho:G_L\to\textnormal{GL}_{2n}(\overline{\mathbf{Q}}_{\ell})$ with $\overline{\rho}^{\textnormal{ss}}\cong (\phi_1\otimes\chi)\oplus \phi_2$. If we write $(r,N):=\textnormal{WD}(\rho)$, we see that $\rho^{\textnormal{ss}}| _{W_{L}}\cong r^{\textnormal{ss}}$. In particular, by Theorem \ref{ModEllLL}, we have
    \begin{equation*}
        r|_{I_L}\cong \rho^{\textnormal{ss}}|_{I_L}\cong \tau_i
    \end{equation*}
    for some integer $1\leq i\leq N$.

    Write $\widetilde{m}:=m_{\tau_1}\cdot...\cdot m_{\tau_N}\cdot (\ell-1)$ and note that, for every Frobenius lift $\varphi\in W_L$, $r(\varphi^{\widetilde{m}})$ is central in $r(W_L)$ with its eigenvalues $\alpha_1,...,\alpha_d$ lying in $\overline{\mathbf{Z}}_{\ell}^{\times}$. Write $\{\delta_{i,1},...,\delta_{i,h_i}\}$ for the set of eigenvalues of $\phi_i(\varphi^{\widetilde{m}})$ for $i=1,2$. Then the set
    \begin{equation*}
        \{\alpha_j\mod \mathfrak{m}_{\overline{\mathbf{Z}}_{\ell}}\}_{j=1,...,d}
    \end{equation*}
    is given by $\{\delta_{1,j}\chi(\varphi^{\widetilde{m}})\}_{j=1,...,h_1}\cup\{\delta_{2,j}\}_{j=1,...,h_2}$.

    We claim that if the sets $\{\delta_{1,j}\chi(\varphi^{\widetilde{m}})\}_{j=1,...,h_1}$, and $\{\delta_{2,j}\}_{j=1,...,h_2}$ are disjoint, $\rho$ admits a direct sum decomposition as in the statement. It is clear that $r$ decomposes into $W_L$-invariant subspaces labelled by $\{\alpha_j\mod \mathfrak{m}_{\overline{\mathbf{Z}}_{\ell}}\}_{j=1,...,d}$. This means that there is a $W_L$-equivariant decomposition $r\cong r_1\oplus r_2$ with respect to the partition
    \begin{equation*}
        \{\alpha_i\mod \mathfrak{m}_{\overline{\mathbf{Z}}_{\ell}}\}_{j=1,...,d}=\{\delta_{1,j}\chi(\varphi^{\widetilde{m}})\}_{j=1,...,h_1}\bigsqcup\{\delta_{2,j}\}_{j=1,...,h_2}.
    \end{equation*}
    Moreover, given $\widetilde{r}_i\in \textnormal{JH}(r_i)$ for $i=1,2$, we see that $\widetilde{r}_1\ncong \widetilde{r}_2|\cdot|^{m}$ for any $m\in \mathbf{Z}$. Indeed, such an isomorphism would imply that the eigenvalues of $\widetilde{r}_1(\varphi^{\widetilde{m}})$ and $\widetilde{r}_2(\varphi^{\widetilde{m}})$ are congruent modulo $\mathfrak{m}_{\overline{\mathbf{Z}}_{\ell}}$ as $|\cdot|:W_L\to \overline{\mathbf{Z}}_{\ell}^{\times}$ takes values in $\mathbf{Z}_{\ell}^{\times}$ and $\ell-1|\widetilde{m}$. In particular, we obtain a decomposition $(r,N)\cong (r_1,N_1)\oplus (r_2,N_2)$ and hence the desired decomposition $\rho\cong \rho_1\oplus\rho_2$.

    Therefore, it suffices to find $\chi$ such that $\{\delta_{1,i}\chi(\varphi^{\widetilde{m}})\}_{j=1,...,h_1}$, and $\{\delta_{2,j}\}_{j=1,...,h_2}$ are disjoint. This is possible since there are only finitely many unramified characters $\chi:G_L\to \overline{\mathbf{F}}_{\ell}^{\times}$ that make the sets $\{\delta_{1,i}\chi(\varphi^{\widetilde{m}})\}_{j=1,...,h_1}$, and $\{\delta_{2,j}\}_{j=1,...,h_2}$ have a common element.

To prove the final statement, note that $\textnormal{rec}^{-1}_L(\textnormal{WD}(\rho))=\pi(\rho)|\det|^{\frac{1-2n}{2}}$, and $\textnormal{rec}^{-1}_L(\textnormal{WD}(\rho_i))=\pi(\rho_i)|\det|^{\frac{1-n}{2}}$ for $i=1,2$. Moreover, we see that $(\phi_1\otimes\chi,\phi_2)$ corresponds to an $(n,n)$-generic maximal ideal of $\mathfrak{Z}_{\textnormal{M}}$ implying that $\pi(\rho_1)$ and $\pi(\rho_2)$ cannot be twists of each other by an integral power of $|\det|$. Consequently, an application of Proposition 8.5 of \cite{Zel80} shows that we have isomorphisms
\begin{equation*}
    \pi(\rho)\cong \textnormal{n-Ind}_{\textnormal{Q}}^{\textnormal{G}}\left((\pi(\rho_1)|\det|^{\frac{n}{2}})\otimes( \pi(\rho_2)|\det|^{\frac{n}{2}})\right)\cong \textnormal{Ind}_{\textnormal{Q}}^{\textnormal{G}}\left(\pi(\rho_1)\otimes (\pi(\rho_2)|\det|^{n})\right).
\end{equation*}
We can then apply Lemma~\ref{GeomLemmaLemma} to get the isomorphism
\begin{equation*}
    J_{\textnormal{Q}}(\pi(\rho))_{\mathfrak{m}_{\chi}}\cong \left(\pi(\rho_1)\otimes(\pi(\rho_2)|\det|^{n})\right)_{\mathfrak{m}_{\chi}}\cong \pi(\rho_1)\otimes(\pi(\rho_2)|\det|^{n}).
\end{equation*}
\end{proof}

\section{Preliminaries on the cohomology of locally symmetric spaces}

In this section, we collect some results on locally symmetric spaces and their cohomology. For a more in-depth introduction to some of the background material, we invite the reader to visit \cite{ACC23} and \cite{CN23}.
\subsection{The general theory}

\subsubsection{Locally symmetric spaces}\label{sec_locally_symm_spaces} Let $G$ be a connected linear algebraic group over a number field $F$. We denote by $X^G$ the symmetric space associated with $\textnormal{Res}_{F/\mathbf{Q}}G$ in the sense of \cite{BS73}, a homogeneous $G(F\otimes_{\mathbf{Q}}\mathbf{R})$-space.

We call a subgroup $K_G\leq G(\mathbf{A}_F^{\infty})$ \textit{good} if it is a neat compact open subgroup (in the sense of \cite{Pin90}, page 12-13) of the form $K_G=\prod_vK_{G,v}$, the product taken over the set of finite places of $F$. For a good subgroup $K_G\leq G(\mathbf{A}^{\infty}_F)$ we introduce the corresponding locally symmetric space
\begin{equation*}
    X_{K_G}:=G(F)\backslash(X^G\times G(\mathbf{A}_F^{\infty})/K_G),
\end{equation*}
a smooth orientable Riemannian manifold. Let $\overline{X}^G$ denote the Borel--Serre partial compactification of the symmetric space $X^G$ (cf. \cite{BS73}, \textsection7.1). We obtain the Borel--Serre compactification
\begin{equation*}
    \overline{X}_{K_G}:=G(F)\backslash(\overline{X}^G\times G(\mathbf{A}_F^{\infty})/K_G),
\end{equation*}
of $X_{K_G}$,
an orientable compact smooth manifold with corners. The corresponding open embedding $j_{K_G}:X_{K_G}\hookrightarrow \overline{X}_{K_G}$ is known to be a homotopy equivalence and realises the former as its interior. We write $\partial X^G:=\overline{X}^G\setminus X^G$ respectively, $\partial X_{K_G}:=\overline{X}_{K_G}\setminus X_{K_G}$ for the corresponding boundary.

Following \cite{CN23}, we define the "infinite level" locally symmetric spaces
\begin{equation}\label{inflevelspaces}
    \mathfrak{X}_G:=\varprojlim_{K_G}X_{K_G}, \textnormal{ }\overline{\mathfrak{X}}_G:=\varprojlim_{K_G} \overline{X}_{K_G}, \textnormal{ }\partial\mathfrak{X}_G=\varprojlim_{K_G}\partial X_{K_G}
\end{equation}
where the limits are taken over good subgroups of $G(\mathbf{A}_{F}^{\infty})$. We endow these sets with the projective limit topology. In particular, both $\overline{\mathfrak{X}}_G$ and $\partial\mathfrak{X}_G$ become compact Hausdorff spaces and therefore, $\mathfrak{X}_G$ will be a locally compact Hausdorff space. We write $j_G:\mathfrak{X}_G\hookrightarrow \overline{\mathfrak{X}}_G$ for the natural open embedding.

Each of the spaces \ref{inflevelspaces} admits a natural continuous right action of $G(\mathbf{A}_F^{\infty})$ where the latter is equipped with its locally profinite topology. The embedding $j_G$ is equivariant with respect to this action. Moreover, the induced action of any good subgroup $K_G$ turns $\overline{\mathfrak{X}}_G$ and $\partial\mathfrak{X}_G$ into free $K_G$-spaces in the sense of \cite{NT16}, Definition 2.23.

\subsubsection{Cohomology and Hecke actions}\label{sec2.2}
For a (non-empty) finite set of finite places $S$ of $F$, write $G^S:=G(\mathbf{A}_{F}^{S\cup\{\infty\}})$ and $G_S:=G(\mathbf{A}_{F,S})$, respectively. For a good subgroup $K_G$ of $G(\mathbf{A}^{\infty}_F)$, we similarly set $K^S_G:=\prod_{v\notin S}K_{G,v}$ and $K_{G,S}:=\prod_{v\in S}K_{G,v}$.

Let $\Lambda$ be a commutative ring. According to \cite{Sch98}, Lemma 1, the category $\textnormal{Sh}_{G(\mathbf{A}_F^{\infty})}(\mathfrak{X},\Lambda)$ of $G(\mathbf{A}_F^{\infty})$-equivariant sheaves of $\Lambda$-modules (in the sense of \cite{NT16}, Definition 2.22, (2)) has enough injectives for $\mathfrak{X}\in \{\overline{\mathfrak{X}}_G,\mathfrak{X}_G,\partial\mathfrak{X}_G\}$. Moreover, by \cite{NT16}, Lemma 2.26, we have an equivalence of categories $\textnormal{Sh}_{G(\mathbf{A}_F^{\infty})}(\ast,\Lambda)\cong \textnormal{Mod}_{\textnormal{sm}}(\Lambda[G(\mathbf{A}_F^{\infty})])$.

In particular, we can define
\begin{equation*}
    \pi^G(\Lambda):=R\Gamma(\overline{\mathfrak{X}}_G,\Lambda)\in D^+_{\textnormal{sm}}(G(\mathbf{A}_F^{\infty}),\Lambda),
\end{equation*}
\begin{equation*}
    \pi^G_c(\Lambda):=R\Gamma(\overline{\mathfrak{X}}_G,j_{G,!}\Lambda)\in D^+_{\textnormal{sm}}(G(\mathbf{A}_F^{\infty}),\Lambda),
\end{equation*}
\begin{equation*}
    \pi^G_{\partial}(\Lambda):=R\Gamma(\partial\mathfrak{X}_G,\Lambda)\in D^+_{\textnormal{sm}}(G(\mathbf{A}_F^{\infty}),\Lambda).
\end{equation*}
Taking derived $K_G$-invariants then yields objects
\begin{equation*}
    R\Gamma(K_G,\pi^G(\Lambda)),\textnormal{ }R\Gamma(K_G,\pi^G_c(\Lambda)),\textnormal{ }R\Gamma(K_G,\pi^G_{\partial}(\Lambda))
\end{equation*}
in $D^+(\mathcal{H}(G,K_G)\otimes_{\mathbf{Z}}\Lambda)$ where
\begin{equation*}
    \mathcal{H}(G,K_G):=\mathbf{Z}[K_G\backslash G(\mathbf{A}_F^{\infty})/K_G]
\end{equation*}
is the algebra of $\mathbf{Z}$-valued compactly supported $K_G$-bi-invariant functions on $G(\mathbf{A}_F^{\infty})$ equipped with the convolution product.

Let $(-)^{\sim}:D^+(\mathcal{H}(G,K_G)\otimes_{\mathbf{Z}}\Lambda)\to D^+(\Lambda)$ denote the forgetful functor. By \cite{CN23}, Proposition 2.1.3 and base change in the compactly supported case, there are natural isomorphisms
\begin{equation*}
    R\Gamma(X_{K_G},\Lambda)\cong R\Gamma(K_G,\pi^G(\Lambda))^{\sim},
\end{equation*}
\begin{equation*}
     R\Gamma_c(X_{K_G},\Lambda):=R\Gamma(X_{K_G},j_{K_G,!}\Lambda)\cong R\Gamma(K_G,\pi^G_c(\Lambda))^{\sim},
\end{equation*}
\begin{equation*}
    R\Gamma(\partial X_{K_G},\Lambda)\cong R\Gamma(K_G,\pi^G_{\partial}(\Lambda))^{\sim}
\end{equation*}
in $D^+(\Lambda)$. These identifications yield algebra homomorphisms
\begin{equation*}
    Z(\mathcal{H}(G,K_G))\otimes_{\mathbf{Z}}\Lambda\to \textnormal{End}(R\Gamma_{(c)}(X_{K_G},\Lambda)),
\end{equation*}
\begin{equation*}
    Z(\mathcal{H}(G,K_G))\otimes_{\mathbf{Z}}\Lambda\to \textnormal{End}(R\Gamma(\partial X_{K_G},\Lambda)).
\end{equation*}

We record a standard lemma on twists of automorphic Hecke eigensystems. Let $K_G\leq G(\mathbf{A}_F^{\infty})$ be a good subgroup, $S$ be a finite set of finite places of $F$. Let $\mathcal{O}$ be the ring of integers of some finite extension of $\mathbf{Q}_{\ell}$ with a choice of uniformiser $\varpi$. Let $\chi:G_F\to \mathcal{O}^{\times}$ be a continuous character such that $\chi\circ \textnormal{Art}_{F_v}$ is trivial on $\det(K_{G,v})$ for $v\notin S$. We define an isomorphism of $\mathcal{O}$-algebras
\begin{equation}\label{twistingmap}
   f_{\chi}:\mathcal{H}(G^S,K_G^S)\otimes_{\mathbf{Z}}\mathcal{O}\to \mathcal{H}(G^S,K_G^S)\otimes_{\mathbf{Z}}\mathcal{O},
\end{equation}
\begin{equation*}
    [K_G^SgK_G^S]\mapsto \chi(\textnormal{Art}_{F_v}(\det(g)))^{-1}[K_G^SgK_G^S].
\end{equation*}
\begin{Lemma}\label{CharTwistLemma}
    Let $S$ be a finite set of finite places of $F$, $ K_G\leq G(\mathbf{A}_{F}^{\infty})$ be a good subgroup. Let $\chi:G_{F}\to \mathcal{O}^{\times}$ be a continuous character such that $\chi\circ \textnormal{Art}_{F_v}$ is trivial on $\det(K_{G,v})$ for every finite place $v$ of $F$. There are isomorphisms
    \begin{equation*}
        R\Gamma_{(c)}(X_{K_G},\mathcal{O})\xrightarrow{\sim} R\Gamma_{(c)}(X_{K_G},\mathcal{O})
    \end{equation*}
        in $D^+(\mathcal{O})$ that are equivariant for the action of $\mathcal{H}(G^S,K_G^S)\otimes_{\mathbf{Z}}\mathcal{O}$ where this Hecke algebra acts via $f_{\chi}$ on the right-hand side and by the standard Hecke action on the left-hand side.
\end{Lemma}
\begin{proof}
See for instance \cite{Hev23}, Lemma 2.5.
\end{proof}

Finally, we recall the definition of the unnormalised Satake transform. Let $G$ be a reductive group over $F$ and $P=M\rtimes N\leq G$ be a parabolic subgroup with a Levi decomposition. Let $K_G\leq G(\mathbf{A}_F^{\infty})$ be a good subgroup and write $K_P=K_G\cap P(\mathbf{A}^{\infty}_F)$, $K_N= K_G\cap N(\mathbf{A}^{\infty}_F)$ and $K_M=\textnormal{im}(K_P\to M(\mathbf{A}_F^{\infty}))$. Given a finite set of finite places $S$ of $F$, we say that $K_{G}$ is \textit{$S$-decomposed} if $K_{P,v}=K_{M,v}\rtimes K_{N,v}$ for every $v\notin S$. Note that if $S$ is so that $G_{F_v}$ is unramified and $K_{G,v}$ is hyperspecial for $v\notin S$, then $K_{G}$ is $S$-decomposed. For $S$-decomposed $K_{G}$ we have the homomorphisms of Hecke algebras
\begin{equation*}
    r_P:\mathcal{H}(G^S,K_G^S)\to \mathcal{H}(P^S,K_P^S) \textnormal{ and } r_M:\mathcal{H}(P^S,K_P^S)\to \mathcal{H}(M^S,K_M^S)
\end{equation*}
given by "restriction to $P$" and "integration along $N$", respectively (cf. \cite{NT16}, 2.2.3, 2.2.4, 2.2.6). We write $\mathcal{S}:=r_M\circ r_P$, and call it the unnormalised Satake transform.

\subsubsection{Cohomology of the Borel--Serre boundary}
Let $G$ denote a connected reductive group over a number field $F$. Fix a rational prime $\ell$ and let $\Lambda$ be an Artinian $\mathbf{Z}_{\ell}$-algebra.

Recall that we have a $G(\mathbf{A}_{F}^{\infty})$-equivariant decomposition
\begin{equation*}
    \partial \mathfrak{X}_G=\coprod_{Q}\mathfrak{X}_G^{Q}
\end{equation*}
into locally closed subspaces labelled by the set of standard $F$-rational parabolic subgroups $Q\leq G$. Namely, the members of the decomposition are given by the quotient spaces
\begin{equation*}
    \mathfrak{X}_G^Q:=(\mathfrak{X}_Q\times G(\mathbf{A}_{F}^{\infty}))/Q(\mathbf{A}_{F}^{\infty})
\end{equation*}
where $q\in Q(\mathbf{A}_{F}^{\infty})$ acts on $(x,g)\in \mathfrak{X}_Q\times G(\mathbf{A}_{F}^{\infty})$ by the formula $(xq,q^{-1}g)$ and $G(\mathbf{A}_{F}^{\infty})$ is equipped with its profinite topology. This space is further equipped with the continuous right $G(\mathbf{A}_{F}^{\infty})$-action given by the formula $(x,g)\cdot g'=(x,gg')$. 

For a good subgroup $K_G\leq G(\mathbf{A}_{F}^{\infty})$, we obtain the corresponding stratification
\begin{equation*}
    \partial X_{K_G}=\coprod_{Q}X_{K_G}^Q
\end{equation*}
into locally closed subspaces.

In an analogous manner, we introduce the closed counterpart
\begin{equation*}
    \overline{\mathfrak{X}}_{G}^{Q}:=(\overline{\mathfrak{X}}_Q\times G(\mathbf{A}_{F}^{\infty}))/Q(\mathbf{A}_{F}^{\infty})
\end{equation*}
and write $j_G^Q$ for the $G(\mathbf{A}_F^{\infty})$-equivariant open embedding $\mathfrak{X}_G^Q\hookrightarrow\overline{\mathfrak{X}}_G^Q$.

Let $Q=M\rtimes N\leq G$ be an $F$-rational parabolic subgroup, $S$ be a finite set of finite places of $F$ such that for every $v\notin S$ the base change $G\times_FF_v$ admits a reductive model over $\mathcal{O}_{F_v}$. For any compact open subgroup $K_{G,S}\leq G_S$ we have a finite decomposition $G_S=\coprod_{i=1}^rQ_{S}g_iK_{G,S}$ where, without loss of generality, $g_1=1$. This yields a $G^S\times K_{G,S}$-equivariant decomposition
    \begin{equation}\label{decomposition}
        \overline{\mathfrak{X}}_{G}^Q\cong \coprod_{i=1}^r\left(\overline{\mathfrak{X}}_Q\times(G^S\times K_{G,S})\right)/\left(Q^S\times ((g_iK_{G,S}g_{i}^{-1})\cap Q_{S})\right)
    \end{equation}
    where on the $i$-th factor the quotient is taken with respect to the action $(x,g)\cdot q=(xq,g_i^{-1}q^{-1}g_ig)$.
    We will set $\overline{\mathfrak{X}}_{G,K_{G,S}}^{Q}$ to be the factor corresponding to $g_1(=1)$.

    One analogously introduces the open $G^S\times K_{G,S}$-equivariant subspace $\mathfrak{X}_{G,K_{G,S}}^Q\subset \overline{\mathfrak{X}}_{G,K_{G,S}}^{Q}$ and we write $j_{G,K_{G,S}}^{Q}$ for the induced embedding $\mathfrak{X}_{G,K_{G,S}}^{Q}\hookrightarrow\overline{\mathfrak{X}}_{G,K_{G,S}}^{Q}$.
    
    Using the Iwasawa decomposition $G^S=Q^SG(\widehat{\mathcal{O}}_{F}^S)$ and \ref{decomposition}, one easily deduces the $K_G$-equivariant isomorphisms 
    \begin{equation*}
        \overline{\mathfrak{X}}_{G}^Q\cong \coprod_{i=1}^r(\overline{\mathfrak{X}}_Q\times K_G)/K^i_Q,
    \end{equation*}
    \begin{equation}\label{K_Q^1}
        \overline{\mathfrak{X}}_{G,K_{G,S}}^{Q}\cong (\overline{\mathfrak{X}}_Q\times K_G)/K^1_Q
    \end{equation}
    where $K_G:= G(\widehat{\mathcal{O}}_F^S)K_{G,S}$, and $K^i_Q:=(g_iK_Gg_i^{-1})\cap Q(\mathbf{A}_F^{\infty})$. In  particular, both spaces are visibly compact and Hausdorff and therefore become free $K_G$-spaces.
    
    \begin{Lemma}\label{BorelSerreLemma}
        Let $Q=M\rtimes N\leq G$ be an $F$-rational parabolic subgroup, $S$ be a finite set of finite places of $F$ such that $G\times_FF_v$ admits a reductive model over $\mathcal{O}_{F_v}$ for $v\notin S$. Let $K_{G,S}\leq G_S$ be a compact open subgroup. Then we have a commutative diagram
        \begin{equation*}
\begin{tikzcd}
	{R\Gamma(\overline{\mathfrak{X}}_{G,K_{G,S}}^Q,j_{G,K_{G,S},!}^Q\Lambda)} && {\textnormal{Ind}^{G^S\times K_{G,S}}_{Q^S\times K_{Q,S}}R\Gamma(\overline{\mathfrak{X}}_M,j_{M,!}\Lambda)} \\
	{R\Gamma(\overline{\mathfrak{X}}_{G,K_{G,S}}^Q,\Lambda)} && {\textnormal{Ind}^{G^S\times K_{G,S}}_{Q^S\times K_{Q,S}}R\Gamma(\overline{\mathfrak{X}}_M,\Lambda)}
	\arrow["\sim", from=1-1, to=1-3]
	\arrow["{i_{c,G,K_{G,S}}^Q}", from=1-1, to=2-1]
	\arrow["{i_{c,M}}", from=1-3, to=2-3]
	\arrow["\sim", from=2-1, to=2-3]
\end{tikzcd}
    \end{equation*}
    in $D^+_{\textnormal{sm}}(G^S\times K_{G,S},\Lambda)$, where the horizontal maps are isomorphisms and the vertical ones are induced by natural maps $j_{G,K_{G,S},!}^Q\Lambda\to \Lambda$ and $j_{M,!}\Lambda\to \Lambda$, respectively.
    \end{Lemma}

\begin{proof}
    We first prove that we have a commutative diagram
    \begin{equation}\label{G-to-Q}
        \begin{tikzcd}
	{R\Gamma(\overline{\mathfrak{X}}_{G,K_{G,S}}^Q,j_{G,K_{G,S},!}^Q\Lambda)} && {\textnormal{Ind}^{G^S\times K_{G,S}}_{Q^S\times K_{Q,S}}R\Gamma(\overline{\mathfrak{X}}_Q,j_{Q,!}\Lambda)} \\
	{R\Gamma(\overline{\mathfrak{X}}_{G,K_{G,S}}^Q,\Lambda)} && {\textnormal{Ind}^{G^S\times K_{G,S}}_{Q^S\times K_{Q,S}}R\Gamma(\overline{\mathfrak{X}}_Q,\Lambda)}
	\arrow["\sim", from=1-1, to=1-3]
	\arrow["{i_{c,G,K_{G,S}}^Q}", from=1-1, to=2-1]
	\arrow["{i_{c,Q}}", from=1-3, to=2-3]
	\arrow["\sim", from=2-1, to=2-3]
\end{tikzcd}
    \end{equation}
    in $D^+_{\textnormal{sm}}(G^S\times K_{G,S},\Lambda)$.
    
    We follow the proof of \cite{CN23}, Lemma 2.3.14. Consider the $Q^S\times K_{Q,S}$-equivariant maps
    $i_Q:\mathfrak{X}_Q\to\mathfrak{X}_{G,K_{G,S}}^Q$, $\overline{i}_Q:\overline{\mathfrak{X}}_Q\to\overline{\mathfrak{X}}_{G,K_{G,S}}^Q$ defined by the formula $x\mapsto (x,1)$. By base change (cf. \cite{KS94}, Proposition 2.5.11), we have $\overline{i}_Q^{\ast}\circ j_{G,K_{G,S},!}^Q=j_{Q,!}\circ i_Q^{\ast}$. 
    We obtain maps
    \begin{equation*}
        R\Gamma(\overline{\mathfrak{X}}_{G,K_{G,S}}^Q,(j_{G,K_{G,S},!}^Q)\Lambda)\to R\Gamma(\overline{\mathfrak{X}}_Q,(\overline{i}_Q^{\ast}j_{G,K_{G,S},!}^Q)\Lambda)=
    \end{equation*}
    \begin{equation*}
        R\Gamma(\overline{\mathfrak{X}}_Q,(j_{Q,!}i^{\ast}_Q)\Lambda)
    \end{equation*}
    in $D^+_{\textnormal{sm}}(Q^S\times K_{Q,S},\Lambda)$.
    An application of Frobenius reciprocity gives rise to the commutative diagram \ref{G-to-Q}.

    To see that the horizontal maps are isomorphisms, one can apply $\textnormal{Res}_{K_G}^{G^S\times K_{G,S}}$ to \ref{G-to-Q} to reduce the problem to checking that the analogously defined maps
    \begin{equation*}
        R\Gamma((\overline{\mathfrak{X}}_Q\times K_G)/K^1_Q,(j_{G,K_{G,S},!}^Q)\Lambda)\xrightarrow{\sim}\textnormal{Ind}_{K_Q^1}^{K_G}R\Gamma(\overline{\mathfrak{X}}_Q,(j_{Q,!})\Lambda)
    \end{equation*}
    are isomorphisms in $D^+_{\textnormal{sm}}(K_G,\Lambda)$. Here we used that by the prime-to-$S$ Iwasawa decomposition we have the identifications \ref{K_Q^1} and $\textnormal{Res}_{K_G}^{G^S\times K_{G,S}}\circ \textnormal{Ind}_{Q^S\times K_{Q,S}}^{G^S\times K_{G,S}}=\textnormal{Ind}_{K_{Q}^1}^{K_G}$.
    
    We then argue just as in the second part of the proof of \cite{CN23}, Lemma 2.3.14. More precisely, we apply their argument with the choices $X:=\overline{\mathfrak{X}}_Q$, $K_P:=K_{Q}^1$ and $K:=K_G$. Moreover, we run the argument for the sheaves $\Lambda,j_{K_{Q}^1,!}\Lambda\in \textnormal{Sh}(\overline{X}_{K_{Q}^1},\Lambda)$. In the non-compactly supported case it yields the desired isomorphism. In the compactly supported case we obtain that pullback along $\overline{i}_Q$ induces an isomorphism
    \begin{equation*}
        R\Gamma((\overline{\mathfrak{X}}_Q\times K_G)/K_Q^1,\phi_2^{\ast}(j_{K_Q^1,!}\Lambda))\cong \textnormal{Ind}_{K_Q^1}^{K_G}R\Gamma(\overline{\mathfrak{X}}_Q,\phi_1^{\ast}(j_{K_Q^1,!}\Lambda))
    \end{equation*}
    in $D^+_{\textnormal{sm}}(K_G,\Lambda)$ where we write $\phi_1$ and $\phi_2$ for the projections
\begin{equation*}
        \phi_1:\overline{\mathfrak{X}}_Q\to \overline{\mathfrak{X}}_Q/K_Q^1=\overline{X}_{K_Q^1}\textnormal{, and}
    \end{equation*}
    \begin{equation*}
        \phi_2:(\overline{\mathfrak{X}}_{Q}\times K_G)/K^1_Q\to \overline{X}_{K_Q^1},
    \end{equation*} respectively. To conclude, we note that by base change we have $\phi_1^{\ast}(j_{K_Q^1,!}\Lambda)=j_{Q,!}\Lambda$, and $\phi_2^{\ast}(j_{K_Q^1,!}\Lambda)= j_{G,K_{G,S},!}^Q\Lambda$.
    
To finish the proof it suffices to show the existence of isomorphisms
    \begin{equation}\label{inflation}
        R\Gamma(\overline{\mathfrak{X}}_Q,(j_{Q,!})\Lambda)\cong \textnormal{Inf}_{M(\mathbf{A}_{F}^{\infty})}^{Q(\mathbf{A}_{F}^{\infty})}R\Gamma(\overline{\mathfrak{X}}_M,(j_{M,!})\Lambda)
    \end{equation}
    in $D^+_{\textnormal{sm}}(Q(\mathbf{A}_{F}^{\infty}),\Lambda)$ compatible with the natural maps "forgetting the support".

    For cohomology without compact support, this is proved in \cite{CN23}, Lemma 4.1.6. The isomorphism is realised by pullback along the $Q(\mathbf{A}_{F}^{\infty})$-equivariant projection $\overline{\mathfrak{X}}_Q\to \overline{\mathfrak{X}}_M$. It suffices to check that the induced map is an isomorphism on cohomology. However, the cohomology of the source and target recover the corresponding completed cohomology groups (\cite{CN23}, Lemma 2.1.7). In particular, we can reduce the question to the vanishing of $\varinjlim_{K_N}H^j(\overline{X}_{K_N},\Lambda)$ for $j\neq 0$. Indeed, the reduction follows from the Leray--Serre spectral sequence associated with $\overline{X}_{K_N}\to \overline{X}_{K_Q}\to \overline{X}_{K_M}$ for good subgroups of $Q(\mathbf{A}_{F}^{\infty})$ of the form $K_Q=K_M\rtimes K_N$ applied to the sheaves $\Lambda$, and $j_{K_Q,!}\Lambda$ on $\overline{X}_{K_Q}$, respectively. Note that in the case of compactly supported cohomology, this reduction uses that the fibers are compact, giving $j_{K_N,!}\Lambda=\Lambda$. The claimed vanishing is proved in \textit{loc. cit.}, Lemma 4.1.6.

    Finally, the obtained isomorphisms are clearly compatible with natural maps forgetting the support.
\end{proof}

\subsection{The quasi-split unitary group and the general linear group}

\subsubsection{The quasi-split unitary group}\label{Sec2.3}
For the rest of the section, we restrict our setup to the groups of interest for this article, the quasi-split unitary group of rank $2n$ and the general linear group of rank $n$ appearing as its Levi subgroup. Let $n\geq 2$ be an integer and $F$ be an imaginary CM field. Denote by $F^+$ its maximal totally real subfield $F^+\leq F$. We write $c\in \textnormal{Gal}(F/F^+)$ for the complex conjugation on $F$. We define the $2n\times 2n$ matrix
$$J_n:=\begin{pmatrix}
0 & \Psi_n \\
-\Psi_n & 0
\end{pmatrix}$$
where $\Psi_n$ is the $n\times n$ matrix with $1$'s on the anti-diagonal and $0$'s elsewhere. We denote by $\widetilde{G}/\mathcal{O}_{F^+}$ the group scheme with $R$-points
\begin{equation*}
    \widetilde{G}(R)=\{g\in \textnormal{GL}_{2n}(R\otimes_{\mathcal{O}_{F^+}}\mathcal{O}_F)\mid {}^tgJ_ng^c=J_n\}
\end{equation*}
where $^t(-)$ denotes the transpose matrix. This is an $\mathcal{O}_{F^+}$-model of the quasi-split unitary group $U(n,n)/F^+$. Since for any finite place $v$ of $F$ we have an identification $\widetilde{G}_{\mathcal{O}_{F_v}}\cong \textnormal{GL}_{2n}$, $\widetilde{G}$ becomes reductive after base change to $\mathcal{O}_{F^+_{\Bar{v}}}$ for any finite place $\Bar{v}$ of $F^+$ that is unramified in $F$.

We will denote by $P\leq \widetilde{G}$ the so-called Siegel parabolic subgroup given by block upper triangular matrices with blocks of size $n\times n$. Write $P=G\ltimes U$ for its Levi decomposition where $G$ is the closed subgroup of block diagonal matrices. One can then identify $G$ with the group $\textnormal{Res}_{\mathcal{O}_F/\mathcal{O}_{F^+}}\textnormal{GL}_n$. Namely, write $(-)^{\ast}$ for the anti-involution $A^{\ast}=\psi_n^tA^c\psi_n^{-1}$ of $\textnormal{Res}_{\mathcal{O}_{F}/\mathcal{O}_{F^+}}\textnormal{GL}_n$. Then $P\leq \widetilde{G}$ can be identified with the subgroup of $\textnormal{Res}_{\mathcal{O}_F/\mathcal{O}_{F^+}}\textnormal{GL}_{2n}$ of matrices of the form
$$\begin{pmatrix}
A & B \\
C & D
\end{pmatrix} =\begin{pmatrix}
D^{-\ast} & B \\
0 & D
\end{pmatrix}$$
where $D$ and $B$ are points of $ \textnormal{Res}_{\mathcal{O}_F/\mathcal{O}_{F^+}}\textnormal{GL}_n$ so that $B^{\ast}=B$. The subgroup $G\leq P$ is then cut out by the equation $B=0$. In particular, the association $\begin{pmatrix}
D^{-\ast} & B \\
0 & D
\end{pmatrix}\mapsto D$ gives our identification $G\cong \textnormal{Res}_{\mathcal{O}_{F}/\mathcal{O}_{F^+}}\textnormal{GL}_n$.

From now on, we set $\widetilde{X}$ to be the symmetric space $X^{\widetilde{G}}$ and write $\widetilde{X}_{\widetilde{K}}$ for the locally symmetric space associated with a good subgroup $\widetilde{K}\leq \widetilde{G}(\mathbf{A}_{F^+}^{\infty})$. Furthermore, we write $X$ for the symmetric space $X^G$ and $X_K$ for the locally symmetric space attached to a good subgroup $K\leq G(\mathbf{A}^{\infty}_{F^+})=\textnormal{GL}_n(\mathbf{A}_F^{\infty})$. 

%We set $T\leq \widetilde{B}\leq \widetilde{G}$ to be the standard choice of maximal torus and Borel subgroup given by the diagonal and, respectively, the upper triangular matrices of $\widetilde{G}$. Note that $B=\widetilde{B}\cap G\leq G$ is given by the Borel subgroup of upper triangular matrices.

We note that for a place $\Bar{v}$ of $F^+$ splitting in $F$, a choice of place $v|\Bar{v}$ in $F$ gives a canonical isomorphism $\iota_v:\widetilde{G}(F^+_{\Bar{v}})\cong \textnormal{GL}_{2n}(F_v)$. Namely, noting the identification $F^+_{\Bar{v}}\otimes_{F^+}F\cong F_v\times F_{v^c}$, $\iota_v$ is given by the projection to the first factor of the standard inclusion $\widetilde{G}(F^+_{\Bar{v}})\subset \textnormal{GL}_{2n}(F_v)\times \textnormal{GL}_{2n}(F_{v^c})$. 

The inclusion
\begin{equation*}
    G(F^+_{\Bar{v}})=\textnormal{GL}_n(F_v)\times\textnormal{GL}_n(F_{v^c})\hookrightarrow\textnormal{GL}_{2n}(F_v)
\end{equation*}
under $\iota_v$ is given  by the formula
\begin{equation}\label{eq-unitary}
    (A,B)\mapsto \begin{pmatrix}
(\Psi_n {}^{t}B^{-1}\Psi_n)^c & 0 \\
0 & A
\end{pmatrix}.
\end{equation}

\subsubsection{Explicit Hecke operators}\label{sec2.4} We introduce the explicit Hecke operators that we will need later in this article. For each finite place $v$ of $F$, we fix a choice of a uniformiser $\varpi_v\in \mathcal{O}_{F_v}$. Consider a finite place $v$ of $F$, and an integer $1\leq i\leq n$. Denote by $T_{v,i}\in \mathcal{H}(\textnormal{GL}_n(F_v),\textnormal{GL}_n(\mathcal{O}_{F_v}))$ the Hecke operator
\begin{equation*}
    T_{v,i}=[\textnormal{GL}_n(\mathcal{O}_{F_v})\textnormal{diag}(\varpi_v,...,\varpi_v,1,...,1)\textnormal{GL}_n(\mathcal{O}_{F_v})]
\end{equation*}
corresponding to the diagonal matrix with $\varpi_v$ appearing in its diagonal exactly $i$ times. Consider the polynomial
\begin{equation*}
    P_v(X)=X^n-T_{v,1}X^{n-1}+...+(-1)^iq_v^{i(i-1)/2}T_{v,i}X^{n-i}+...
\end{equation*}
\begin{equation*}
    +q_v^{n(n-1)/2}T_{v,n}\in \mathcal{H}(\textnormal{GL}_n(F_v),\textnormal{GL}_n(\mathcal{O}_{F_v}))[X]
\end{equation*}
where $q_v=| \mathcal{O}_{F_v}/\varpi_v|$. Note that $P_v(X)$ evaluated on the Hecke eigenspace $\pi_v^{\textnormal{GL}_n(\mathcal{O}_{F_v})}$ for an unramified representation of $\textnormal{GL}_n(F_v)$ is exactly the characteristic polynomial of the Frobenius element acting on $\textnormal{rec}^T_{F_v}(\pi_v)$.

Let $\Bar{v}$ be a finite place of $F^+$ that is unramified in $F$ and $v$ be a choice of place of $F$ above it. For an integer $1\leq j \leq 2n$ denote by $\widetilde{T}_{v,j}\in \mathcal{H}(\widetilde{G}(F^+_{\Bar{v}}),\widetilde{G}(\mathcal{O}_{F^+_{\Bar{v}}}))\otimes_{\mathbf{Z}}\mathbf{Z}[q_{\Bar{v}}^{-1}]$ the Hecke operator denoted by $T_{G,v,j}$ in \cite{NT16}, Proposition-Definition 5.2. %Concretely, $q_v^{-j(2n-j)/2}\widetilde{T}_{v,j}$ is the image of the $i$th symmetric polynomial in $2n$ variables under the composition of the normalised Satake isomorphism and the dual map on Hecke algebras corresponding to the unramified endoscopic transfer from $\widetilde{G}(F^+_{\Bar{v}})$ to $\textnormal{GL}_{2n}(F_v)$. 
We then set
\begin{equation*}
    \widetilde{P}_v(X)=X^{2n}-\widetilde{T}_{v,1}X^{2n-1}+...+(-1)^jq_v^{j(j-1)/2}\widetilde{T}_{v,j}X^{2n-j}+...
\end{equation*}
\begin{equation*}
    +q_v^{2n(2n-1)/2}\widetilde{T}_{v,2n}\in \mathcal{H}(\widetilde{G}(F^+_{\Bar{v}}),\widetilde{G}(\mathcal{O}_{F^+_{\Bar{v}}}))\otimes_{\mathbf{Z}}\mathbf{Z}[q_{\Bar{v}}^{-1}][X].
\end{equation*}
We note that the evaluation of $\widetilde{P}_v(X)$ on $\pi_v^{\textnormal{GL}_{2n}(F_v)}$ is exactly the characteristic polynomial of the Frobenius element acting on $\textnormal{rec}^T_{F_v}(\pi_v)$ for $\pi_v$ the base change with respect to $F_v/F^{+}_{\Bar{v}}$ of an unramified representation $\sigma_{\Bar{v}}$ of $\widetilde{G}(F^+_{\Bar{v}})$.

To describe the effect of the unnormalised Satake transform (cf. end of \S\ref{sec2.2}) at unramified places we introduce the following notation. For a degree $d$ polynomial $f(X)$ with unit constant $a_0$, write $f^{\vee}(X):=a_0^{-1}X^df(X^{-1})$. It is the monic polynomial with zeroes given by the inverse of the zeroes of $f(X)$.
\begin{Prop}\label{Sataketransform}
Let $\Bar{v}$ be a place of $F^+$ that is unramified in $F$ and let $v\mid \Bar{v}$. Then the unnormalised Satake transform
\begin{equation*}
    \mathcal{S}:\mathcal{H}(\widetilde{G}(F^+_{\Bar{v}}),\widetilde{G}(\mathcal{O}_{F^+_{\Bar{v}}}))\to \mathcal{H}(G(F^+_{\Bar{v}}),G(\mathcal{O}_{F^+_{\Bar{v}}}))
\end{equation*}
maps $\widetilde{P}_v(X)$ to $P_v(X)q_v^{n(2n-1)}P_{v^c}^{\vee}(q_v^{1-2n}X)$.
\end{Prop}
\begin{proof}
This can easily be checked using the formula given in \cite{NT16}, Proposition-Definition 5.3 for $\mathcal{S}$.
\end{proof}

\subsubsection{Ramified Hecke algebras}\label{sec2.6}
We set up some notation for the ramified Hecke algebras we will need in this article. 

Fix a rational prime $\ell$. Let $F$ be an imaginary CM field with maximal totally real subfield $F^+$ as in the previous subsection. We fix a choice of uniformiser $\varpi_v$ in $F_v$ for every finite place of $F$. From now on, for every finite place $\Bar{v}$ of $F^+$, we fix a preferred choice of finite place $\Tilde{v}|\Bar{v}$ in $F$.

For a finite set of finite places $\overline{S}$ of $F^+$, and a good subgroup $K\leq G(\mathbf{A}_{F^+}^{\infty})$, we set
\begin{equation*}
\mathbf{T}^{\overline{S}}_{K}:=\left(\otimes'_{\Bar{v}\notin \overline{S}}Z(\mathcal{H}(G(F^+_{\Bar{v}}),K_{\Bar{v}}))\right)\otimes_{\mathbf{Z}}\mathbf{Z}_{\ell}.
\end{equation*}
Similarly, for a good subgroup $\widetilde{K}\leq \widetilde{G}(\mathbf{A}_{F^+}^{\infty})$, we set
\begin{equation*}
\widetilde{\mathbf{T}}^{\overline{S}}_{\widetilde{K}}:=\left(\otimes'_{\Bar{v}\notin \overline{S}}Z(\mathcal{H}(\widetilde{G}(F^+_{\Bar{v}}),\widetilde{K}_{\Bar{v}}))\right)\otimes_{\mathbf{Z}}\mathbf{Z}_{\ell}.
\end{equation*}
In particular, if $\widetilde{K}$ is ${\overline{S}}$-decomposed with respect to $P=GU$, and $\widetilde{K}^{\overline{S}}=\widetilde{G}(\widehat{\mathcal{O}}_{F^+}^{\overline{S}})$, we have the unnormalised Satake transform
\begin{equation*}
\mathcal{S}:\widetilde{\mathbf{T}}^{\overline{S}}_{\widetilde{K}}\to \mathbf{T}_K^{\overline{S}}
\end{equation*}
for $K:=\widetilde{K}\cap G(\mathbf{A}_{F^+}^{\infty})$.
\begin{Def}\label{DefinitionOnPlaces}

\begin{enumerate}
    \item We say that a finite place $\Bar{v}$ of $F^+$ is \textit{nice} if the rational prime $p$ under $\Bar{v}$ is different from $\ell$ and splits in an imaginary quadratic subfield\footnote{In particular, if $F$ admits no imaginary quadratic subfield, then none of the finite places of $F$ are nice!} of $F$.
    \item Given a finite set $\overline{S}_{\textnormal{bad}}$ of finite places of $F^+$, we call it \textit{unconditional} if it contains all $\ell$-adic places and the following holds.
    \begin{itemize}
    \item Given a finite place $\Bar{v}$ of $F^+$ not lying in $\overline{S}_{\textnormal{bad}}$. Then either $\Bar{v}$ is nice or the residue characteristic $p$ of $\Bar{v}$ is unramified in $F$ and $\overline{S}_{\textnormal{bad}}$ contains no $p$-adic places.
\end{itemize}
    Given a subset of an unconditional set $\overline{S}_{\textnormal{ram}}\subset \overline{S}_{\textnormal{bad}}$, we say that the pair $(\overline{S}_{\textnormal{bad}},\overline{S}_{\textnormal{ram}})$ is \textit{allowable} if $\overline{S}_{\textnormal{ram}}$ only consists of nice places. We write $\overline{S}_{\textnormal{avoid}}:=\overline{S}_{\textnormal{bad}}\setminus\overline{S}_{\textnormal{ram}}$.
    \item We say that a good subgroup $\widetilde{K}\leq \widetilde{G}(\mathbf{A}_{F^+}^{\infty})$ is $(\overline{S}_{\textnormal{bad}},\overline{S}_{\textnormal{ram}})$-\textit{parahoric} for an allowable pair if the following are satisfied.
    
    \begin{itemize}
        \item It is decomposed with respect to $P=GU$ (i.e. $\widetilde{K}\cap P(\mathbf{A}_{F^+}^{\infty})=(\widetilde{K}\cap M(\mathbf{A}_{F^+}^{\infty}))\rtimes (\widetilde{K}\cap U(\mathbf{A}_{F^+}^{\infty}))$).
        \item We have $\widetilde{K}^{\overline{S}_{\textnormal{bad}}}=\widetilde{G}(\widehat{\mathcal{O}}_{F^+}^{\overline{S}_{\textnormal{bad}}})$.
        \item For every $\Bar{v}\in \overline{S}_{\textnormal{ram}}$, $\widetilde{K}_{\Bar{v}}$ admits an Iwahori decomposition with respect to $P_{F^+_{\Bar{v}}}=G_{F^+_{\Bar
    v}}U_{F^+_{\Bar{v}}}\leq\widetilde{G}_{F^+_{\Bar
    v}}$ of the form $\overline{U}^{m}_{\Bar{v}}K_{\Bar{v}}U_{\Bar{v}}^0$ for some integer $m\geq 1$ where $\overline{U}_{\Bar{v}}^m:=\ker(\overline{U}(\mathcal{O}_{F^+_{\Bar{v}}})\to \overline{U}(\mathcal{O}_{F^+_{\Bar{v}}}/\varpi_v^m))$ and $U^0_{\Bar
    {v}}:=U(\mathcal{O}_{F^+_{\Bar{v}}})$.
    \end{itemize} 
\end{enumerate}
\end{Def}

Consider an allowable pair $(\overline{S}_{\textnormal{bad}},\overline{S}_{\textnormal{ram}})$ of finite sets of finite places of $F^+$ and an $(\overline{S}_{\textnormal{bad}},\overline{S}_{\textnormal{ram})}$-parahoric subgroup $\widetilde{K}\leq \widetilde{G}(\mathbf{A}_{F^+}^{\infty})$. Consider a place $\bar{v}\in \overline{S}_{\textnormal{ram}}$ and write $\overline{U}^m_{\Bar{v}}K_{\Bar{v}}U^0_{\Bar{v}}$ for the Iwahori decomposition of $\widetilde{K}_{\Bar{v}}$. Note that the formalism of \cite{ACC23}, \textsection2.1.9 applies to $\widetilde{K}_{\Bar{v}}$. We introduce the submonoid of $\widetilde{K}_{\Bar{v}}$-positive elements
\begin{equation*}
    G^+_{\Bar{v}}:=\{g\in G(F^+_{\Bar{v}})\mid gU^0_{\Bar{v}}g^{-1}\subset U^0_{\Bar{v}}\textnormal{, }g^{-1}\overline{U}^m_{\Bar{v}}g\subset\overline{U}^m_{\Bar{v}}  \}=
\end{equation*}
\begin{equation*}
    \{g\in G(F^+_{\Bar{v}})\mid gU^0_{\Bar{v}}g^{-1}\subset U^0_{\Bar{v}}\textnormal{, }g^{-1}\overline{U}^1_{\Bar{v}}g\subset\overline{U}^1_{\Bar{v}}  \}
\end{equation*}
of $G(F^+_{\Bar{v}})$ in the sense of \textit{loc. cit.}

For $K_{\Bar{v}}\leq G^+_{\Bar{v}}$, we can form the subalgebra $\mathcal{H}(G^+_{\Bar{v}},K_{\Bar{v}})$ of $\mathcal{H}(G(F^+_{\Bar{v}}),K_{\Bar{v}})$ of functions supported on $G^+_{\Bar{v}}$. We define the injective map
\begin{equation*}
    t_{\Bar{v}}:\mathcal{H}(G^+_{\Bar{v}},K_{\Bar{v}})\otimes_{\mathbf{Z}}\mathbf{Z}_{\ell}\to \mathcal{H}(\widetilde{G}(F^+_{\Bar{v}}),\widetilde{K}_{\Bar{v}})\otimes_{\mathbf{Z}}\mathbf{Z}_{\ell},
\end{equation*}
\begin{equation*}
    [K_{\Bar{v}}gK_{\Bar{v}}]\mapsto \delta_{P}(g)[\widetilde{K}_{\Bar{v}}g\widetilde{K}_{\Bar{v}}]
\end{equation*}
where we use that the rational number $\delta_{P}(g)$ lies in $\mathbf{Z}_{(p)}\subset \mathbf{Z}_{\ell}$. Thanks to \cite{BK98}, Corollary 6.12 (see \cite{ACC23}, Lemma 2.1.12), $t_{\Bar{v}}$ is in fact an \textit{algebra homomorphism}.

For a finite place $v|\Bar{v}$ of $F$, write
\begin{equation*}
    u_v:=\iota_v^{-1}(\textnormal{diag}(\varpi_v,...,\varpi_v,1,...,1))\in \widetilde{G}(F^+_{\Bar{v}})
\end{equation*}
where, in the diagonal, the first $n$ entries are given by $\varpi_v$ and the rest are all $1$'s. Note that $u_v$ lies in the center of $G(F^+_{\Bar{v}})$ and is in fact a so-called strongly positive element in the sense of \cite{ACC23}, \textsection2.1.9. In particular, $[K_{\Bar{v}}u_vK_{\Bar{v}}]$ is a central element in $\mathcal{H}(G^+_{\Bar{v}},K_{\Bar{v}})$ and \cite{ACC23}, Lemma 2.1.13, (i) shows that
\begin{equation*}
    \mathcal{H}(G(F^+_{\Bar{v}}),K_{\Bar{v}})=\mathcal{H}(G^+_{\Bar{v}},K_{\Bar{v}})[[K_{\Bar{v}}u_vK_{\Bar{v}}]^{-1}].
\end{equation*}

We define the commutative $\mathbf{Z}_{\ell}$-algebras
\begin{equation*}
    \widetilde{\mathbf{T}}_{\widetilde{K}}^{(\overline{S}_{\textnormal{bad}},\overline{S}_{\textnormal{ram}}),+}:=\widetilde{\mathbf{T}}^{\overline{S}_{\textnormal{bad}}}_{\widetilde{K}}\otimes_{\mathbf{Z}}\left(\otimes_{\Bar{v}\in \overline{S}_{\textnormal{ram}}}Z(\mathcal{H}(G^+_{\Bar{v}},K_{\Bar{v}}))\right),
\end{equation*}
\begin{equation*}
    \widetilde{\mathbf{T}}_{\widetilde{K}}^{(\overline{S}_{\textnormal{bad}},\overline{S}_{\textnormal{ram}})}:=\widetilde{\mathbf{T}}^{\overline{S}_{\textnormal{bad}}}_{\widetilde{K}}\otimes_{\mathbf{Z}}\left(\otimes_{\Bar{v}\in \overline{S}_{\textnormal{ram}}}Z(\mathcal{H}(G(F^+_{\Bar{v}}),K_{\Bar{v}}))\right).
\end{equation*}
Note that we have an injective $\mathbf{Z}_{\ell}$-algebra homomorphism
\begin{equation*}
    T_{\overline{S}_{\textnormal{ram}}}:=(\otimes_{\Bar{v}\notin \overline{S}_{\textnormal{bad}}}\textnormal{id})\otimes(\otimes_{\Bar{v}\in \overline{S}_{\textnormal{ram}}}t_{\Bar{v}}):\widetilde{\mathbf{T}}_{\widetilde{K}}^{(\overline{S}_{\textnormal{bad}},\overline{S}_{\textnormal{ram}}),+}\hookrightarrow \widetilde{\mathbf{T}}_{\widetilde{K}}^{\overline{S}_{\textnormal{avoid}}}.
\end{equation*}
Moreover, we have
\begin{equation*}
    \widetilde{\mathbf{T}}_{\widetilde{K}}^{(\overline{S}_{\textnormal{bad}},\overline{S}_{\textnormal{ram}}),+}[[K_{\Bar{v}}u_{\Tilde{v}}K_{\Bar{v}}]^{-1}]_{\overline{v}\in \overline{S}_{\textnormal{ram}}}=\widetilde{\mathbf{T}}_{\widetilde{K}}^{(\overline{S}_{\textnormal{bad}},\overline{S}_{\textnormal{ram}})}.
\end{equation*}
Finally, we extend the unnormalised Satake transform $\mathcal{S}:\widetilde{\mathbf{T}}_K^{\overline{S}_{\textnormal{bad}}}\to \mathbf{T}_K^{\overline{S}_{\textnormal{bad}}}$ to
\begin{equation*}
    \mathcal{S}_{\overline{S}_{\textnormal{ram}}}^{(+)}:\widetilde{\mathbf{T}}_{\widetilde{K}}^{(\overline{S}_{\textnormal{bad}},\overline{S}_{\textnormal{ram}}),(+)}\xrightarrow{\mathcal{S}\otimes(\otimes_{\Bar{v}\in \overline{S}_{\textnormal{ram}}}id)}\mathbf{T}_{K}^{\overline{S}_{\textnormal{avoid}}}.
\end{equation*}

\vspace{5mm}

We can now introduce the various faithful Hecke algebras we will need. In general, given $\mathbf{T}$ any of the abstract Hecke algebras discussed above, $C^{\bullet}\in D^+(\Lambda)$ for some $\mathbf{Z}_{\ell}$-algebra $\Lambda$ and a homomorphism $\mathbf{T}\to \textnormal{End}_{D^+(\Lambda)}(C^{\bullet})$, we write $\mathbf{T}(C^{\bullet}):=\textnormal{im}(\mathbf{T}\to \textnormal{End}_{D^+(\Lambda)}(C^{\bullet}))$. We also introduce some abbreviations.

Let $m\geq 1$ be an integer, $\overline{S}$ be a finite set of finite places of $F^+$, and $ K\leq G(\mathbf{A}_{F^+}^{\infty})$ be a good subgroup. We write
\begin{equation*}
    \mathbf{T}^{\overline{S}}(K,m):=\mathbf{T}^{\overline{S}}_{K}(R\Gamma(X_{K},\mathbf{Z}/\ell^m\mathbf{Z})),
\end{equation*}
\begin{equation*}
    \mathbf{T}_c^{\overline{S}}(K,m):=\mathbf{T}^{\overline{S}}_{K}(R\Gamma_c(X_{K},\mathbf{Z}/\ell^m\mathbf{Z})),
\end{equation*}
\begin{equation*}
    \mathbf{T}_{\partial}^{\overline{S}}(K,m):=\mathbf{T}^{\overline{S}}_{K}(R\Gamma(\partial X_{K},\mathbf{Z}/\ell^m\mathbf{Z}))\textnormal{, and}
\end{equation*}
\begin{equation*}
    \mathbf{T}^{\overline{S}}_!(K,m):=\mathbf{T}^{\overline{S}}_{K}/\textnormal{Ann}_{\mathbf{T}^{\overline{S}}_{K}}(i_{c,K,m})
\end{equation*}
where in the last formula $i_{c,K,m}$ is the $\mathbf{T}^{\overline{S}}_{K}$-equivariant map $R\Gamma_c(X_{K},\mathbf{Z}/\ell^m\mathbf{Z})\to R\Gamma(X_{K},\mathbf{Z}/\ell^m\mathbf{Z})$ in $D^+(\mathbf{Z}/\ell^m\mathbf{Z})$ that forgets the support and we act on it by $\mathbf{T}^{\overline{S}}_{K}$ on the left (or equivalently on the right).

For a good subgroup $ \widetilde{K}\leq \widetilde{G}(\mathbf{A}_{F^+}^{\infty})$, completely analogous definitions give rise to the faithful Hecke algebras $\widetilde{\mathbf{T}}_{?}^{\overline{S}}(\widetilde{K},m)$ where $?\in \{\varnothing,c,\partial,!\}$.

Moreover, if $(\overline{S}_{\textnormal{bad}}:=\overline{S},\overline{S}_{\textnormal{ram}})$ is an allowable pair of finite sets of finite places of $F^+$, and $\widetilde{K}\leq \widetilde{G}(\mathbf{A}_{F^+}^{\infty})$ is an $(\overline{S}_{\textnormal{bad}},\overline{S}_{\textnormal{ram}})$-parahoric subgroup, then we also introduce $\widetilde{\mathbf{T}}^{(\overline{S}_{\textnormal{bad}},\overline{S}_{\textnormal{ram}}),(+)}_{?}(\widetilde{K},m)$
where $?\in \{\varnothing,c,\partial,!\}$ in a completely analogous manner.

\subsubsection{The Borel--Serre boundary for $\widetilde{G}$}
Using Lemma~\ref{BorelSerreLemma}, we relate the boundary cohomology of $\widetilde{G}$ to the cohomology of $G$-locally symmetric spaces.

\begin{Lemma}\label{ImplicationOfBSLemma}
    Let $(\overline{S}_{\textnormal{bad}},\overline{S}_{\textnormal{ram}})$ be an allowable pair, and $K\leq G(\mathbf{A}_{F^+}^{\infty})$ be a good subgroup such that $K^{\overline{S}_{\textnormal{bad}}}=G(\widehat{\mathcal{O}}_{F^+}^{\overline{S}_{\textnormal{bad}}})$. Let $\widetilde{K}\leq \widetilde{G}(\mathbf{A}_{F^+}^{\infty})$ be an $(\overline{S}_{\textnormal{bad}},\overline{S}_{\textnormal{ram}})$-parahoric subgroup such that $\widetilde{K}\cap G(\mathbf{A}_{F^+}^{\infty})=K$. Then, for every subset $\overline{S}\subset \overline{S}_{\textnormal{ram}}$, there is a commutative $\widetilde{\mathbf{T}}^{\overline{S}_{\textnormal{bad}}}_{\widetilde{K}}$-equivariant diagram
    \begin{equation*}
        \begin{tikzcd}
	{\textnormal{Ind}_{P_{\overline{S}}}^{\widetilde{G}_{\overline{S}}}R\Gamma\left(K^{\overline{S}},\pi^G_c(\mathbf{Z}/\ell^m\mathbf{Z})\right)} && {\textnormal{Ind}_{P_{\overline{S}}}^{\widetilde{G}_{\overline{S}}}R\Gamma\left(K^{\overline{S}},\pi^G(\mathbf{Z}/\ell^m\mathbf{Z})\right)} \\
	{R\Gamma\left(\widetilde{K}^{\overline{S}},R\Gamma(\overline{\mathfrak{X}}_{\widetilde{G}}^{P},j_!(\mathbf{Z}/\ell^m\mathbf{Z}))\right)} && {R\Gamma\left(\widetilde{K}^{\overline{S}},R\Gamma(\overline{\mathfrak{X}}_{\widetilde{G}}^{P},\mathbf{Z}/\ell^m\mathbf{Z})\right)}
	\arrow["{i_c}", from=1-1, to=1-3]
	\arrow[from=1-1, to=2-1]
	\arrow[from=1-3, to=2-3]
	\arrow["{i_c}"', from=2-1, to=2-3]
\end{tikzcd}
    \end{equation*}
    in $D^+_{\textnormal{sm}}(\widetilde{G}_{\overline{S}},\mathbf{Z}/\ell^m\mathbf{Z})$ with the vertical maps realising the source as a $\widetilde{\mathbf{T}}^{\overline{S}_{\textnormal{bad}}}_{\widetilde{K}}$-equivariant direct summand\footnote{For the formulation of what this means, see for instance the discussion below \cite{ACC23}, Theorem 4.2.1.} of the target where, by abuse of notation, $j$ denotes the open inclusions induced by $\mathfrak{X}_{\widetilde{G}}\hookrightarrow \overline{\mathfrak{X}}_{\widetilde{G}}$ and the horizontal maps are the maps "forgetting the support".
\end{Lemma}
\begin{proof}
    We apply Lemma~\ref{BorelSerreLemma} with $G:=\widetilde{G}$, $Q:=P$, $M:=G$ and $S:=\overline{S}_{\textnormal{bad}}\setminus \overline{S}$.
    Applying $R\Gamma(\widetilde{K}^{\overline{S}_{\textnormal{bad}}},-)$ to the members of the obtained diagram, combining it with the Iwasawa decomposition away from $\overline{S}_{\textnormal{bad}}$ and passing to the direct summand corresponding to the representative $1$ in the Mackey formula for $\textnormal{Res}_{\widetilde{K}_{\overline{S}_{\textnormal{bad}}\setminus \overline{S}}}\circ \textnormal{Ind}_{P(\mathcal{O}_{F^+,\overline{S}_{\textnormal{bad}}\setminus\overline{S}})}^{\widetilde{G}(\mathcal{O}_{F^+,\overline{S}_{\textnormal{bad}}\setminus\overline{S}})}$ results in
    \begin{equation*}
        \textnormal{Ind}_{P_{\overline{S}}\times \widetilde{K}_{P,\overline{S}_{\textnormal{bad}}\setminus \overline{S}}}^{\widetilde{G}_{\overline{S}}\times \widetilde{K}_{\overline{S}_{\textnormal{bad}}\setminus \overline{S}}}R\Gamma\left(K^{\overline{S}_{\textnormal{bad}}},\pi_{(c)}^{G}(\mathbf{Z}/\ell^m\mathbf{Z})\right)\in D^+_{\textnormal{sm}}(\widetilde{G}_{\overline{S}}\times \widetilde{K}_{\overline{S}_{\textnormal{bad}}\setminus \overline{S}} ,\widetilde{\mathbf{T}}^{\overline{S}_{\textnormal{bad}}}_{\widetilde{K}}).
    \end{equation*}
Taking $\widetilde{K}_{\overline{S}_{\textnormal{bad}}\setminus \overline{S}}$-invariants, we obtain
\begin{equation}\label{eq_uglyformula1}
    R\Gamma\Bigg(K_{\overline{S}_{\textnormal{bad}}\setminus \overline{S}},R\Gamma\bigg(\widetilde{K}_{U,\overline{S}_{\textnormal{bad}}\setminus \overline{S}},\textnormal{Inf}_{G^{\overline{S}}}^{P^{\overline{S}}}\Big(\textnormal{Ind}_{P_{\overline{S}}}^{\widetilde{G}_{\overline{S}}}R\Gamma\left(K^{\overline{S}_{\textnormal{bad}}},\pi_{(c)}^G(\mathbf{Z}/\ell^m\mathbf{Z})\right)\Big)\bigg)\Bigg).
\end{equation}
We claim that this can be identified with
    \begin{equation}\label{eq_uglyformula2}
        \textnormal{Ind}_{P_{\overline{S}}}^{\widetilde{G}_{\overline{S}}}R\Gamma\left(K^{\overline{S}},\pi_{(c)}^G(\mathbf{Z}/\ell^m\mathbf{Z})\otimes_{\mathbf{Z}/\ell^m\mathbf{Z}}^{\mathbf{L}} R\Gamma(\widetilde{K}_{U,\overline{S}_{\textnormal{bad}}\setminus \overline{S}},\mathbf{Z}/\ell^m\mathbf{Z})\right).
    \end{equation}
To see this, write $T=\overline{S}_{\textnormal{bad}}\setminus \overline{S}$ and $\pi:=\textnormal{Ind}_{P_{\overline{S}}}^{\widetilde{G}_{\overline{S}}}R\Gamma\left(K^{\overline{S}_{\textnormal{bad}}},\pi_{(c)}^G(\mathbf{Z}/\ell^m\mathbf{Z})\right)\in D^+_{\textnormal{sm}}(\widetilde{G}_{\overline{S}}\times G_{T},\mathbf{Z}/\ell^m\mathbf{Z})$. Note that $\pi$ has bounded Tor-dimension (cf. \cite{CN23}, Remark 2.1.10) and so it admits a bounded $\mathbf{Z}/\ell^m\mathbf{Z}$-flat resolution $\mathcal{F}^{\bullet}\to \pi$. Let $\mathbf{Z}/\ell^m\mathbf{Z}[0]\to \mathcal{I}^{\bullet}$ be an injective resolution of the trivial representation of $\widetilde{G}_{\overline{S}}\times P_{T}$. Then $\textnormal{Inf}_{G_T}^{P_T}\mathcal{F}^{\bullet}\otimes_{\mathbf{Z}/\ell^m\mathbf{Z}}\mathcal{I}^{\bullet}$ is a resolution of $\textnormal{Inf}_{G_T}^{P_T}(\pi)$ that is bounded from below. Moreover, by a theorem of Lazard, the members of the complex are direct limits of injective $K_U$-representations. Therefore, they are injectives themselves by \cite{Eme10b}, Proposition 2.1.3. In particular, we obtain that \ref{eq_uglyformula1} is computed by
\begin{equation*}
    R\Gamma\Bigg(K_{\overline{S}_{\textnormal{bad}}\setminus \overline{S}},\Gamma\bigg(\widetilde{K}_{U,\overline{S}_{\textnormal{bad}}\setminus \overline{S}},\textnormal{Inf}_{G_T}^{P_T}\mathcal{F}^{\bullet}\otimes_{\mathbf{Z}/\ell^m\mathbf{Z}}\mathcal{I}^{\bullet}\bigg)\Bigg)=
\end{equation*}
\begin{equation*}
    R\Gamma\Bigg(K_{\overline{S}_{\textnormal{bad}}\setminus \overline{S}},\mathcal{F}^{\bullet}\otimes_{\mathbf{Z}/\ell^m\mathbf{Z}}\Gamma\bigg(\widetilde{K}_{U,\overline{S}_{\textnormal{bad}}\setminus \overline{S}},\mathcal{I}^{\bullet}\bigg)\Bigg)
\end{equation*}
which is exactly \ref{eq_uglyformula2}.
    
By Lemma~\ref{DirectSummandLemma} and the Künneth formula for group cohomology complexes, the complexes \ref{eq_uglyformula2} admit
    \begin{equation*}
        \textnormal{Ind}_{P_{\overline{S}}}^{\widetilde{G}_{\overline{S}}}R\Gamma\left(K^{\overline{S}},\pi_{(c)}^G(\mathbf{Z}/\ell^m\mathbf{Z})\right)
    \end{equation*}
    as $\widetilde{\mathbf{T}}^{\overline{S}_{\textnormal{bad}}}_{\widetilde{K}}$-equivariant direct summands in $D^+_{\textnormal{sm}}(\widetilde{G}_{\overline{S}},\mathbf{Z}/\ell^m\mathbf{Z})$, compatible with the maps $i_c$ forgetting the support.
\end{proof}
\begin{Lemma}\label{DirectSummandLemma}
    Let $\Bar{v}$ be a finite place of $F^+$ and $K_P\leq P(\mathcal{O}_{F^+_{\Bar{v}}})$ be a compact open subgroup admitting a semidirect product decomposition  $K_G\rtimes K_U$ with $K_G=K_P\cap G(F^{+}_{\Bar{v}})$, and $K_U=K_P\cap U(F^+_{\Bar{v}})$. Then the complex
    \begin{equation*}
        R\Gamma(K_U, \mathbf{Z}/\ell^m\mathbf{Z})
    \end{equation*}
    in $D^+_{\textnormal{sm}}(K_G,\mathbf{Z}/\ell^m\mathbf{Z})$
    admits the trivial representation $\mathbf{Z}/\ell^m\mathbf{Z}$ in degree $0$ as a direct summand.
\end{Lemma}
\begin{proof}
    Consider the map of functors $\alpha:\Gamma(K_U,-)\to \textnormal{Res}_{K_G}(-)$ such that, for $\pi\in\textnormal{Mod}_{\textnormal{sm}}(K_P,\mathbf{Z}/\ell^m\mathbf{Z})$, $\alpha(\pi):\pi^{K_U}\hookrightarrow \textnormal{Res}_{K_G}(\pi)$ is the canonical inclusion. If we write $\beta(\pi):\Gamma(K_U,\pi)[0]\to R\Gamma(K_U,\pi)$ for the morphism in $D^+_{\textnormal{sm}}(K_G,\mathbf{Z}/\ell^m\mathbf{Z})$ provided by the universal property of derived functors, we have $R\alpha(\pi)\circ \beta(\pi) =\alpha(\pi)$. Applying this discussion to $\pi=\mathbf{Z}/\ell^m\mathbf{Z}$ we obtain that the composition
    \begin{equation*}
        \mathbf{Z}/\ell^m\mathbf{Z}[0]\xrightarrow{\beta(\mathbf{Z}/\ell^m\mathbf{Z})}R\Gamma(K_U, \mathbf{Z}/\ell^m\mathbf{Z})\xrightarrow{R\alpha(\mathbf{Z}/\ell^m\mathbf{Z})}\mathbf{Z}/\ell^m\mathbf{Z}[0]
    \end{equation*}
    is the identity.
\end{proof} 
\begin{Prop}\label{TransferProposition}
    Consider the setup of Lemma~\ref{ImplicationOfBSLemma}. There is a $\widetilde{\mathbf{T}}^{(\overline{S}_{\textnormal{bad}},\overline{S}_{\textnormal{ram}}),+}_{\widetilde{K}}$-equivariant diagram
        \begin{equation}\label{transferdiagram}
            \begin{tikzcd}
	& {R\Gamma(\partial\widetilde{X}_{\widetilde{K}},\mathbf{Z}/\ell^m\mathbf{Z})} \\
	{R\Gamma_c(\widetilde{X}^P_{\widetilde{K}},\mathbf{Z}/\ell^m\mathbf{Z})} && {R\Gamma(\widetilde{X}^P_{\widetilde{K}},\mathbf{Z}/\ell^m\mathbf{Z})} \\
	{R\Gamma_c(X_{K},\mathbf{Z}/\ell^m\mathbf{Z})} && {R\Gamma(X_{K},\mathbf{Z}/\ell^m\mathbf{Z})}
	\arrow[from=1-2, to=2-3]
	\arrow[from=2-1, to=1-2]
	\arrow["i_{c,m}^P",from=2-1, to=2-3]
	\arrow["\alpha", from=2-1, to=3-1]
	\arrow["\beta", from=2-3, to=3-3]
	\arrow["{i_{c,m}}", from=3-1, to=3-3]
\end{tikzcd}
        \end{equation}
in $D^+(\mathbf{Z}/\ell^m\mathbf{Z})$  where we act through the map $\mathcal{S}_{\overline{S}_{\textnormal{ram}}}^+$ on the members of the bottom row. Moreover, the maps $\alpha$ and $\beta$ are split in $D^+(\mathbf{Z}/\ell^m\mathbf{Z})$.

In particular, the map $\mathcal{S}_{\overline{S}_{\textnormal{ram}}}$ descends to a surjection 
\begin{equation*}
    \overline{\mathcal{S}}_{\partial}:\widetilde{\mathbf{T}}^{(\overline{S}_{\textnormal{bad}},\overline{S}_{\textnormal{ram}})}_{\partial}(\widetilde{K},m)\twoheadrightarrow \mathbf{T}_!^{(\overline{S}_{\textnormal{bad}},\overline{S}_{\textnormal{ram}})}(K,m).
\end{equation*}
\end{Prop}
\begin{proof}
 The open Borel--Serre stratum $\widetilde{X}_{\widetilde{K}}^{P}\subset \partial \widetilde{X}_{\widetilde{K}}$ induces $\widetilde{\mathbf{T}}^{(\overline{S}_{\textnormal{bad}},\overline{S}_{\textnormal{ram}}),+}_{\widetilde{K}}$-equivariant maps $R\Gamma_c(\widetilde{X}^P_{\widetilde{K}},\mathbf{Z}/\ell^m\mathbf{Z})\to R\Gamma(\partial\widetilde{X}_{\widetilde{L}},\mathbf{Z}/\ell^m\mathbf{Z})\to R\Gamma(\widetilde{X}^P_{\widetilde{K}},\mathbf{Z}/\ell^m\mathbf{Z})$ in $D^+(\mathbf{Z}/\ell^m\mathbf{Z})$ with the composition being the natural map that forgets the support. This gives the top triangle of the diagram.

To obtain the bottom square, we apply Lemma~\ref{ImplicationOfBSLemma} with $\overline{S}=\overline{S}_{\textnormal{ram}}$. By applying $\widetilde{K}_{\overline{S}_{\textnormal{ram}}}$-invariants to the obtained diagram, we have to show the existence of Hecke equivariant maps
\begin{equation*}
    R\Gamma\left(\widetilde{K}_{\overline{S}_{\textnormal{ram}}},\textnormal{Ind}_{P_{\overline{S}_{\textnormal{ram}}}}^{\widetilde{G}_{\overline{S}_{\textnormal{ram}}}}\left(R\Gamma(K^{\overline{S}_{\textnormal{ram}}},\pi_{(c)}^G(\mathbf{Z}/\ell^m\mathbf{Z}))\right)\right)\to 
\end{equation*}
\begin{equation*}
    R\Gamma_{(c)}(X_K,\mathbf{Z}/\ell^m\mathbf{Z})
\end{equation*}
in $D^+(\mathbf{Z}/\ell^m\mathbf{Z})$, admitting a splitting and compatible with the maps forgetting the support. Given any representation $\pi\in \textnormal{Mod}_{\textnormal{sm}}(\mathbf{Z}/\ell^m\mathbf{Z}[G_{\overline{S}_{\textnormal{ram}}}])$, the association $\textnormal{Ind}_{P_{\overline{S}_{\textnormal{ram}}}}^{\widetilde{G}_{\overline{S}_{\textnormal{ram}}}}\pi\to \pi$, $f\mapsto f(1)$ induces a map
    \begin{equation*}
\Gamma(\widetilde{K}_{\overline{S}_{\textnormal{ram}}},\textnormal{Ind}_{P_{\overline{S}_{\textnormal{ram}}}}^{\widetilde{G}_{\overline{S}_{\textnormal{ram}}}}\pi)\to \Gamma(K_{\overline{S}_{\textnormal{ram}}},\pi)
    \end{equation*}
    and, using that $\widetilde{K}_{\overline{S}_{\textnormal{ram}}}$ admits an Iwahori decomposition, it is a simple exercise to check that this map is $\mathcal{H}(G_{\overline{S}_{\textnormal{ram}}}^+,K_{\overline{S}_{\textnormal{ram}}})$-equivariant when we act on the former via $t:\mathcal{H}(G_{\overline{S}_{\textnormal{ram}}}^+,K_{\overline{S}_{\textnormal{ram}}})\to \mathcal{H}(\widetilde{G}_{\overline{S}_{\textnormal{ram}}},\widetilde{K}_{\overline{S}_{\textnormal{ram}}})$. Moreover, by Mackey's formula, we see that it admits a splitting (functorial in $\pi$) as a $\mathbf{Z}/\ell^m\mathbf{Z}$-module (cf. \cite{ACC23}, Lemma 2.1.14). As $\textnormal{Ind}_{P_{\overline{S}_{\textnormal{ram}}}}^{\widetilde{G}_{\overline{S}_{\textnormal{ram}}}}(-): \textnormal{Mod}_{\textnormal{sm}}(\mathbf{Z}/\ell^m\mathbf{Z}[G_{\overline{S}_{\textnormal{ram}}}])\to \textnormal{Mod}_{\textnormal{sm}}(\mathbf{Z}/\ell^m\mathbf{Z}[\widetilde{G}_{\overline{S}_{\textnormal{ram}}}])$ admits an exact left adjoint, it preserves injectives and so, for $\pi\in D^+_{\textnormal{sm}}(G_{\overline{S}_{\textnormal{ram}}},\mathbf{Z}/\ell^m\mathbf{Z})$ we get a morphism
    \begin{equation*}
        R\Gamma(\widetilde{K}_{\overline{S}_{\textnormal{ram}}},\textnormal{Ind}_{P_{\overline{S}_{\textnormal{ram}}}}^{\widetilde{G}_{\overline{S}_{\textnormal{ram}}}}\pi)\cong R(\Gamma(\widetilde{K}_{\overline{S}_{\textnormal{ram}}},-)\circ\textnormal{Ind}_{P_{\overline{S}_{\textnormal{ram}}}}^{\widetilde{G}_{\overline{S}_{\textnormal{ram}}}}(-))(\pi)\to R\Gamma(K_{\overline{S}_{\textnormal{ram}}},\pi)
    \end{equation*}
    in $D^+(\mathcal{H}(G_{\overline{S}_{\textnormal{ram}}},K_{\overline{S}_{\textnormal{ram}}})\otimes_\mathbf{Z}\mathbf{Z}_{\ell})$, natural in $\pi$. Applying this discussion to 
    \begin{equation*}
        R\Gamma(K^{\overline{S}_{\textnormal{bad}}},\pi_{(c)}^G(\mathbf{Z}/\ell^m\mathbf{Z}))\in D^+_{\textnormal{sm}}(G_{\overline{S}_{\textnormal{ram}}},\mathbf{Z}/\ell^m\mathbf{Z})
    \end{equation*} gives the desired maps.

    To see the final claim, we have to show that if $t\in \textnormal{Ann}_{\widetilde{\mathbf{T}}^{(\overline{S}_{\textnormal{bad}},\overline{S}_{\textnormal{ram}}),+}_{\widetilde{K}}}(i_{c,\widetilde{K},m}^P)$, then $t\in\textnormal{Ann}_{\widetilde{\mathbf{T}}^{(\overline{S}_{\textnormal{bad}},\overline{S}_{\textnormal{ram}}),+}_{\widetilde{K}}}(i_{c,K,m})$ i.e., $i_{c,K,m}\circ t=0$ in $D^+(\mathbf{Z}/\ell^m\mathbf{Z})$. By \ref{transferdiagram} and assumption, we have $i_{c,K,m}\circ t\circ \alpha =\beta\circ i_{c,\widetilde{K},m}^P\circ t=0$. However, if we write $\gamma$ for the map splitting $\alpha$, we have 
    \begin{equation*}
    i_{c,K,m}\circ t\circ \alpha\circ \gamma= i_{c,K,m}\circ t
    \end{equation*}
    in $\textnormal{Hom}_{D^+(\mathbf{Z}/\ell^m\mathbf{Z})}(R\Gamma_c(X_K,\mathbf{Z}/\ell^m\mathbf{Z}),R\Gamma(X_K,\mathbf{Z}/\ell^m\mathbf{Z}))$.
\end{proof}

\subsubsection{The Borel--Serre boundary for $G$}\label{sec_BSboundary_for_G}
To describe the cohomology of the Borel--Serre boundary for $G$, it will be more convenient for us to work with the prime-to-$\overline{S}$ Bernstein centres instead of $\mathbf{T}^{\overline{S}}_K$. Therefore, for a finite set $\overline{S}$ of finite places of $F^+$ containing all $\ell$-adic places, we set
\begin{equation*}
    \mathfrak{Z}_G^{\overline{S}}:=\otimes_{\Bar{v}\notin \overline{S},\mathbf{Z}_{\ell}}'\mathfrak{Z}_{G(F^+_{\Bar{v}})}.
\end{equation*}
Note that we have a map
\begin{equation*}
    t_{G,K}^{\overline{S}}:\mathfrak{Z}_G^{\overline{S}}\to \mathbf{T}_K^{\overline{S}}
\end{equation*}
of $\mathbf{Z}_{\ell}$-algebras
given by the product of the maps $t_{G(F^{+}_{\Bar{v}}),K_{\Bar{v}}}$, $\Bar{v}\notin \overline{S}$.

Given an $F^+$-rational standard parabolic subgroup $Q=MN\subset G$ with its standard Levi factorisation, we define $\mathfrak{Z}_M^{\overline{S}}$ and $t_{M,K_M}^{\overline{S}}$ in analogous manner. We then further have a map
\begin{equation*}
    I^{\overline{S}}:\mathfrak{Z}_G^{\overline{S}}\to \mathfrak{Z}_M^{\overline{S}}
\end{equation*}
of $\mathbf{Z}_{\ell}$-algebras
given by the product of the maps $I:\mathfrak{Z}_{G(F^+_{\Bar{v}})}\to \mathfrak{Z}_{M(F^{+}_{\Bar{v}})}$ from Theorem~\ref{TheMapI}.

Given another (possibly empty) finite set $\overline{T}$ of finite places of $F^+$ with $\overline{S}\cap \overline{T}=\emptyset$, we have natural algebra maps
\begin{equation*}
    \mathfrak{Z}_G^{\overline{S}}\to \textnormal{End}_{D^+_{\textnormal{sm}}( G_{\overline{T}},\mathbf{Z}/\ell^m\mathbf{Z})}\left(R\Gamma\left(K^{\overline{T}},\pi_{(c)}^G(\mathbf{Z}/\ell^m\mathbf{Z})\right)\right),
\end{equation*}
\begin{equation*}
    \mathfrak{Z}_G^{\overline{S}}\to \textnormal{End}_{D^+_{\textnormal{sm}}( G_{\overline{T}},\mathbf{Z}/\ell^m\mathbf{Z})}\left(R\Gamma\left(K^{\overline{T}},\pi_{\partial}^G(\mathbf{Z}/\ell^m\mathbf{Z})\right)\right),
\end{equation*}
\begin{equation*}
    \mathfrak{Z}_G^{\overline{S}}\to \textnormal{End}_{D^+_{\textnormal{sm}}( G_{\overline{T}},\mathbf{Z}/\ell^m\mathbf{Z})}\left(R\Gamma\left(K^{\overline{T}},R\Gamma(\overline{\mathfrak{X}}_G^Q,(j_!)\mathbf{Z}/\ell^m\mathbf{Z})\right)\right)\textnormal{, and}
\end{equation*}
\begin{equation*}
    \mathfrak{Z}_M^{\overline{S}}\to \textnormal{End}_{D^+_{\textnormal{sm}}( M_{\overline{T}},\mathbf{Z}/\ell^m\mathbf{Z})}\left(R\Gamma\left(K^{\overline{T}}_M,\pi_{(c)}^M(\mathbf{Z}/\ell^m\mathbf{Z})\right)\right)
\end{equation*}
with images denoted by $\mathfrak{Z}_{G,(c)}^{\overline{S}}(K^{\overline{T}},m)$, $\mathfrak{Z}_{G,\partial}^{\overline{S}}(K^{\overline{T}},m)$, $\mathfrak{Z}_{G,(c)}^{\overline{S},Q}(K^{\overline{T}},m)$ and $\mathfrak{Z}_{M,(c)}^{\overline{S}}(K^{\overline{T}}_M, m)$, respectively.
Finally, we define
\begin{equation*}
    \mathfrak{Z}_{G,(c)}^{\overline{S}}(K_M^{\overline{T}},m):=\im\left(\mathfrak{Z}_G^{\overline{S}}\xrightarrow{I^{\overline{S}}}\mathfrak{Z}_M^{\overline{S}}\twoheadrightarrow\mathfrak{Z}_{M,(c)}^{\overline{S}}(K^{\overline{T}}_M, m)\right).
\end{equation*}

\begin{Lemma}\label{BSInductionLemma}
    Let $\overline{S}$ be a finite set of finite places of $F^+$ with a decomposition $\overline{S}=\overline{S}_1\coprod\overline{S}_2$ and assume that $\overline{S}_2$ contains all $\ell$-adic places of $F^+$. Let $K^{\overline{S}_1}\leq G(\widehat{\mathcal{O}}_{F^+}^{\overline{S}_1})$ be a good subgroup such that $K^{\overline{S}}=G(\widehat{\mathcal{O}}_{F^+}^{\overline{S}})$ and set $K_M^{\overline{S}}:=K^{\overline{S}}\cap M(\mathbf{A}_{F^+}^{\infty})$. There are nilpotent ideals $J_{(c)}^Q\subset \mathfrak{Z}_{G,(c)}^{\overline{S}_2,Q}(K^{\overline{S}_1},m)$ with nilpotence degree bounded in terms of $n$ and $[F^+:\mathbf{Q}]$ and commutative diagrams
    \begin{equation*}
    \begin{tikzcd}
	&& {\mathfrak{Z}_{G,(c)}^{\overline{S}_2,Q}(K^{\overline{S}_1},m)}/J_{(c)}^Q \\
	{\mathfrak{Z}_G^{\overline{S}_2}} \\
	{} && {\mathfrak{Z}^{\overline{S}_2}_{G,(c)}(K_M^{\overline{S}}K'_{M,\overline{S}_2},m)}.
	\arrow[two heads, from=2-1, to=1-3]
	\arrow[two heads, from=2-1, to=3-3]
	\arrow[dashed, two heads, from=3-3, to=1-3]
\end{tikzcd}
\end{equation*}
    where $K'_{M,\overline{S}_2}\leq K_{M,\overline{S}_2}$ is some suitably deep normal compact open subgroup depending on $K_{\overline{S}_2}$ and $m$.
\end{Lemma}
\begin{proof}
    By Lemma~\ref{BorelSerreLemma}, we have $\mathfrak{Z}^{\overline{S}_2}_G$-equivariant isomorphisms
    \begin{equation*}
        R\Gamma(K^{\overline{S}_1},R\Gamma(\overline{\mathfrak{X}}_G^Q,(j_!)\mathbf{Z}/\ell^m\mathbf{Z}))\cong
    \end{equation*}
    \begin{equation*}
R\Gamma(K^{\overline{S}_1},\textnormal{Ind}_{Q(\mathbf{A}_{F^+}^{\infty})}^{G(\mathbf{A}_{F^{+}}^{\infty})}R\Gamma(\overline{\mathfrak{X}}_M,(j_!)\mathbf{Z}/\ell^m\mathbf{Z}))\cong
    \end{equation*}
    \begin{equation}\label{Lemma4.4Eqn}
        R\Gamma\left(K_{\overline{S}_2},\textnormal{Ind}_{Q_{\overline{S}}}^{G_{\overline{S}}}R\Gamma\left(K_M^{\overline{S}},R\Gamma(\overline{\mathfrak{X}}_M,(j_!)\mathbf{Z}/\ell^m\mathbf{Z})\right)\right)
    \end{equation}
    in $D^+_{\textnormal{sm}}(G_{\overline{S}_1},\mathbf{Z}/\ell^m\mathbf{Z})$. Here, $\mathfrak{Z}_G^{\overline{S}_2}$ acts on\footnote{We note that the restriction of this action to $\mathfrak{Z}^{\overline{S}}_G$ on the third complex is the same as the spherical action via $t_G^{\overline{S}}$ by Lemma~\ref{BernsteinCentreLemma}.} \ref{Lemma4.4Eqn} by acting on $R\Gamma(\overline{\mathfrak{X}}_M,(j_!)\mathbf{Z}/\ell^m\mathbf{Z})$ via $I^{\overline{S}_2}$. 
    By Mackey's formula, \ref{Lemma4.4Eqn} is further $\mathfrak{Z}_M^{\overline{S}_2}$-equivariantly isomorphic to
    \begin{equation}\label{Lemma4.4Eqn3}
        \bigoplus_g \textnormal{Ind}_{Q_{\overline{S}_1}}^{G_{\overline{S}_1}}R\Gamma\left(gK_{\overline{S}_2}g^{-1}\cap Q_{\overline{S}_2},R\Gamma(K_M^{\overline{S}},R\Gamma(\overline{\mathfrak{X}}_M,(j_!)\mathbf{Z}/\ell^m\mathbf{Z}))\right)
    \end{equation}
    in $D^+_{\textnormal{sm}}(G_{\overline{S}_1},\mathbf{Z}/\ell^m\mathbf{Z})$ where $g$ runs over a set of representatives of the \textit{finite} double coset 
    \begin{equation}\label{Lemma4.4Eqn2}
        (\prod_{\Bar{v}\in \overline{S}_2}Q(\mathcal{O}_{F^+_{\Bar{v}}}))\setminus (\prod_{\Bar{v}\in \overline{S}_2}G(\mathcal{O}_{F^+_{\Bar{v}}}))/K_{\overline{S}_2}.
    \end{equation}
Set $K^g_{\overline{S}_2}:=gK_{\overline{S}_2}g^{-1}\cap Q_{\overline{S}_2}\leq \prod_{\Bar{v}\in \overline{S}_2}Q(\mathcal{O}_{F^+_{\Bar{v}}})$ and choose an integer $c\geq 1$ such that\footnote{Recall that for a linear algebraic group $H$ over the ring of integers of an $\ell$-adic field $L/\mathbf{Q}_{\ell}$, we defined $H^c:=\ker\left(H(\mathcal{O}_L)\to H(\mathcal{O}_L/\varpi_L^c)\right)$.} $Q_{\overline{S}_2}(c,c):=M_{\overline{S}_2}^cN_{\overline{S}_2}^c\leq K^g_{\overline{S}_2}$ for every representative $g$ in \ref{Lemma4.4Eqn2}. For another integer $b\geq c\geq 1$, set $Q_{\overline{S}_2}(b,c):=M^b_{\overline{S}_2}N^c_{\overline{S}_2}\leq K^g_{\overline{S}_2}$ and note that it is a normal subgroup. Arguing as in Lemma \ref{ImplicationOfBSLemma}, we can further rewrite \ref{Lemma4.4Eqn3} (in a $\mathfrak{Z}_M^{\overline{S}_2}$-equivariant way in $D^+_{\textnormal{sm}}(G_{\overline{S}_1},\mathbf{Z}/\ell^m\mathbf{Z})$) as follows
\begin{equation*}
    \bigoplus_g R\Gamma\left(K^g_{\overline{S}_2}/Q_{\overline{S}_2}(b,c),\textnormal{Ind}_{Q_{\overline{S}_1}}^{G_{\overline{S}_1}}R\Gamma\left(K_M^{\overline{S}}M_{\overline{S}_2}^b,\pi_{(c)}^M(\mathbf{Z}/\ell^m\mathbf{Z})\otimes_{\mathbf{Z}/\ell^m\mathbf{Z}}^{\mathbf{L}} R\Gamma(N_{\overline{S}_2}^c,\mathbf{Z}/\ell^m\mathbf{Z})\right)\right).
\end{equation*}

If we choose $b$ so that $b-c$ is large enough (depending only on $m$), then $R\Gamma(N_{\overline{S}_2}^c,\mathbf{Z}/\ell^m\mathbf{Z})\cong \oplus_i H^i(N_{\overline{S}_2}^c,\mathbf{Z}/\ell^m\mathbf{Z})[-i]$ in $D^+_{\textnormal{sm}}(M_{\overline{S}_2}^b,\mathbf{Z}/\ell^m\mathbf{Z})$ with each cohomology group being given by finite direct sums of the trivial representation (cf. \cite{CN23}, Lemma 2.3.17). For such a choice of pairs of integers $b\gg_m c\geq 1$, using the Hochschild--Serre spectral sequence, we obtain a Hecke equivariant filtration of the cohomology of \ref{Lemma4.4Eqn} with subquotients given by Hecke equivariant subquotients of (finite group cohomology of) the cohomology groups
\begin{equation*}
    \textnormal{Ind}_{Q_{\overline{S}_1}}^{G_{\overline{S}_1}}\bigg(\varinjlim_{K_{M,\overline{S}_1}} H^{\ast}_{(c)}(X^M_{K_{M}^{\overline{S}_2}M^b_{\overline{S}_2}},\mathbf{Z}/\ell^m\mathbf{Z})\bigg).
\end{equation*} 
We also note that the lengths of these filtrations are all bounded by the cohomological degree of the ambient spaces, which is no more than $[F^+:\mathbf{Q}](n^2-1)$. In particular, using \cite{ACC23}, Lemma 2.2.3 to obtain a statement for the derived Hecke algebras, we finish the proof of the lemma by setting $K'_{M,\overline{S}_2}:=M^b_{\overline{S}_2}$.
\end{proof}

\subsection{Automorphic Galois representations}\label{AutGaloisReps}
Here we collect the necessary results on the existence of automorphic Galois representations. Throughout the subsection, we make the following assumption.
\begin{assumption}\label{runningassumption}
    The number field $F$ is CM containing an imaginary quadratic subfield.
\end{assumption}

The following result on cohomological cusp forms for $\widetilde{G}$ is proved in \cite{ACC23}, Theorem 2.3.5 and is a direct consequence of reciprocity for conjugate self-dual regular algebraic automorphic representations and the base change result \cite{Shi14} of Shin.

\begin{Th}\label{SelfdualLGC} Assume that $F$ contains an imaginary quadratic field. Let $\xi$ be an irreducible algebraic representation of $\widetilde{G}_{\mathbf{C}}$ and $\widetilde{\pi}$ be a $\xi$-cohomological cuspidal automorphic representation of $\widetilde{G}(\mathbf{A}_{F^+})$. For any choice of field isomorphism $\iota:\overline{\mathbf{Q}}_{\ell}\xrightarrow{\sim}\mathbf{C}$, there is a semisimple continuous Galois representation
\begin{equation*}
    r_{\iota}(\widetilde{\pi}):\textnormal{Gal}(\overline{F}/F)\to\textnormal{GL}_{2n}(\overline{\mathbf{Q}}_{\ell})
\end{equation*}
satisfying the following properties:
\begin{enumerate}
    \item Let $p$ be a prime different from $\ell$ that is unramified in $F$ and above which $\Tilde{\pi}$ is unramified as well. Then, for every place $v| p$ of $F$, $r_{\iota}(\Tilde{\pi})|_{G_{F_v}}$ is unramified and the characteristic polynomial of $r_{\iota}(\textnormal{Frob}_v)$ is given by the image in $\overline{\mathbf{Q}}_{\ell}[X]$ of $\widetilde{P}_v(X)$ evaluated on $(\iota^{-1}\Tilde{\pi}_{\Bar{v}})^{\widetilde{G}(\mathcal{O}_{F^+_{\Bar{v}}})}$.
    \item If $p$ is a prime that splits in some imaginary quadratic subfield $F_0\subset F$, then for every place $\Bar{v}|p$ of $F^+$ and a place $v|\bar{v}$ of $F$, there is an isomorphism
    \begin{equation*}
        \textnormal{WD}(r_{\iota}(\Tilde{\pi})|_{G_{F_v}})^{F-ss}\cong \textnormal{rec}^T(\Tilde{\pi}_{\Bar{v}}\circ \iota_v).
    \end{equation*}
\end{enumerate}
\end{Th}

The following is a direct consequence of a theorem of Scholze (cf. \cite{Sch15}, Theorem IV.3.1) and is going to be our source of congruences between Eisenstein series coming from $G$ and cusp forms for $\widetilde{G}$.

\begin{Th}[Scholze]\label{EisensteinCongCusp}
    Let $\iota:\overline{\mathbf{Q}}_{\ell}\xrightarrow{\sim}\mathbf{C}$ be a field isomorphism, $m\geq 1$ be an integer, $(\overline{S}_{\textnormal{bad}},\overline{S}_{\textnormal{ram}})$ be an allowable pair, and $ \widetilde{K}\leq \widetilde{G}(\mathbf{A}_{F^+}^{\infty})$ be a $(\overline{S}_{\textnormal{bad}},\overline{S}_{\textnormal{ram}})$-parahoric subgroup. There is an integer $N$ depending only on $n$ and $[F^+:\mathbf{Q}]$ and an ideal $\widetilde{J}_{\partial}^{(\overline{S}_{\textnormal{bad}},\overline{S}_{\textnormal{ram}})}\leq \widetilde{\mathbf{T}}_{\partial}^{(\overline{S}_{\textnormal{bad}},\overline{S}_{\textnormal{ram}})}(\widetilde{K},m)$ with $(\widetilde{J}_{\partial}^{(\overline{S}_{\textnormal{bad}},\overline{S}_{\textnormal{ram}})})^{N}=0$ such that the following is satisfied. The natural surjection
    \begin{equation}\label{NaturalSurjection}
\widetilde{\mathbf{T}}^{(\overline{S}_{\textnormal{bad}},\overline{S}_{\textnormal{ram}})}_{\widetilde{K}}\twoheadrightarrow \widetilde{\mathbf{T}}^{(\overline{S}_{\textnormal{bad}},\overline{S}_{\textnormal{ram}})}_{\partial}(\widetilde{K},m)/\widetilde{J}_{\partial}^{(\overline{S}_{\textnormal{bad}},\overline{S}_{\textnormal{ram}})}
    \end{equation}
    factors through a surjection
    \begin{equation*}
    \widetilde{\mathbf{T}}_{\widetilde{K}}^{(\overline{S}_{\textnormal{bad}},\overline{S}_{\textnormal{ram}})}\twoheadrightarrow \widetilde{A}_{\widetilde{K},m,\partial}^{(\overline{S}_{\textnormal{bad}},\overline{S}_{\textnormal{ram}})}
    \end{equation*}
    where $\widetilde{A}_{\widetilde{K},m,\partial}^{(\overline{S}_{\textnormal{bad}},\overline{S}_{\textnormal{ram}})}$ is a $\widetilde{\mathbf{T}}_{\widetilde{K}}^{(\overline{S}_{\textnormal{bad}},\overline{S}_{\textnormal{ram}})}$-algebra, finite flat over $\mathbf{Z}_{\ell}$, of the form
    \begin{equation*}
        \textnormal{im}(\widetilde{\mathbf{T}}_{\widetilde{K}}^{(\overline{S}_{\textnormal{bad}},\overline{S}_{\textnormal{ram}})}\to \textnormal{End}_{\overline{\mathbf{Q}}_{\ell}}(\oplus_{\widetilde{\pi},\widetilde{K}^{\circ}_{\ell}}(\iota^{-1}\widetilde{\pi}^{\infty})^{\widetilde{K}^{\ell}\widetilde{K}^{\circ}_{\ell}})),
    \end{equation*}
with the direct sum running over some finite multiset of pairs $(\widetilde{\pi},\widetilde{K}_{\ell}^{\circ})$ of compact open subgroups $\widetilde{K}_{\ell}^{\circ}\leq \prod_{\Bar{v}|\ell}\widetilde{G}(F^+_{\Bar{v}})$ and cohomological cuspidal automorphic representations $\widetilde{\pi}$ of $\widetilde{G}(\mathbf{A}_{F^+})$ of level $\widetilde{K}^{\ell}\widetilde{K}_{\ell}^{\circ}$.
\end{Th}
\begin{proof}
    The analogous statement for compactly supported cohomology follows from \cite{Sch15}, Theorem IV.3.1, the argument of \textit{loc. cit.} Corollary V.1.11 and \cite{HLTT16}, Lemma 5.11. By Poincar\'e duality, we obtain the statement for cohomology without compact support. In particular, we deduce the theorem using the excision long exact sequence.
\end{proof}

Finally, we recall the main theorem of \cite{Sch15} on the existence of automorphic determinants for $G$.

\begin{Th}[Scholze]\label{ScholzeTHM}
    Let $m\geq 1$ be an integer, $(\overline{S}_{\textnormal{bad}},\overline{S}_{\textnormal{ram}})$ be an allowable pair, $K\leq G(\mathbf{A}_{F^+}^{\infty})$ be a good subgroup satisfying $K^{\overline{S}_{\textnormal{bad}}}=G(\widehat{\mathcal{O}}_{F^+}^{\overline{S}_{\textnormal{bad}}})$.
    
    %Then there exists an ideal $J_!^{\overline{S}_{\textnormal{avoid}}}\leq \mathbf{T}_!^{\overline{S}_{\textnormal{avoid}}}(K,L,m)$ such that $(J^{\overline{S}_{\textnormal{avoid}}}_!)^4=0$ and an $n$-dimensional continuous determinant
    %\begin{equation*}
    %    D_!:G_{F,S_{\textnormal{bad}}}\to \mathbf{T}_!^{\overline{S}_{\textnormal{avoid}}}(K,L,m)/J_!^{\overline{S}_{\textnormal{avoid}}}
    %\end{equation*}
    %with $D_!(1-X\textnormal{Frob}_v)=P_v(X)$ for every $v\notin S_{\textnormal{bad}}$.

    For $?\in\{\emptyset, c,!\}$, there exists an integer $N_?$ depending only on $n$ and $[F^+:\mathbf{Q}]$, an ideal $J_{?}^{\overline{S}_{\textnormal{avoid}}}\leq \mathbf{T}_{?}^{\overline{S}_{\textnormal{avoid}}}(K,m)$ with $(J_{?}^{\overline{S}_{\textnormal{avoid}}})^{N_{?}}=0$ and a continuous determinant
    \begin{equation*}
        D_{?}:G_{F,S_{\textnormal{bad}}}\to \mathbf{T}_{?}^{\overline{S}_{\textnormal{avoid}}}(K,m)/J_{?}^{\overline{S}_{\textnormal{avoid}}}
    \end{equation*}
    with $D_{?}(1-X\textnormal{Frob}_v)=P_v(X)$ for every $v\notin S_{\textnormal{bad}}$.
\end{Th}

\section{Local-global compatibility at $p\neq \ell$}
In this section we prove our main result on local-global compatibility. After formulating the theorem using interpolation of the semisimple local Langlands correspondence, we proceed with the proof in several steps. First we verify compatibility with the mod $\ell$ semisimple local Langlands correspondence of Vigneras. Then we treat the case of interior cohomology using the Borel--Serre boundary of $\textnormal{Res}_{F^+/\mathbf{Q}}\widetilde{G}$-locally symmetric spaces. Finally, we handle boundary cohomology of the $\textnormal{Res}_{F/\mathbf{Q}}\textnormal{GL}_n$-locally symmetric space by induction on $n\geq 1$.

\subsection{Formulation of local-global compatibility}
We continue with the setup of \S\ref{AutGaloisReps}. In particular, we remind the reader of our running assumption \ref{runningassumption} about $F$ containing an imaginary quadratic field.

Let $(\overline{S}_{\textnormal{bad}},\overline{S}_{\textnormal{ram}})$ be an allowable pair of finite sets of finite places of $F^+$ in the sense of Definition~\ref{DefinitionOnPlaces} and denote by $(S_{\textnormal{bad}},S_{\textnormal{ram}})$ the pair of finite sets of finite places in $F$ above them. In particular, $S_{\textnormal{avoid}}:=S_{\textnormal{bad}}\setminus S_{\textnormal{ram}}$ contains all $\ell$ -adic places of $F$. Let $K\leq G(\mathbf{A}_{F^+}^{\infty})$ be a good subgroup such that $K^{\overline{S}_{\textnormal{bad}}}=G(\widehat{\mathcal{O}}_{F^+}^{\overline{S}_{\textnormal{bad}}})$. For every integer $m\geq 1$ and $?\in \{\emptyset,c,!\}$, Theorem~\ref{ScholzeTHM} provides a nilpotent ideal $J_?^{\overline{S}_{\textnormal{avoid}}}\leq\mathbf{T}_?^{\overline{S}_{\textnormal{avoid}}}(K,m)$ and a continuous group determinant
\begin{equation*}
    D_{?}:G_{F,S_{\textnormal{bad}}}\to \mathbf{T}_?^{\overline{S}_{\textnormal{avoid}}}(K,m)/J_{?}^{\overline{S}_{\textnormal{avoid}}}
\end{equation*}
satisfying semisimple local-global compatibility at $v\notin S_{\textnormal{bad}}$. This can be phrased as follows. By Proposition~\ref{ArtinianPoints}, $D_{?}$ induces a map
\begin{equation*}
    \theta_v^{\textnormal{Gal}}:\mathfrak{R}^{\textnormal{ps}}_{F_v,n}\to \mathbf{T}_{?}^{\overline{S}_{\textnormal{avoid}}}(K,m)/J_{?}^{\overline{S}_{\textnormal{avoid}}}.
\end{equation*}
On the other hand, the integral Bernstein centre at $v$ admits a map
\begin{equation*}
    \mathfrak{Z}_{\textnormal{GL}_n(F_v)}\xrightarrow{t_{\textnormal{GL}_n(F_v),K_v}}\mathbf{T}_{?}^{\overline{S}_{\textnormal{avoid}}}(K,m)/J_{?}^{\overline{S}_{\textnormal{avoid}}}
\end{equation*}
induced by the map
\begin{equation*}
    t_{\textnormal{GL}_n(F_v),K_v}:\mathfrak{Z}_{\textnormal{GL}_n(F_v)}\to Z\left(\mathcal{H}(\textnormal{GL}_n(F_v),K_v)\otimes_{\mathbf{Z}}\mathbf{Z}_{\ell}\right)\left(= \mathcal{H}(\textnormal{GL}_n(F_v),\textnormal{GL}_n(\mathcal{O}_{F_v}))\otimes_{\mathbf{Z}}\mathbf{Z}_{\ell})\right).
\end{equation*}
Therefore, postcomposition with the map $\Phi_{F_v,n}$ interpolating the semisimple local Langlands correspondence yields a map
\begin{equation*}
    \theta_v^{\textnormal{Aut}}:\mathfrak{R}_{F_v,n}^{\textnormal{ps}}\to \mathbf{T}_?^{\overline{S}_{\textnormal{avoid}}}(K,m)/J_{?}^{\overline{S}_{\textnormal{avoid}}}.
\end{equation*}
Then semisimple local-global compatibility at $v$ is equivalent to the assertion that
\begin{equation*}
    \theta_v^{\textnormal{Gal}}=\theta_v^{\textnormal{Aut}}.
\end{equation*}
Indeed, for a lift $\widetilde{\textnormal{Frob}}_v\in W_{F_v}^0$ of the geometric Frobenius, the universal characteristic polynomial $D^{\textnormal{univ}}(1-X\widetilde{\textnormal{Frob}}_v)\in \mathfrak{R}_{F_v,n}^{\textnormal{ps}}[X]$ under $t_{\textnormal{GL}_n(F_v),\textnormal{GL}_n(\mathcal{O}_{F_v})}\circ\Phi_{F_v,n}$ is sent to $P_v(X)\in (\mathcal{H}(\textnormal{GL}_n(F_v),\textnormal{GL}_n(\mathcal{O}_{F_v}))\otimes_{\mathbf{Z}}\mathbf{Z}_{\ell})[X]$.

Our main theorem states that the equality $\theta_v^{\textnormal{Gal}}=\theta_v^{\textnormal{Aut}}$ holds at $v\in S_{\textnormal{ram}}$ as well. More precisely, we have the following.

\begin{Th}\label{MainTHM}
    Let $(\overline{S}_{\textnormal{bad}},\overline{S}_{\textnormal{ram}})$ be an allowable pair of finite sets of finite places of $F^+$. Let $K\leq G(\mathbf{A}_{F^+}^{\infty})$ be a good subgroup such that $K^{\overline{S}_{\textnormal{bad}}}=G(\widehat{\mathcal{O}}_{F^+}^{\overline{S}_{\textnormal{bad}}})$. Let $m\geq 1$ be an integer and $?\in\{\emptyset,c,!\}$. There is an integer $M_?\geq 1$ depending only on $n$ and $[F^+:\mathbf{Q}]$ and an ideal $J_?^{\overline{S}_{\textnormal{avoid}}}\leq I_?^{\overline{S}_{\textnormal{avoid}}}\leq \mathbf{T}^{\overline{S}_{\textnormal{avoid}}}_?(K,m)$ with $(I_?^{\overline{S}_{\textnormal{avoid}}})^{M_?}=0$ such that the following holds. 
    
    The associated continuous group determinant
    \begin{equation*}
D_?:G_{F,S_{\textnormal{bad}}}\to\mathbf{T}^{\overline{S}_{\textnormal{avoid}}}_?(K,m)/I_?^{\overline{S}_{\textnormal{avoid}}}
    \end{equation*}
    of Theorem~\ref{ScholzeTHM} satisfies local-global compatibility at every $ v\notin S_{\textnormal{\textnormal{avoid}}}$ in the following sense.
    \begin{itemize}
        \item If $K_v = \textnormal{GL}_n(\mathcal{O}_{F_v})$, then $D_{?}$ factors through $G_{F,S_{\textnormal{bad}}\setminus\{v\}}$.
        \item The map
    \begin{equation*}
\theta_v^{\textnormal{Gal}}:\mathfrak{R}^{\textnormal{ps}}_{F_v,n}\to \mathbf{T}_{?}^{\overline{S}_{\textnormal{avoid}}}(K,m)/I_?^{\overline{S}_{\textnormal{avoid}}}
    \end{equation*}
    induced by $D_?|_{G_{F_v}}$ coincides with the map
    \begin{equation*}
        \theta_v^{\textnormal{Aut}}:\mathfrak{R}^{\textnormal{ps}}_{F_v,n}\to \mathbf{T}_{?}^{\overline{S}_{\textnormal{avoid}}}(K,m)/I_?^{\overline{S}_{\textnormal{avoid}}}
    \end{equation*}
    induced by the Hecke action at $v$ via $t_{\textnormal{GL}_n(F_v),K_v}\circ\Phi_{F_v,n}$.
    \end{itemize}
\end{Th}

\subsection{ The proof}

\subsubsection{Compatibility with the mod $\ell$ local Langlands correspondence}
We first prove a weaker theorem on local-global compatibility with the mod $\ell$ local Langlands correspondence of Vigneras. Let $\mathfrak{m}\leq \mathbf{T}^{\overline{S}_{\textnormal{bad}}}_K$ be a maximal ideal in the support of $R\Gamma_{(c)}(X_K,\mathbf{Z}/\ell^m\mathbf{Z})$ and write $\overline{D}_{\mathfrak{m}}:G_{F,S_{\textnormal{bad}}}\to\overline{\mathbf{F}}_{\ell}$ for the corresponding group determinant. For $v\in S_{\textnormal{ram}}$, the restriction $\overline{D}_{\mathfrak{m}}|_{G_{F_v}}$ corresponds to an $\overline{\mathbf{F}}_{\ell}$-point of $(\mathfrak{R}_{F_v,n})^{\textnormal{GL}_n}$ and consequently, under $\Psi_{F_v,n}$, to a maximal ideal $\mathfrak{m}_v\leq \mathfrak{Z}_{\textnormal{GL}_n(F_v)}$. We define the corresponding maximal ideal $\mathfrak{m}_{\Bar{v}}:=(\mathfrak{m}_v,\mathfrak{m}_{v^c})\leq  \mathfrak{Z}_{G(F^+_{\Bar{v}})}\cong \mathfrak{Z}_{\textnormal{GL}_n(F_v)}\otimes_{\mathbf{Z}_{\ell}}
 \mathfrak{Z}_{\textnormal{GL}_n(F_{v^c})}$.
\begin{Th}\label{modellLGC}
    Let $\mathfrak{m}_{\Bar{v}}'=(\mathfrak{m}_v',\mathfrak{m}_{v^c}')\leq \mathfrak{Z}_{G(F^+_{\Bar{v}})}\cong\mathfrak{Z}_{\textnormal{GL}_n(F_v)}\otimes_{\mathbf{Z}_{\ell}}
 \mathfrak{Z}_{\textnormal{GL}_n(F_{v^c})}$ be a maximal ideal. If the localisation
 \begin{equation}\label{Localisation}
     (R\Gamma_{(c)}(X_{K^{\Bar{v}}K'_{\Bar{v}}},\mathbf{Z}/\ell^m\mathbf{Z})_{\mathfrak{m}})_{\mathfrak{m}_{\Bar{v}}'}
 \end{equation} 
 is non-trivial for some compact open subgroup $K'_{\Bar{v}}\leq G(\mathcal{O}_{F^+_{\overline{v}}})$, then the maximal ideal $\mathfrak{m}_{\Bar{v}}'$ coincides with $\mathfrak{m}_{\Bar{v}}$.
 
In other words, for every $K'_{\overline{v}}$, localisation along $\mathfrak{m}_{\Bar{v}}$ induces an isomorphism
    \begin{equation*}
        R\Gamma_{(c)}(X_{K^{\Bar{v}}K'_{\Bar{v}}},\mathbf{Z}/\ell^m\mathbf{Z})_{\mathfrak{m}}\xrightarrow{\sim} (R\Gamma_{(c)}(X_{K^{\Bar{v}}K'_{\Bar{v}}},\mathbf{Z}/\ell^m\mathbf{Z})_{\mathfrak{m}})_{\mathfrak{m}_{\Bar{v}}}.
    \end{equation*}
\end{Th}
\begin{proof}
We start by discussing some reduction steps. 
\begin{enumerate}
    \item By passing to cohomology, we see that \ref{Localisation} being non-trivial is equivalent to the localisation $(H^{\ast}_{(c)}(X_{K^{\Bar{v}}K'_{\Bar{v}}},\mathbf{F}_{\ell})_{\mathfrak{m}})_{\mathfrak{m}'_v}$ being non-trivial.
    \item For $?\in\{\emptyset,c,!\}$, set
\begin{equation*}
    H^{?}:=\varinjlim_{K_{\Bar{v}}'}H_?^{\ast}(X_{K^{\Bar{v}}K'_{\Bar{v}}},\overline{\mathbf{F}}_{\ell})\in \textnormal{Mod}_{\textnormal{sm}}(\mathbf{Z}_{\ell}[G(F^+_{\Bar{v}})]),
\end{equation*}
and
\begin{equation*}
    H^{\partial}:=\varinjlim_{K_{\Bar{v}}'}H^{\ast}(\partial X_{K^{\Bar{v}}K'_{\Bar{v}}},\overline{\mathbf{F}}_{\ell})\in \textnormal{Mod}_{\textnormal{sm}}(\mathbf{Z}_{\ell}[G(F^+_{\Bar{v}})]).
\end{equation*}
By Corollary~\ref{EquivConditions}, we can equivalently show that, for $?\in\{\emptyset,c\}$, $(H^{^{?}}_{\mathfrak{m}})_{\mathfrak{m}_{\Bar{v}}'}$ can only be non-trivial for $\mathfrak{m}_{\Bar{v}}'=\mathfrak{m}_{\Bar{v}}$.
\item By the excision long exact sequence for the Borel--Serre boundary, it suffices to prove the claim of the previous point for $?\in\{!,\partial\}$.
\item We can enlarge $\overline{S}_{\textnormal{bad}}$, provided that it is still unconditional in the sense of Definition~\ref{DefinitionOnPlaces}.
\item By a standard argument with the Hochschild--Serre spectral sequence, it suffices to prove the claim after passing to a deeper level subgroup at $\overline{S}_{\textnormal{bad}}\setminus\{\overline{v}\}$.
\item Finally, given an almost everywhere unramified continuous character $\chi:G_F\to \overline{\mathbf{F}}_{\ell}^{\times}$, after possibly enlarging $\overline{S}_{\textnormal{bad}}$ and deepening the level at $\overline{S}_{\textnormal{bad}}$, we can ensure that $\chi_w:=\chi|_{G_{F_w}}$ is trivial on $\textnormal{Art}_{F_w}(\det(K_{w}))$ for every $w\nmid \Bar{v}$. Denote by $\mathfrak{m}(\chi)$ and $\mathfrak{m}_v(\chi_v)$, the corresponding maximal ideals. Using Lemma~\ref{CharTwistLemma}, we see that one can equivalently prove that, 
$((H^{?})_{\mathfrak{m}(\chi)})_{\mathfrak{m}_{\Bar{v}}'}$ is trivial unless $\mathfrak{m}_{\Bar{v}}'=\mathfrak{m}_{\Bar{v}}(\chi_{\Bar{v}})=(\mathfrak{m}_v(\chi_v),\mathfrak{m}_{v^c}(\chi_{v^c}))$.

In particular, if we write $\phi_{\mathfrak{m}_v}$, and $\phi_{\mathfrak{m}_{v^c}}$ for the mod $\ell$ semisimple Langlands parameters associated with $\mathfrak{m}_{v}$, and $\mathfrak{m}_{v^c}$, respectively, after twisting, we can and \textit{do assume} that $\phi_{\mathfrak{m}_v}$ and $\phi_{\mathfrak{m}_{v^c}}^{\vee}(1-2n)$ share no common Jordan--H\"older factors. Indeed, by the Grunwald--Wang theorem (cf. \cite{AT61}, Chapter X, Theorem 5) we can construct $\chi$ with prescribed local behavior at an arbitrary finite set of finite places.
\end{enumerate}

We start with the case $? = !$. Let $\widetilde{K}$ be an $(\overline{S}_{\textnormal{bad}},\overline{S}_{\textnormal{ram}})$-parahoric subgroup satisfying $\widetilde{K}\cap G(\mathbf{A}_{F^+}^{\infty})=K$. Such a level subgroup always exists. Let $\widetilde{\mathfrak{m}}\leq \widetilde{\mathbf{T}}^{\overline{S}_{\textnormal{bad}}}_{\widetilde{K}}$ be the pullback of $\mathfrak{m}$ along the unnormalised Satake transform and let $\widetilde{\mathfrak{m}}_{\Bar{v}}\leq \mathfrak{Z}_{\widetilde{G}(F^+_{\Bar{v}})}$ be the maximal ideal associated with
    \begin{equation*}
        \overline{D}_{\widetilde{\mathfrak{m}}}|_{G_{F_v}}\cong \overline{D}_{\mathfrak{m}}|_{G_{F_v}}\oplus \overline{D}_{\mathfrak{m}}^{\vee,c}(1-2n)|_{G_{F_v}}.
    \end{equation*}
    Note that $\widetilde{\mathfrak{m}}_{\Bar{v}}$ is the pullback of
    \begin{equation*}
        \mathfrak{m}_{\Bar{v}}=(\mathfrak{m}_v,\mathfrak{m}_{v^c})\leq \mathfrak{Z}_{\textnormal{GL}_n(F_v)}\otimes_{\mathbf{Z}_{\ell}}\mathfrak{Z}_{\textnormal{GL}_n(F_{v^c})}=\mathfrak{Z}_{G(F^+_{\Bar{v}})}\cong^{\iota_v}\mathfrak{Z}_{\textnormal{GL}_n(F_v)}\otimes_{\mathbf{Z}_{\ell}}\mathfrak{Z}_{\textnormal{GL}_n(F_v)}
    \end{equation*}
    along the map $I$ from Theorem~\ref{TheMapI}.

We first prove the following lemma.
\begin{Lemma}\label{pullbackcompatibility}
    The pullback $I^{\ast}(\mathfrak{m}'_{\Bar{v}})$ coincides with $\widetilde{\mathfrak{m}}_{\Bar{v}}$.
\end{Lemma}
\begin{proof}
We set
\begin{equation*}
    \widetilde{H}^{\partial}:=\varinjlim_{\widetilde{K}_{\Bar{v}}}H^{\ast}(\partial\widetilde{X}_{\widetilde{K}^{\Bar{v}}\widetilde{K}_{\Bar{v}}},\overline{\mathbf{F}}_{\ell})\in \textnormal{Mod}_{\textnormal{sm}}(\mathbf{Z}_{\ell}[\widetilde{G}(F^+_{\Bar{v}})]).
\end{equation*}
We first prove that the localisation $(\widetilde{H}^{\partial}_{\widetilde{\mathfrak{m}}})_{\widetilde{\mathfrak{m}}_{\Bar{v}}'}$ can only be non-trivial for $\widetilde{\mathfrak{m}}_{\Bar{v}}'=\widetilde{\mathfrak{m}}_{\Bar{v}}$. To see this, note that by Corollary~\ref{EisensteinCongCusp}, the natural map
    \begin{equation*}
\widetilde{\mathbf{T}}^{\overline{S}_{\textnormal{avoid}}}_{\widetilde{K}}\to (\widetilde{\mathbf{T}}_{\partial}^{\overline{S}_{\textnormal{avoid}}}(\widetilde{K},m)/\widetilde{J}_{\partial}^{(\overline{S}_{\textnormal{bad}},\overline{S}_{\textnormal{ram}})})_{\widetilde{\mathfrak{m}}}
    \end{equation*}
     factors through the map
    \begin{equation*}
        \widetilde{\mathbf{T}}^{\overline{S}_{\textnormal{avoid}}}_{\widetilde{K}}\to (\widetilde{A}_{\widetilde{K},m,\partial}^{\overline{S}_{\textnormal{avoid}}})_{\widetilde{\mathfrak{m}}}.
    \end{equation*}
    Moreover, by Theorem~\ref{SelfdualLGC} on local-global compatibility for cusp forms for $\widetilde{G}$, we have an isomorphism
    \begin{equation*}
(\widetilde{A}_{\widetilde{K},m,\partial}^{\overline{S}_{\textnormal{avoid}}})_{\widetilde{\mathfrak{m}}}\cong \left((\widetilde{A}_{\widetilde{K},m,\partial}^{\overline{S}_{\textnormal{avoid}}})_{\widetilde{\mathfrak{m}}}\right)_{\widetilde{\mathfrak{m}}_{\Bar{v}}}.
    \end{equation*}
    Combining this with Corollary~\ref{EquivConditions}, we obtain the claim.

    An application of Lemma~\ref{ImplicationOfBSLemma} with $\overline{S}=\{\Bar{v}\}$ yields a commutative diagram
    \begin{equation*}
        \begin{tikzcd}
	& {(\widetilde{H}^{\partial}_{\widetilde{\mathfrak{m}}})_{\widetilde{\mathfrak{m}}_{\Bar{v}}'}} \\
	{\left((\textnormal{Ind}_{P_{\Bar{v}}}^{\widetilde{G}_{\Bar{v}}}H^c)_{\widetilde{\mathfrak{m}}}\right)_{\widetilde{\mathfrak{m}}_{\Bar{v}}'}} && {\left((\textnormal{Ind}_{P_{\Bar{v}}}^{\widetilde{G}_{\Bar{v}}}H)_{\widetilde{\mathfrak{m}}}\right)_{\widetilde{\mathfrak{m}}_{\Bar{v}}'}}
	\arrow[from=1-2, to=2-3]
	\arrow[from=2-1, to=1-2]
	\arrow["{i_c}"', from=2-1, to=2-3]
\end{tikzcd}
    \end{equation*}
    where the bottom horizontal map is the one induced by forgetting the support.
    Therefore, we see that $\im(i_c)=\left((\textnormal{Ind}_{P_{\Bar{v}}}^{\widetilde{G}_{\Bar{v}}}H^!)_{\widetilde{\mathfrak{m}}}\right)_{\widetilde{\mathfrak{m}}_{\Bar{v}}'}$ can only be non-trivial for $\widetilde{\mathfrak{m}}_{\Bar{v}}'=\widetilde{\mathfrak{m}}_{\Bar{v}}$. In particular, $(\textnormal{Ind}_{P_{\Bar{v}}}^{\widetilde{G}_{\Bar{v}}}H^!_{\mathfrak{m}})_{\widetilde{\mathfrak{m}}_{\Bar{v}}'}$ can only be non-trivial for $\widetilde{\mathfrak{m}}_{\Bar{v}}'=\widetilde{\mathfrak{m}}_{\Bar{v}}$. However, we assumed that $(H^!_{\mathfrak{m}})_{\mathfrak{m}'}$ is non-trivial and so, by Corollary~\ref{ParabolicLemma}, we see that $(\textnormal{Ind}_{P_{\Bar{v}}}^{\widetilde{G}_{\Bar{v}}}H^!_{\mathfrak{m}})_{I^{\ast}(\mathfrak{m}_{\Bar{v}}')}$ is non-trivial, yielding $I^{\ast}(\mathfrak{m}'_{\Bar{v}})=\widetilde{\mathfrak{m}}_{\Bar{v}.}$
\end{proof}

    Consider an almost everywhere unramified continuous character $\chi:G_F\to \overline{\mathbf{F}}_{\ell}^{\times}$ with trivial restriction to $\textnormal{Art}_{F_w}^{-1}( \det(K_w))$ for all finite places $w$ of $F$. A combination of Lemma~\ref{CharTwistLemma} and Lemma~\ref{pullbackcompatibility} shows that
    \begin{equation}\label{twistedideal}
        I^{\ast}(\mathfrak{m}_{\Bar{v}}'(\chi_{\Bar{v}}))=\widetilde{\mathfrak{m}(\chi)}_{\Bar{v}}.
    \end{equation}
    Denote by $\phi_{\mathfrak{m}}:G_{F^+_{\Bar{v}}}\to \textnormal{GL}_n(\overline{\mathbf{F}}_{\ell})$ the semisimple Langlands parameter associated with $\mathfrak{m}\in\{\mathfrak{m}_v,\mathfrak{m}_{v^c},\mathfrak{m}_v',\mathfrak{m}_{v^c}'\}$.
    By \ref{twistedideal} we have
    \begin{equation}\label{twistedequation}
        (\phi_{\mathfrak{m}_v'}\otimes\chi_v)\oplus (\phi_{\mathfrak{m}_{v^c}'}^{\vee}(1-2n)\otimes \chi_{v^c}^{-1})\cong(\phi_{\mathfrak{m}_v}\otimes\chi_v)\oplus (\phi_{\mathfrak{m}_{v^c}}^{\vee}(1-2n)\otimes \chi_{v^c}^{-1}).
    \end{equation}
    
To reach a contradiction, we assume that there is $\phi\in \textnormal{JH}(\phi_{\mathfrak{m}_v'})\cap \textnormal{JH}(\phi_{\mathfrak{m}_{v^c}}^{\vee}(1-2n))$. (We remind the reader that we have arranged at the start of the proof that $\phi_{\mathfrak{m}_v}$ and $\phi_{\mathfrak{m}_{v^c}}^{\vee}(1-2n)$ share no Jordan--H\"older factors and so $\phi\notin \textnormal{JH}(\phi_{\mathfrak{m}_v})$.)

Note that for all but finitely many unramified characters $\psi_v:G_{F_v}\to \overline{\mathbf{F}}_{\ell}^{\times}$, we have
\begin{equation}\label{characterequation}
    \textnormal{JH}(\phi_{\mathfrak{m}_{v^c}}^{\vee}(1-2n)\otimes \psi_v)\cap \textnormal{JH}(\phi_{\mathfrak{m}_{v^c}}^{\vee}(1-2n))=\emptyset.
\end{equation}
Pick such a $\psi_v$ and, using the theorem of Grunwald--Wang
(cf. \cite{AT61}, Chapter X, Theorem 5), construct a continuous character $\chi:G_F\to \overline{\mathbf{F}}_{\ell}^{\times}$ with $\chi_{v}=\psi_v$, and $\chi_{v^c}$ being trivial. After possibly enlarging $\overline{S}_{\textnormal{bad}}$, and deepening the level at $\overline{S}_{\textnormal{bad}}$, we can further assume that $\chi$ is trivial on $\textnormal{Art}_{F_w}^{-1}(\det(K_{w}))$ for all finite places of $F$. Therefore, by \ref{twistedequation} using that $\phi\otimes\psi_v\notin \textnormal{JH}(\phi_{\mathfrak{m}_v}\otimes\psi_v)$, we obtain that $\phi\otimes\psi_v\in \textnormal{JH}(\phi_{\mathfrak{m}_{v^c}}^{\vee}(1-2n))$. However,
we also have $\phi\otimes\psi_v\in \textnormal{JH}(\phi_{\mathfrak{m}_{v^c}}^{\vee}(1-2n)\otimes \psi_v)$ which contradicts \ref{characterequation}. This finishes the proof of the case of interior cohomology.

\vspace{5mm}

To prove the claim for the boundary cohomology and finish the proof of the theorem, we run an induction argument on $n\geq 1$. As for $n=1$ there is no boundary, the theorem is already proved in this case. Assume that $n>1$ and that the theorem has already been established for $1\leq m <n$. By looking at the excision long exact sequence for the stratification of the Borel--Serre boundary, we see that it suffices to prove that the following holds for an arbitrary $F^+$-rational standard parabolic subgroup $Q=MN\leq G$
\begin{itemize}
    \item $((H^{Q}_c)_{\mathfrak{m}})_{\mathfrak{m}_{\Bar{v}}'}\neq 0$ only if $\mathfrak{m}_{\Bar{v}}'=\mathfrak{m}_{\Bar{v}}$
\end{itemize}
where
\begin{equation*}
H^Q_c:=\varinjlim_{K_{\Bar{v}}'}H^{\ast}_c(X_{K^{\Bar{v}}K_{\Bar{v}}'}^Q,\mathbf{F}_{\ell})\in \textnormal{Mod}_{\textnormal{sm}}(\mathbf{Z}_{\ell}[G(F^+_{\Bar{v}})]).
\end{equation*}
By Lemma~\ref{BSInductionLemma}, it suffices to prove that
\begin{itemize}
    \item $((H^{M}_c)_{\mathfrak{m}})_{\mathfrak{m}_{\Bar{v}}'}\neq 0$ only if $\mathfrak{m}_{\Bar{v}}'=\mathfrak{m}_{\Bar{v}}$
\end{itemize}
where
\begin{equation*}
H^M_c:=\varinjlim_{K_{M,\Bar{v}}'}H^{\ast}_c(X_{K^{\Bar{v}}_MK_{M,\Bar{v}}'}^M,\mathbf{F}_{\ell})\in \textnormal{Mod}_{\textnormal{sm}}(\mathbf{Z}_{\ell}[M(F^+_{\Bar{v}})])
\end{equation*}
for some good subgroup $K_M^{\Bar{v}}$ with $K_M^{\overline{S}_{\textnormal{bad}}}=M(\widehat{\mathcal{O}}_{F^+}^{\overline{S}_{\textnormal{bad}}})$.

Let $\mathfrak{m}_1,...,\mathfrak{m}_k\leq \mathfrak{Z}^{\overline{S}_{\textnormal{bad}}}_M$ be the \textit{finite} list of maximal ideals that are supported on $H^M_c$ and satisfy $I^{\overline{S}_{\textnormal{bad}},\ast}(\mathfrak{m}_i)=t_{G,K}^{\overline{S}_{\textnormal{bad}},\ast}(\mathfrak{m})$ and write $\mathfrak{m}_{i,\Bar{v}}\leq \mathfrak{Z}_{M(F^+_{\Bar{v}})}$ for the maximal ideal associated with the pair of determinants $(\overline{D}_{\mathfrak{m}_i}|_{G_{F_v}}, \overline{D}_{\mathfrak{m}_i}|_{G_{F_{v^c}}})$. We then have an inclusion
\begin{equation*}
    (H_c^M)_{\mathfrak{m}}\hookrightarrow \oplus_{i=1}^k(H_c^M)_{\mathfrak{m}_i}\cong \oplus_{i=1}^k\left((H_c^M)_{\mathfrak{m}_i}\right)_{\mathfrak{m}_{i,\Bar{v}}}
\end{equation*}
where we used the induction hypothesis (and the K\"unneth formula) for the last isomorphism. This shows that $\mathfrak{m}_{\Bar{v}}'$ must be of the form $I^{\ast}(\mathfrak{m}_{i,\Bar{v}})$ for some $i=1,...,k$. But this pullback is $\mathfrak{m}_{\Bar{v}}$ for any $i=1,...,k$.
\end{proof}
\subsubsection{Proof of Theorem~\ref{MainTHM} for $?=!$}

We start with a series of reductions.
\begin{enumerate}
    \item We can freely enlarge $\overline{S}_{\textnormal{bad}}$ (and $\overline{S}_{\textnormal{avoid}}$, accordingly). More precisely, let $\overline{S}_{\textnormal{bad}}\subset \overline{S}_{\textnormal{bad}}'$ be so that $(\overline{S}_{\textnormal{ram}},\overline{S}_{\textnormal{bad}}')$ is an allowable pair and assume that there is $D_!':G_{F,S_{\textnormal{bad}}'}\to \mathbf{T}_!^{\overline{S}_{\textnormal{avoid}}'}(K,m)/I_!^{\overline{S}_{\textnormal{avoid}}'}$ satisfying local-global compatibility away from $\overline{S}_{\textnormal{avoid}}'$ for an ideal with $(I_!^{\overline{S}_{\textnormal{avoid}}'})^{M_!'}=0$. By Theorem \ref{ScholzeTHM}, there is a determinant $D_!:G_{F,S_{\textnormal{bad}}}\to \mathbf{T}_!^{\overline{S}_{\textnormal{avoid}}}(K,m)/J_!^{\overline{S}_{\textnormal{avoid}}}$ satisfying local-global compatibility away from $\overline{S}_{\textnormal{bad}}$. We can then set $M_!=M_!'+N_!$ and $I_!^{\overline{S}_{\textnormal{avoid}}}:=I_!^{\overline{S}_{\textnormal{avoid}}'}\mathbf{T}_!^{\overline{S}_{\textnormal{avoid}}}(K,m)+J_!^{\overline{S}_{\textnormal{avoid}}}.$
    \item We have finitely many maximal ideals $\mathfrak{m}\leq \mathbf{T}_K^{\overline{S}_{\textnormal{bad}}}$ such that 
    \begin{equation*}
        \mathbf{T}_!^{\overline{S}_{\textnormal{avoid}}}(K,m)\cong \prod_{\mathfrak{m}}\mathbf{T}_!^{\overline{S}_{\textnormal{avoid}}}(K,m)_{\mathfrak{m}}.
    \end{equation*}
    It suffices to prove for every maximal ideal $\mathfrak{m}$, the existence of an ideal $I_!\leq \mathbf{T}_!^{\overline{S}_{\textnormal{avoid}}}(K,m)_{\mathfrak{m}}$ and a continuous determinant $D_!$ with values in $\mathbf{T}_!^{\overline{S}_{\textnormal{avoid}}}(K,m)_{\mathfrak{m}}/I_!$ satisfying local-global compatibility away from $\overline{S}_{\textnormal{avoid}}$. In particular, we fix a maximal ideal $\mathfrak{m}\leq \mathbf{T}_K^{\overline{S}_{\textnormal{bad}}}$ and prove the claim after localisation.
    \item For every $\Bar{v}\in \overline{S}_{\textnormal{ram}}$, we fix a preferred choice of place $v| \Bar{v}$ in $S_{\textnormal{ram}}$ and denote the other place above $\overline{v}$ by $v^c$. In particular, we get a corresponding decomposition $S_{\textnormal{ram}}=\widetilde{S}_{\textnormal{ram}}\coprod \widetilde{S}_{\textnormal{ram}}^c$. Since we can freely change our preferred choice of lifts, it suffices to prove local-global compatibility for $v\in \widetilde{S}_{\textnormal{ram}}$.
    \item Let $\overline{\chi}:G_F\to \overline{\mathbf{F}}_{\ell}^{\times}$ be an almost everywhere unramified continuous character that is unramified at $S_{\textnormal{ram}}$. By enlarging $\overline{S}_{\textnormal{bad}}$, we can assume that it is of the form $\overline{\chi}:G_{F,S_{\textnormal{avoid}}}\to k^{\times}$ for the residue field $k$ of the ring of integers $\mathcal{O}$ of some finite field extension of $\mathbf{Q}_{\ell}$. Let $\chi:G_{F,S_{\textnormal{avoid}}}\to \mathcal{O}^{\times}$ be its canonical lift. Let $L_{\overline{S}_{\textnormal{avoid}}}\leq K_{\overline{S}_{\textnormal{avoid}}}$ be a normal subgroup such that $K_{\overline{S}_{\textnormal{avoid}}}/L_{\overline{S}_{\textnormal{avoid}}}$ is an abelian group of prime-to-$\ell$ order and $\chi\circ \det$ is trivial on $L:=K^{\overline{S}_{\textnormal{avoid}}}L_{\overline{S}_{\textnormal{avoid}}}$. Then the map $f_{\chi}^{-1}:\mathbf{T}_K^{\overline{S}_{\textnormal{avoid}}}\otimes_{\mathbf{Z}_{\ell}}\mathcal{O}\to\mathbf{T}_K^{\overline{S}_{\textnormal{avoid}}}\otimes_{\mathbf{Z}_{\ell}}\mathcal{O}$ from Lemma \ref{CharTwistLemma} can be shown to descend to a surjection
    \begin{equation*}
        \mathbf{T}_{!,\mathcal{O}}^{\overline{S}_{\textnormal{avoid}}}(L,m)_{\mathfrak{m}(\chi)}\twoheadrightarrow \mathbf{T}_{!,\mathcal{O}}^{\overline{S}_{\textnormal{avoid}}}(K,m)_{\mathfrak{m}}
    \end{equation*}
    where $\mathbf{T}_{!,\mathcal{O}}^{\overline{S}_{\textnormal{avoid}}}(K,m):=(\mathbf{T}_K^{\overline{S}_{\textnormal{avoid}}}\otimes_{\mathbf{Z}_{\ell}}\mathcal{O})/\textnormal{Ann}_{\mathbf{T}_K^{\overline{S}_{\textnormal{avoid}}}\otimes_{\mathbf{Z}_{\ell}}\mathcal{O}}(i_{c,K,m}\otimes_{\mathbf{Z}_{\ell}}\mathcal{O})$.
    Using that $\Phi_{v,n}$ is compatible with character twists, we see that it suffices to prove the claim from (ii) for $\mathfrak{m}(\chi)$ in place of $\mathfrak{m}$.

    Using the theorem of Grunwald--Wang (cf. \cite{AT61}, Chapter X, Theorem 5), we construct and fix a character $\overline{\chi}$ such that, for $\Bar{v}\in \overline{S}_{\textnormal{ram}}$, $\overline{\chi}_{v^c}:=\overline{\chi}|_{G_{F_{v^c}}}$ is trivial and $\overline{\chi}_v:=\overline{\chi}|_{G_{F_{v}}}$ is one of the characters provided by Lemma~\ref{KeyLemma} for $\phi_1=\phi_{\mathfrak{m}_v}$ and $\phi_2=\phi_{\mathfrak{m}_{v^c}}^{\vee}(1-2n)$. We then prove the theorem for $\mathfrak{n}:=\mathfrak{m}(\chi)$. In particular, we note that $\mathfrak{n}_{\Bar{v}}=(\mathfrak{n}_v,\mathfrak{n}_{v^c})=(\mathfrak{m}_{v}(\chi_v),\mathfrak{m}_{v^c})\leq \mathfrak{Z}_{G(F^+_{\Bar{v}})}\cong^{\iota_v}\mathfrak{Z}_{\textnormal{GL}_n(F_v)}\otimes_{\mathbf{Z}_{\ell}}\mathfrak{Z}_{\textnormal{GL}_n(F_v)}$ is $(n,n)$-generic where we view $\textnormal{GL}_n(F_v)\times \textnormal{GL}_n(F_v)$ as a standard Levi subgroup of $\textnormal{GL}_{2n}(F_v)$ via the diagonal embedding.
\end{enumerate}

We now define an ideal $I_!\leq(J_!^{\overline{S}_{\textnormal{avoid}}})_{\mathfrak{n}}$ (with nilpotence degree bounded in terms of $n$ and $[F^+:\mathbf{Q}]$) and prove local-global compatibility at $v\in \widetilde{S}_{\textnormal{ram}}$ for the induced determinant $D_!:G_{F,\overline{S}_{\textnormal{bad}}}\to \mathbf{T}_!^{\overline{S}_{\textnormal{avoid}}}(K,m)_{\mathfrak{n}}/I_!$. 

Let $\widetilde{K}\leq \widetilde{G}(\mathbf{A}_{F^+}^{\infty})$ be an $(\overline{S}_{\textnormal{bad}},\overline{S}_{\textnormal{ram}})$-parahoric subgroup such that $\widetilde{K}\cap G(\mathbf{A}_{F^+}^{\infty})=K$. Denote by $\widetilde{\mathfrak{n}}\subset \widetilde{\mathbf{T}}^{\overline{S}_{\textnormal{bad}}}$ the pullback of $\mathfrak{n}$ along the unnormalised Satake transform and set $I_!:=\mathcal{S}_{\overline{S}_{\textnormal{ram}}}(\widetilde{J}_{\partial}^{(\overline{S}_{\textnormal{bad}},\overline{S}_{\textnormal{ram}})})_{\mathfrak{n}}+(J_!^{\overline{S}_{\textnormal{avoid}}})_{\mathfrak{n}}$ (cf. Theorem \ref{EisensteinCongCusp}). We then define the faithful Hecke algebras
    \begin{equation*}
\widetilde{A}_{\partial}:=\im\left(\widetilde{\mathbf{T}}_{\widetilde{K}}^{\overline{S}_\textnormal{bad}}\otimes_{\mathbf{Z}_{\ell}}\mathfrak{Z}_{G(F^+_{\Bar{v}})}\to (\widetilde{\mathbf{T}}_{\partial}^{(\overline{S}_{\textnormal{bad}},\overline{S}_{\textnormal{ram}})}(\widetilde{K},m)/\widetilde{J}_{\partial}^{(\overline{S}_{\textnormal{bad}},\overline{S}_{\textnormal{ram}})})_{\widetilde{\mathfrak{n}},\mathfrak{n}_{\Bar{v}}}\right),
    \end{equation*}
    \begin{equation*}
        A_!:=\im\left(\mathbf{T}_{K}^{\overline{S}_\textnormal{bad}}\otimes_{\mathbf{Z}_{\ell}}\mathfrak{Z}_{G(F^+_{\Bar{v}})}\to (\mathbf{T}_!^{\overline{S}_{\textnormal{avoid}}}(K,m)_{\mathfrak{n}}/I_!)_{\mathfrak{n}_{\Bar{v}}}\right)=^{\textnormal{Thm}~\ref{modellLGC}}
    \end{equation*}
    \begin{equation*}
        \im\left(\mathbf{T}_{K}^{\overline{S}_\textnormal{bad}}\otimes_{\mathbf{Z}_{\ell}}\mathfrak{Z}_{G(F^+_{\Bar{v}})}\to \mathbf{T}_!^{\overline{S}_{\textnormal{avoid}}}(K,L,m)_{\mathfrak{n}}/I_!\right).
    \end{equation*}
    Note that $\widetilde{A}_{\partial}$ and $A_!$ are Artinian local $\mathbf{Z}_{\ell}$-algebras. The preimage of their maximal ideals under the natural surjection from $\widetilde{\mathbf{T}}_{\widetilde{K}}^{\overline{S}_\textnormal{bad}}\otimes_{\mathbf{Z}_{\ell}}\mathfrak{Z}_{G(F^+_{\Bar{v}})}$  and $\mathbf{T}_{K}^{\overline{S}_\textnormal{bad}}\otimes_{\mathbf{Z}_{\ell}}\mathfrak{Z}_{G(F^+_{\Bar{v}})}$ is given by $(\widetilde{\mathfrak{n}},\mathfrak{n}_{\Bar{v}})$ and $(\mathfrak{n},\mathfrak{n}_{\Bar{v}})$, respectively. Moreover, $\overline{\mathcal{S}}_{\partial}$ restricts to a surjection $\widetilde{A}_{\partial}\to A_!$ of \textit{local} $\mathbf{Z}_{\ell}$-algebras descending
    $\mathcal{S}_{\Bar{v}}:=\mathcal{S}\otimes \textnormal{id}:\widetilde{\mathbf{T}}_{\widetilde{K}}^{\overline{S}_\textnormal{bad}}\otimes_{\mathbf{Z}_{\ell}}\mathfrak{Z}_{G(F^+_{\Bar{v}})}\to\mathbf{T}_{K}^{\overline{S}_\textnormal{bad}}\otimes_{\mathbf{Z}_{\ell}}\mathfrak{Z}_{G(F^+_{\Bar{v}})}$.

    We further set 
    \begin{equation*}
\widetilde{A}:=\im\left(\widetilde{\mathbf{T}}_{\widetilde{K}}^{\overline{S}_\textnormal{bad}}\otimes_{\mathbf{Z}_{\ell}}\mathfrak{Z}_{G(F^+_{\Bar{v}})}\to \left((\widetilde{A}_{\widetilde{K},m,\partial}^{(\overline{S}_{\textnormal{bad}},\overline{S}_{\textnormal{ram}})})_{\widetilde{\mathfrak{n}}}\right)_{\mathfrak{n}_{\Bar{v}}}\right),
    \end{equation*}
    where the notation is taken from Theorem~\ref{EisensteinCongCusp}. Note that $\widetilde{A}$ is a finite flat and local $\mathbf{Z}_{\ell}$-algebra. The preimage in $\widetilde{\mathbf{T}}_{\widetilde{K}}^{\overline{S}_\textnormal{bad}}\otimes_{\mathbf{Z}_{\ell}}\mathfrak{Z}_{G(F^+_{\Bar{v}})}$ of its maximal ideal under the natural surjection is $(\widetilde{\mathfrak{n}},\mathfrak{n}_{\Bar{v}})\leq\widetilde{\mathbf{T}}_{\widetilde{K}}^{\overline{S}_\textnormal{bad}}\otimes_{\mathbf{Z}_{\ell}}\mathfrak{Z}_{G(F^+_{\Bar{v}})}$. Moreover, by Theorem~\ref{EisensteinCongCusp} we see that $\mathcal{S}_{\Bar{v}}$ descends further to a surjection $\overline{\mathcal{S}}:\widetilde{A}\twoheadrightarrow \widetilde{A}_{\partial}\twoheadrightarrow A_!$ of local $\mathbf{Z}_{\ell}$-algebras.
    \begin{Lemma}
        There exists a continuous $2n$-dimensional $\widetilde{A}$-valued group determinant
        \begin{equation*}
            \widetilde{D}:G_{F,S_{\textnormal{bad}}}\to \widetilde{A}
        \end{equation*}
        such that the following hold.
        \begin{enumerate}
            \item We have $\widetilde{D}\otimes_{ \widetilde{A}}A_!\cong D_{!}D_{!}^{\vee,c}(1-2n)$.
            \item The induced pseudorepresentation $\widetilde{D}_v:=\widetilde{D}|_{G_{F_v}}$ of $G_{F_v}$ admits a decomposition $\widetilde{D}_{v}=D_1D_2$ into $n$-dimensional pseudorepresentations.
            \item If $K_v=\textnormal{GL}_n(\mathcal{O}_{F_v})$, $D_1$ is unramified.
            \item The diagram
            \begin{equation*}
                \begin{tikzcd}
	{\mathfrak{R}_{F_v,n}^{\textnormal{ps}}} && {\widetilde{A}} \\
	& {\mathfrak{Z}_{\textnormal{GL}_n(F_{v})}}
	\arrow["{D_1}", from=1-1, to=1-3]
	\arrow["{\Phi_{F_v,n}}"', from=1-1, to=2-2]
	\arrow["{\textnormal{nat}_{\widetilde{A}}}"', from=2-2, to=1-3]
\end{tikzcd}
            \end{equation*}
            is commutative.
            \item We have $D_1\otimes_{\widetilde{A}}A_!=D_{\mathfrak{n}}|_{G_{F_v}}$.
        \end{enumerate}
    \end{Lemma}
    \begin{proof} By Lemma \ref{Casselmanslift}, we have a finite $\widetilde{\mathbf{T}}^{\overline{S}_{\textnormal{bad}}}\otimes_{\mathbf{Z}_{\ell}}\mathfrak{Z}_{G(F^+_{\Bar{v}})}$-equivariant product decomposition
        \begin{equation*}
            \widetilde{A}[\frac{1}{\ell}]\cong\prod_{\widetilde{\pi}} \widetilde{A}_{\widetilde{\pi}}
        \end{equation*}
        where the product runs over a finite set of cohomological cuspidal automorphic representations of $\widetilde{G}$ with $\overline{D}_{r_{\iota}(\widetilde{\pi})}\cong \overline{D}_{\mathfrak{n}}\oplus \overline{D}_{\mathfrak{n}}^{\vee,c}(1-2n)$ and
        \begin{equation*}
\widetilde{A}_{\widetilde{\pi}}:=\im\left((\widetilde{\mathbf{T}}^{\overline{S}_{\textnormal{bad}}}\otimes_{\mathbf{Z}_{\ell}}\mathfrak{Z}_{G(F^+_{\Bar{v}})})[\frac{1}{\ell}]\to\textnormal{End}_{\overline{\mathbf{Q}}_{\ell}}((\widetilde{\pi}_f^{\Bar{v},\widetilde{L}^{\Bar{v}}})\otimes_{\overline{\mathbf{Q}}_{\ell}}J_{P_{\Bar{v}}}(\widetilde{\pi}_{\Bar{v}})^{K_{\Bar{v}}})\right)_{(\widetilde{\mathfrak{n}},\mathfrak{n}_{\Bar{v}})}.
        \end{equation*}
Recall that $\mathfrak{n}$ was chosen so that with the choice of $\phi_1\otimes \chi=\phi_{\mathfrak{n}_v}$ and $\phi_2=\phi_{\mathfrak{n}_{v^c}}^{\vee}(1-2n)$ we are in the situation of Lemma~\ref{KeyLemma}. In particular, for any $\widetilde{\pi}$ appearing in the indexing set, the conclusion of Lemma~\ref{KeyLemma} applies with $\rho=r_{\iota}(\widetilde{\pi})|_{G_{F_v}}$. Consequently, using that by Theorem~\ref{SelfdualLGC}, we have $\pi(\rho)=\widetilde{\pi}_{\Bar{v}}$, we obtain that $\widetilde{A}_{\widetilde{\pi}}$ is isomorphic to a finite field extension $E_{\widetilde{\pi}}/\mathbf{Q}_{\ell}$ as a $(\widetilde{\mathbf{T}}^{\overline{S}_{\textnormal{bad}}}\otimes_{\mathbf{Z}_{\ell}}\mathfrak{Z}_{G(F^+_{\Bar{v}})})[\frac{1}{\ell}]$-algebra. We can therefore define a $2n$-dimensional $\widetilde{A}[\frac{1}{\ell}]$-valued Galois representation $\widetilde{\rho}$ with associated pseudorepresentation
\begin{equation*}
    \widetilde{D}[\frac{1}{\ell}]:=\prod_{\widetilde{\pi}}D_{r_{\iota}(\widetilde{\pi})}:G_{F,S_{\textnormal{bad}}}\to \prod_{\widetilde{\pi}}\widetilde{A}_{\widetilde{\pi}}\cong \widetilde{A}[\frac{1}{\ell}].
\end{equation*}
Moreover, $\widetilde{D}[\frac{1}{\ell}]$ comes from a unique $\widetilde{A}$-valued determinant via extension of scalars by \cite{Che14}, Corollary 1.14.

It satisfies (i) by Proposition \ref{Sataketransform}. Moreover, by Lemma~\ref{KeyLemma}, we have a decomposition 
\begin{equation}
    \widetilde{D}_v=D_{1}[1/\ell]D_{2}[1/\ell]
\end{equation}
such that the diagram
\begin{equation*}
                \begin{tikzcd}
	{\mathfrak{R}_{F_v,n}^{\textnormal{ps}}} && {\widetilde{A}[1/\ell]} \\
	& {\mathfrak{Z}_{\textnormal{GL}_n(F_{v})}}
	\arrow["{D_{1}[1/\ell]}", from=1-1, to=1-3]
	\arrow["{\Phi_{F_v,n}}"', from=1-1, to=2-2]
	\arrow["{\textnormal{nat}_{\widetilde{A}}}"', from=2-2, to=1-3]
\end{tikzcd}
            \end{equation*}
is commutative. In particular, we see that $D_{1}[1/\ell]$ and, similarly, $D_{2}[1/\ell]$, descend to $D_1:G_{F_v}\to \widetilde{A}$ and $D_2:G_{F_v}\to \widetilde{A}$, respectively. This proves (ii), (iii) and $(iv)$.

Note that, by construction, $\overline{D}_1=\overline{D}_{\mathfrak{n}}|_{G_{F_v}}$ and $\overline{D}_2=\overline{D}_{\mathfrak{n}}^{\vee,c}(1-2n)|_{G_{F_{v}}}$. In particular, (v) follows from Hensel's Lemma for determinants (cf. \cite{ACC23}, Lemma 3.2.4).
\end{proof}
We now easily finish the proof. The lemma yields a diagram
\begin{equation*}
    \begin{tikzcd}
	{\mathfrak{R}_{F_v,n}^{\textnormal{ps}}} && {A_!} \\
	& {\widetilde{A}} \\
	\\
	& {\mathfrak{Z}_{\textnormal{GL}_n(F_{v})}}
	\arrow["{D_{\mathfrak{n}}|_{G_{F_v}}}", from=1-1, to=1-3]
	\arrow["{D_1}", from=1-1, to=2-2]
	\arrow["{\Phi_{F_v,n}}"', from=1-1, to=4-2]
	\arrow["{\overline{\mathcal{S}}}"', from=2-2, to=1-3]
	\arrow["{\textnormal{nat}_A}"', from=4-2, to=1-3]
	\arrow["{\textnormal{nat}_{\widetilde{A}}}"{pos=0.8}, from=4-2, to=2-2]
\end{tikzcd}
\end{equation*}
with each of the inner triangles being commutative. Therefore, the outer triangle is also commutative.

\qed

\subsubsection{Proof of Theorem~\ref{MainTHM} for $?\in\{\emptyset, c\}$} We argue by induction on $n\geq 1$. For $n=1$, the case of $?=\emptyset$ and $?=c$ are both equivalent to the case of interior cohomology and so the theorem is already proved. We prove the theorem for $n>1$ assuming that it is known for $1\leq k<n$.

Using the excision triangle for the boundary and local-global compatibility for interior cohomology, we reduce the problem to local-global compatibility for the Borel--Serre boundary.

If $K_v = \textnormal{GL}_n(\mathcal{O}_{F_v})$ for some $v\in S_{\textnormal{ram}}$, then the cohomology of any Borel--Serre stratum can be described as the cohomology of the locally symmetric space of the corresponding Levi factor with hyperspecial level at $v$. In particular, the first claim of the Theorem follows easily.

To show the second assertion, it will be convenient to work with the faithful quotients of Bernstein centres as in \S\ref{sec_BSboundary_for_G}. To see that there is no harm in doing this, note that the following assertions are equivalent.
\begin{itemize}
    \item The existence of nilpotent ideals $I_{?,K_{\overline{S}_{\textnormal{ram}}}'}^{\overline{S}_{\textnormal{avoid}}}\leq \mathbf{T}_{?}^{\overline{S}_{\textnormal{avoid}}}(K^{\overline{S}_{\textnormal{ram}}}K_{\overline{S}_{\textnormal{ram}}}',m)$ of nilpotence degree depending only on $n$ and $[F^+:\mathbf{Q}]$ and continuous group determinants
\begin{equation*}
    D_?:G_{F,S_{\textnormal{bad}}}\to \mathbf{T}_{?}^{\overline{S}_{\textnormal{avoid}}}(K^{\overline{S}_{\textnormal{ram}}}K_{\overline{S}_{\textnormal{ram}}}',m)/I_{?,K_{\overline{S}_{\textnormal{ram}}}'}^{\overline{S}_{\textnormal{avoid}}}
\end{equation*}
satisfying local-global compatibility away from $\overline{S}_{\textnormal{avoid}}$.
\item The existence of a nilpotent ideal $\mathcal{I}_{?}^{\overline{S}_{\textnormal{avoid}}}\leq \mathfrak{Z}_{G,?}^{\overline{S}_{\textnormal{avoid}}}(K^{\overline{S}_{\textnormal{ram}}},m)$ depending only on $n$ and $[F^+:\mathbf{Q}]$ and a continuous group determinant
\begin{equation*}
    D_?:G_{F,S_{\textnormal{bad}}}\to \mathfrak{Z}_{G,?}^{\overline{S}_{\textnormal{avoid}}}(K^{\overline{S}_{\textnormal{ram}}},m)/\mathcal{I}_{?}^{\overline{S}_{\textnormal{avoid}}}
\end{equation*}
satisfying local-global compatibility away from $\overline{S}_{\textnormal{avoid}}$.
\end{itemize}
 Indeed, writing
\begin{equation*}
    \pi_{K_{\overline{S}_{\textnormal{ram}}}'}:\mathfrak{Z}_{G,?}^{\overline{S}_{\textnormal{avoid}}}(K^{\overline{S}_{\textnormal{ram}}},m)\to \mathbf{T}_{?}^{\overline{S}_{\textnormal{avoid}}}(K^{\overline{S}_{\textnormal{ram}}}K_{\overline{S}_{\textnormal{ram}}}',m)
\end{equation*} 
for the natural map, the correspondence for the ideals is given by sending $(I_{?,K_{\overline{S}_{\textnormal{ram}}}'}^{\overline{S}_{\textnormal{avoid}}})_{K_{\overline{S}_{\textnormal{ram}}}'}$ to 
\begin{equation*}
    \mathcal{I}_{?}^{\overline{S}_{\textnormal{avoid}}}:=\cap_{K_{\overline{S}_{\textnormal{ram}}}'}\pi_{K_{\overline{S}_{\textnormal{ram}}}'}^{-1}(I_{?,K_{\overline{S}_{\textnormal{ram}}}'}^{\overline{S}_{\textnormal{avoid}}})
\end{equation*}
 and sending $\mathcal{I}_{?}^{\overline{S}_{\textnormal{avoid}}}$ to 
 \begin{equation*}
     I_{?,K_{\overline{S}_{\textnormal{ram}}}'}^{\overline{S}_{\textnormal{avoid}}}:=\pi_{K_{\overline{S}_{\textnormal{ram}}}'}(\mathcal{I}_{?}^{\overline{S}_{\textnormal{avoid}}})\cdot\mathbf{T}_{?}^{\overline{S}_{\textnormal{avoid}}}(K^{\overline{S}_{\textnormal{ram}}}K_{\overline{S}_{\textnormal{ram}}}',m).
 \end{equation*} 
 The correspondence for determinants may be established using Chebotarev's density theorem, \cite{Che09} Example 2.32 and local-global compatibility away from $\overline{S}_{\textnormal{bad}}$. In particular, we prove the latter claim. 
 
 By repeated application of the excision long exact sequence for the stratification of the boundary, we see that it suffices to construct, for every standard $F^+$-rational parabolic subgroup $Q=MN\leq G$, a nilpotent ideal $\mathcal{I}_{\partial}^{Q}\leq \mathfrak{Z}_{G,c}^{\overline{S}_{\textnormal{avoid}},Q}(K^{\overline{S}_{\textnormal{ram}}},m)$ with nilpotence degree bounded only in terms of $n$ and $[F^+:\mathbf{Q}]$, and an $n$-dimensional continuous determinant $D_{Q}$ valued in
    \begin{equation*}
    \mathfrak{Z}_{G,c}^{\overline{S}_{\textnormal{avoid}},Q}(K^{\overline{S}_{\textnormal{ram}}},m)/\mathcal{I}_{\partial}^{Q}
    \end{equation*}
    satisfying local-global compatibility away from $\overline{S}_{\textnormal{avoid}}$.

   By Lemma~\ref{BSInductionLemma}, it suffices to construct a nilpotent ideal 
   \begin{equation*}
       \mathcal{I}_{G}^{M,\overline{S}_{\textnormal{ram}}}\leq \mathfrak{Z}_{G,c}^{\overline{S}_{\textnormal{avoid}}}(K^{\overline{S}_{\textnormal{bad}}}K_{M,\overline{S}_{\textnormal{avoid}}}',m)
   \end{equation*} with nilpotence degree bounded only in terms of $n$ and $[F^+:\mathbf{Q}]$, and an $n$-dimensional continuous determinant $D_{M}$ valued in
    \begin{equation*}
    \mathfrak{Z}_{G,c}^{\overline{S}_{\textnormal{avoid}}}(K_M^{\overline{S}_{\textnormal{bad}}}K_{M,\overline{S}_{\textnormal{avoid}}}',m)/\mathcal{I}_{G}^{M,\overline{S}_{\textnormal{ram}}}
    \end{equation*}
    satisfying local-global compatibility away from $\overline{S}_{\textnormal{avoid}}$.

    Writing $n=n_1+...+n_k$ for the partition corresponding to $M$, we have the corresponding decomposition $K^{\overline{S}_{\textnormal{bad}}}K_{M,\overline{S}_{\textnormal{avoid}}}'=K_1^{\overline{S}_{\textnormal{ram}}}\times ...\times K_k^{\overline{S}_{\textnormal{ram}}}$ into level subgroups for auxiliary general linear groups. By induction, we get for $i=1,...,k$, $n_i$-dimensional continuous determinants $D_i$ of $G_{F,S_{\textnormal{bad}}}$ valued in $\mathfrak{Z}^{\overline{S}_{\textnormal{avoid}}}_{\textnormal{GL}_{n_i},c}(K_i^{\overline{S}_{\textnormal{ram}}},m)$ modulo some nilpotent ideal satisfying local-global compatibility away from $\overline{S}_{\textnormal{avoid}}$. By the Künneth formula, we obtain an $n$-dimensional group determinant
    \begin{equation*}
        D_M:=D_1(-(n_2+...+n_k))\oplus...\oplus D_k
    \end{equation*}
    valued in $\mathfrak{Z}_{M,c}^{\overline{S}_{\textnormal{avoid}}}(K_M^{\overline{S}_{\textnormal{bad}}}K_{M,\overline{S}_{\textnormal{avoid}}}',m)$ modulo a nilpotent ideal $\mathcal{I}$ (with nilpotence degree bounded in terms of $n$ and $[F^+:\mathbf{Q}]$ and satisfying local-global compatibility away from $\overline{S}_{\textnormal{avoid}}$ (for $M$).

    We claim that it is, in fact, valued in the image of $\mathfrak{Z}_{G,c}^{\overline{S}_{\textnormal{avoid}}}(K_M^{\overline{S}_{\textnormal{bad}}}K_{M,\overline{S}_{\textnormal{avoid}}}',m)$ and satisfies local-global compatibility away from $\overline{S}_{\textnormal{avoid}}$ for $G$.

    To see the first claim, note that, by \cite{Che09} Example 1.14, it suffices to check that the continuous functions $\Lambda_i$ (with domain $G_{F,S_{\textnormal{bad}}}$) giving the $i$th coefficient of the characteristic polynomials of $D_M$ are all valued in the image of
    \begin{equation*}
        \mathfrak{Z}_{G,c}^{\overline{S}_{\textnormal{avoid}}}(K_M^{\overline{S}_{\textnormal{bad}}}K_{M,\overline{S}_{\textnormal{avoid}}}',m).
    \end{equation*} However, the image is certainly a compact hence closed subset of the profinite set $\mathfrak{Z}_{M,c}^{\overline{S}_{\textnormal{avoid}}}(K_M^{\overline{S}_{\textnormal{bad}}}K_{M,\overline{S}_{\textnormal{avoid}}}',m)/\mathcal{I}$, and so by Chebotarev's density theorem, it suffices to check this for a subset $\{\textnormal{Frob}_w\mid w\notin S_{\textnormal{bad}}\}$ of $G_{F,S_{\textnormal{bad}}}$. Consequently, the claim follows from Lemma~\ref{LocalEisensteinFunctoriality} applied to places away from $S_{\textnormal{bad}}$.

    The second claim now follows from Lemma~\ref{LocalEisensteinFunctoriality} applied to places in $\overline{S}_{\textnormal{ram}}$.
    \qed

\medskip
\printbibliography

@book{Sch15,
  title={On torsion in the cohomology of locally symmetric varieties},
  author={P. Scholze},
  year={2015},
  publisher={Ann. of Math. (2) 182, no. 3, 945-1066. MR 3418533},
  }

@book{Che14,
  title={The p-adic analytic space of pseudocharacters of a profinite group and pseudorepresentations over arbitrary rings},
  author={G. Chenevier},
  year={2014},
  publisher={Automorphic forms and Galois representations. Vol. 1, London Math. Soc. Lecture Note Ser., vol. 414, Cambridge Univ. Press, Cambridge, pp. 221- 285. MR 3444227},
  }

@book{CG18,
  title={Modularity lifting beyond the Taylor-Wiles method},
  author={F. Calegari and D. Geraghty},
  year={2018},
  publisher={Invent. Math. 211, no. 1, 297-433},
  }

@book{Cl90,
  title={Motifs et formes automorphes: applications du principe de fonctorialit\'e},
  author={L. Clozel},
  year={1990},
  publisher={ In Automorphic forms, Shimura varieties, and L-functions, Vol. I (Ann Arbor, MI, 1988), volume 10 of Perspect. Math., pages 77-159. Academic Press, Boston, MA},
  }

@book{Cas95,
  title={Introduction to the theory of admissible representations of $p$-adic reductive groups},
  author={W. Casselman},
  year={1995},
  publisher={preprint},
  }

@book{BZ77,
  title={Induced representations of reductive $\mathfrak{p}$-adic groups. I},
  author={I.N. Bernstein and A.V. Zelevinsky},
  year={1977},
  publisher={Ann.Sci.\'Ecole Norm. Sup. (4) 10(4), 441-472},
  }

@book{Shi14,
  title={On the cohomological base change for unitary similitude groups},
  author={S. W. Shin},
  year={2014},
  publisher={Compos. Math. 150, no. 2, 220-225, Appendix to Galois representations associated to holomorphic limits of discrete series, by Wushi Goldring},
  }

@book{NT16,
  title={Torsion Galois representations over CM fields and Hecke algebras in the derived category},
  author={J. Newton and J. A. Thorne},
  year={2016},
  publisher={Forum Math. Sigma 4, Paper No. e21, 88},
  }

@book{Zel80,
  title={Induced representations of reductive p-adic groups. II. On irreducible representations of GL(n).},
  author={A. V. Zelevinsky},
  year={1980},
  publisher={Ann. Sci.
 \'Ecole Norm. Sup. (4) 13.2, pp. 165-210},
  }

@book{HT01,
  title={The geometry and cohomology of some simple Shimura varieties},
  author={M. Harris and R. Taylor},
  year={2001},
  publisher={Vol. 151. Annals of Mathematics Studies. With an appendix by Vladimir G. Berkovich. Princeton, NJ: Princeton University Press, pp. viii+276},
  }

@book{Car12,
  title={Local-global compatibility and the action of monodromy on nearby cycles},
  author={A. Caraiani},
  year={2012},
  publisher={Duke Math. J. 161(12): 2311-2413},
  }

@book{BS73,
  title={Corners and arithmetic groups},
  author={A. Borel and J.-P. Serre},
  year={1973},
  publisher={Comment. Math. Helv. 48, 436-491, Avec un appendice: Arrondissement des vari\'et\'es \'a coins, par A. Douady et L. H\'erault},
  }

@book{Sch98,
  title={Equivariant homology for totally disconnected groups},
  author={P. Schneider},
  year={1998},
  publisher={J. Algebra 203, no. 1, 50-68},
  }

@book{KS94,
  title={Sheaves on manifolds},
  author={M. Kashiwara and P. Schapira},
  year={1994},
  publisher={Grundlehren der Mathematischen Wissenschaften [Fundamental Principles of Mathematical Sciences], vol. 292, Springer-Verlag, Berlin, With a chapter in French by Christian Houzel, Corrected reprint of the 1990 original},
  }

@book{FS98,
  title={A decomposition of spaces of automorphic forms, and
the Eisenstein cohomology of arithmetic groups},
  author={J. Franke and J. Schwermer},
  year={1998},
  publisher={Math. Ann., 311(4):765-790},
  }

@book{Eme10b,
  title={Ordinary parts of admissible representations of p-adic reductive groups II},
  author={M. Emerton},
  year={2010},
  publisher={Ast\'erisque, no. 331, 403-459},
  }

@book{BK98,
  title={Smooth representations of reductive p-adic groups: structure theory via types},
  author={C. J. Bushnell and P. C. Kutzko},
  year={1998},
  publisher={Proc. London Math. Soc.
(3) 77, no. 3, 582-634.},
  }

@misc{CN23,
      title={On the modularity of elliptic curves over imaginary quadratic fields}, 
      author={Ana Caraiani and James Newton},
      year={2025},
      eprint={2301.10509},
      archivePrefix={arXiv},
      primaryClass={math.NT},
      url={https://arxiv.org/abs/2301.10509}, 
}

@book{Var14,
  title={Local-global compatibility for regular algebraic cuspidal automorphic representation when $\ell\neq p$},
  author={I. Varma},
  year={2024},
  publisher={Forum of Mathematics, Sigma 12, Number 21, Pg. 1-32},
  }

@book{Hel16,
  title={The Bernstein center of the category of smooth $W(k)[GL_n(F)]$-modules},
  author={D. Helm},
  year={2016},
  publisher={Forum Math. Sigma 4, e11},
  }

@book{BD84,
  title={Le "centre" de Bernstein},
  author={J. Bernstein, and P. Deligne},
  year={1984},
  publisher={In "Representations des groupes r\'eductifs sur un corps local,
Travaux en cours" (P. Deligne ed.), Hermann, Paris, 1-32},
  }

@book{Che09,
  title={Une application des vari\'et\'es de Hecke des
groupes unitaires},
  author={G. Chenevier},
  year={2009},
  publisher={preprint},
  }

@book{Hel20,
  title={Curtis Homomorphisms and the integral Bernstein center for $\textnormal{GL}_n$},
  author={D. Helm},
  year={2020},
  publisher={Algebra and Number Theory 14, no. 10, 2607-2645},
  }

@book{Fr98,
  title={Harmonic analysis in weighted L2-spaces},
  author={J. Franke},
  year={1998},
  publisher={Ann. Sci. Ecol\'e
Norm. Sup. (4) 31, no. 2, 181-279.},
  }

@book{HLTT16,
  title={On the rigid
cohomology of certain Shimura varieties},
  author={M. Harris and K.-W. Lan and R. Taylor and J. Thorne},
  year={2016},
  publisher={Res. Math. Sci. 3, 3:37},
  }

@article{ACC23,
 author = {Allen, Patrick B. and Calegari, Frank and Caraiani, Ana and Gee, Toby and Helm, David and Le Hung, Bao and Newton, James and Scholze, Peter and Taylor, Richard and Thorne, Jack A.},
 title = {Potential automorphy over {CM} fields},
 fjournal = {Annals of Mathematics. Second Series},
 journal = {Ann. Math. (2)},
 issn = {0003-486X},
 volume = {197},
 number = {3},
 pages = {897--1113},
 year = {2023},
 language = {English},
 doi = {10.4007/annals.2023.197.3.2},
 zbMATH = {7668527},
 Zbl = {1521.11034}
}

@book{Lan79,
  title={Automorphic representations, Shimura varieties, and motives},
  author={R. P. Langlands},
  year={1979},
  publisher={M\"archen, Automorphic forms, representations and L-functions (Proc. Sympos. Pure
Math., Oregon State Univ., Corvallis, Ore., 1977), Part 2, Proc. Sympos. Pure Math.,
XXXIII, Amer. Math. Soc., Providence, R.I., 1979, pp. 205-246. MR 546619},
  }

@article{MT22,title={Automorphy lifting with adequate image},volume={11},journal={Forum of Mathematics, Sigma},publisher={Cambridge University Press},author={Miagkov, Konstantin and Thorne, Jack A.},year={2023},pages={e8}}

@misc{stacks-project,
  author       = {The {Stacks project authors}},
  title        = {The Stacks project},
  howpublished = {\url{https://stacks.math.columbia.edu}},
  year         = {2024},
}

@article{Hev23,
author = {B. Hevesi},
year = {2023},
title = {Ordinary parts and local-global compatibility at $\ell=p$},
volume = {},
journal = {arXiv:2311.13514},
doi = {}
}

@article{Pin90,
author = {R. Pink},
year = {1990},
title = {Arithmetical compactification of mixed Shimura varieties},
volume = {},
journal = {Bonner Mathematische
Schriften [Bonn Mathematical Publications], 209. Universitat Bonn, Mathematisches Institut, Bonn},
doi = {}
}

@article{Vig96,
author = {M.-F. Vign\'eras},
year = {1996},
title = {Repr\'esentations l-modulaires d’un groupe r\'eductif $p$-adique avec $\ell\neq p$},
volume = {volume 137},
journal = {Progress in Mathematics. Birkh\"auser Boston, Inc., Boston, MA},
doi = {}
}

@article{Vig98,
author = {Vigneras, Marie-France},
year = {1998},
month = {12},
pages = {},
title = {Induced R-representations of p-adic reductive groups},
volume = {4},
journal = {Selecta Mathematica (N. S.), v.4, 549-623 (1998)},
doi = {10.1007/s000290050040}
}

@misc{DHKM24,
      title={Local Langlands in families: The banal case}, 
      author={Jean-François Dat and David Helm and Robert Kurinczuk and Gilbert Moss},
      year={2024},
      eprint={2406.09283},
      archivePrefix={arXiv},
      primaryClass={math.RT},
      url={https://arxiv.org/abs/2406.09283}, 
}

@article{Vig01,
author = {M.-F. Vigneras},
year = {2001},
title = {Correspondance de Langlands semi-simple pour GL(n, F) modulo $\ell\neq p$},
volume = {177-223},
journal = {Invent. Math. 144},
doi = {}
}

@book{Del73,
  title={Les constantes des \'equations fonctionnelles des fonctions L},
  author={P. Deligne},
  year={1973},
  publisher={In Pierre
Deligne and Willem Kuijk, editors, Modular Functions of One Variable II, pages
501–597, Berlin, Heidelberg, 1973. Springer Berlin Heidelberg},
  }

@article{DHKM20,
  title={Moduli of Langlands parameters},
  author={Dat, Jean-Fran{\c{c}}ois and Helm, David and Kurinczuk, Robert and Moss, Gilbert},
  journal={Journal of the European Mathematical Society},
  volume={27},
  number={5},
  pages={1827--1927},
  year={2025},
  publisher={European Mathematical Society}
}

@article{WE18,
author = {C. Wang-Erickson},
year = {2018},
title = {Algebraic families of Galois representations and potentially semistable pseudodeformation rings},
volume = {1615–1681},
journal = {Math. Ann., 371(3-4)},
doi = {}
}

@article{HM18,
author = {D. Helm and G. Moss},
year = {2018},
title = {Converse theorems and the local Langlands correspondence in families},
volume = {999–1022},
journal = {Invent. Math., 214(2)},
doi = {}
}

@article{DHKM24a,
author = {J.-F. Dat and D. Helm and R. Kurinczuk and G. Moss},
year = {2024},
title = {Finiteness for Hecke algebras of p-adic
groups},
volume = {},
journal = {J. Amer. Math. Soc., 37(3):929–949},
doi = {}
}

@article{MT23,
author = {G. Moss and J. Trias},
year = {2023},
title = {Towards a theta correspondence in families for type II dual pairs},
volume = {},
journal = {arXiv:2312.12031},
doi = {}
}

@book{AT61,
  title={Class Field Theory},
  author={E. Artin and J. Tate},
  year={1961},
  publisher={Harvard, Dept. of Mathematics. Notes
from the Artin-Tate seminar on class field theory given at Princeton University 1951–52.
Reprinted 1968, 1990; second edition AMS Chelsea Publishing, 2009},
  }

@misc{Aca26,
      title={Adjoint {B}loch--{K}ato {S}elmer groups of regular algebraic automorphic {G}alois representations}, 
      author={A'Campo, Lambert and Hevesi, Bence and Thorne, Jack A. and Whitmore, Dmitri},
      year={},
      howpublished={In preparation},
}

@article {New23,
    AUTHOR = {Newton, James and Thorne, Jack A.},
     TITLE = {Adjoint {S}elmer groups of automorphic {G}alois representations of unitary type},
   JOURNAL = {J. Eur. Math. Soc. (JEMS)},
  FJOURNAL = {Journal of the European Mathematical Society (JEMS)},
    VOLUME = {25},
      YEAR = {2023},
    NUMBER = {5},
     PAGES = {1919--1967},
      ISSN = {1435-9855},
   MRCLASS = {11F80 (12G05)},
  MRNUMBER = {4592862},
MRREVIEWER = {Santosh Nadimpalli},
       DOI = {10.4171/jems/1228},
       URL = {https://doi-org.ezp.lib.cam.ac.uk/10.4171/jems/1228},
}

\end{document}